\documentclass[11pt,reqno]{amsart}
\usepackage{fullpage}

\usepackage{graphicx}
\usepackage[draft]{hyperref}
\usepackage{amsmath,amsopn,amssymb,amsfonts,stmaryrd}
\usepackage{verbatim}
\usepackage{amsthm}
\usepackage{mathtools}
\usepackage{color}
\usepackage[x11names]{xcolor}
\usepackage{enumitem}
\usepackage[framemethod=TikZ]{mdframed}
\usepackage{bbm}
\usepackage{mathrsfs}
\usepackage{booktabs}
\usepackage{caption}
\usepackage{bm}
\usepackage{tensor}
\usepackage{cancel}
\usepackage{framed}
\colorlet{shadecolor}{LavenderBlush3}
\usepackage{float}

\usepackage{mathrsfs}
\usepackage[scr=boondox]{mathalfa}

\newcommand{\R}{\mathbb R}%Reals
\newcommand{\IR}{{\mathbb R}}%Reals
\newcommand{\IC}{{\mathbb C}}%Complex
\newcommand{\IP}{{\mathbb P}}%Projective
\newcommand{\IZ}{{\mathbb Z}}%Integers
\newcommand{\Z}{\mathbb{Z}}
\newcommand{\IN}{{\mathbb N}}%Natural numbers
\newcommand{\IH}{{\mathbb H}}%quaternions
\newcommand{\C}{{\mathbb C}}%complex numbers
\renewcommand{\H}{{\mathbb H}}%upper half-plane
\newcommand{\Q}{{\mathbb Q}}%rational numbers

\newcommand{\CA}{\mathscr{A}}

\newcommand{\sgn}{\mbox{sgn}}

\newcommand{\U}{\mathrm{U}}

\newcommand{\SL}{\mathrm{SL}}

\theoremstyle{plain}
\newtheorem{thm}{Theorem}[section]
\newtheorem{cor}[thm]{Corollary}
\newtheorem{lem}[thm]{Lemma}
\newtheorem{lemma}[thm]{Lemma}
\newtheorem{prop}[thm]{Proposition}

\newtheorem*{remark}{Remark}
\newtheorem*{remarks}{Remarks}
\newtheorem{rem}[thm]{Remark}

\theoremstyle{definition}
\newtheorem{defn}[thm]{Definition}

\numberwithin{equation}{section}

\newcommand{\pmat}[1]{\left( \smallmatrix #1 \endsmallmatrix \right)}
\newcommand{\mat}[1]{\left( \begin{matrix} #1 \end{matrix} \right)}

\renewcommand{\sgn}{\textnormal{sgn}}

\def\lp{\left(}
\def\rp{\right)}
\def\lsp{(}
\def\rsp{)}
\def\lb{\left[}
\def\rb{\right]}

\def\a{\alpha}
\def\b{\beta}
\def\d{\delta}

\def\l{\lambda}

\def\z{\zeta}

\def\ve{\varepsilon}
\def\e{\varepsilon}

\def\nn{{\bm n}}
\def\s{\sigma}
\def\g{\gamma}
\def\t{\tau}

\def\DD{\Delta}

\def\del{ \partial}

\newcommand{\re}{{\rm Re}}
\newcommand{\im}{{\rm Im}}

\renewcommand{\sgn}{{\rm sgn}}

\def\wt{\widetilde}
\def\wh{\widehat}
\def\bar{\overline}
\newcommand{\andd}{\quad \mbox{ and } \quad}
\newcommand{\where}{\quad \mbox{ where }}

\setlist[itemize]{noitemsep, topsep=0pt}

\makeatletter
\newcommand{\vast}{\bBigg@{2}}
\newcommand{\Vast}{\bBigg@{5}}
\makeatother

\newcommand{\calU}{\mathcal{U}}

\makeatletter
\newcommand{\nolisttopbreak}{\par\nobreak\@afterheading}
\makeatother

\newcommand{\eu}{\mathrm{eu}}

\newcommand{\od}{\mathrm{od}}

\renewcommand{\pmod}[1]{\    \lp  \mathrm{mod} \  #1 \rp}

\newcommand{\mT}{\Theta}

\newcommand{\Ci}{\mathrm{Ci}}
\newcommand{\Si}{\mathrm{Si}}
\newcommand{\PV}{\operatorname{PV}}

\newcommand{\psum}{\sideset{}{^*}\sum}
\newcommand{\KK}{\mathcal{K}}

\allowdisplaybreaks

\title{Precision Asymptotics for Partitions Featuring False-indefinite Theta Functions}

\author{Kathrin Bringmann}
\address{University of Cologne, Department of Mathematics and Computer Science, Weyertal 86-90, 50931 Cologne, Germany}
\email{cnazarog@uni-koeln.de}

\author{William Craig}
\address{Department of Mathematics, United States Naval Academy, 572C Holloway Road
	Mail Stop 9E. Annapolis, MD 21402}
\email{wcraig@usna.edu}

\author{Caner Nazaroglu}
\address{University of Cologne, Department of Mathematics and Computer Science, Weyertal 86-90, 50931 Cologne, Germany}
\email{cnazarog@uni-koeln.de}

\makeatletter
\@namedef{subjclassname@2020}{%
\textup{2020} Mathematics Subject Classification}
\makeatother

\begin{document}

\begin{abstract}
Andrews--Dyson--Hickerson, Cohen build a striking relation between $q$-hypergeometric series, real quadratic fields, and Maass forms. Thanks to the works of Lewis--Zagier and Zwegers we have a complete understanding on the part of these relations pertaining to Maass forms and false-indefinite theta functions. In particular, we can systematically distinguish and study the class of false-indefinite theta functions related to Maass forms. A crucial component here is the framework of mock Maass theta functions built by Zwegers in analogy with his earlier work on indefinite theta functions and their application to  Ramanujan's mock theta functions. Given this understanding, a natural question is to what extent one can utilize modular properties to investigate the asymptotic behavior of the associated Fourier coefficients, especially in view of their relevance to combinatorial objects. In this paper, we develop the relevant methods to study such a question and show that quite detailed results can be obtained on the asymptotic development, which also enable Hardy--Ramanujan--Rademacher type exact formulas under the right conditions. We develop these techniques by concentrating on a concrete example involving partitions with parts separated by parity and derive an asymptotic expansion that includes all the exponentially growing terms.
\end{abstract}

\maketitle

\section{Introduction and Statement of Results}
In the last few decades, there has been an extensive body of research on objects which are close to modular forms. A major strain here has been Ramanujan's mock theta functions, which include $q$-series such as,
with {\it $q$-Pochhammer symbol} $(a;q)_n:=\prod_{j=0}^{n-1}(1 - a q^j)$ for $a\in\C$, $n\in\mathbb N_0\cup\{\infty\}$,
\begin{align*}
    f(q) := \sum_{n \geq 0} \dfrac{q^{n^2}}{\lp -q;q \rp_n^2}.
\end{align*}
Although $f(q)$ and Ramanujan's other mock theta functions are not directly modular, they do have certain non-holomorphic ``completions" that make them modular.
This was first discovered through Zwegers' work on indefinite theta functions \cite{ZwThesis}
and was unified conveniently into the framework of harmonic Maass forms as formulated by Bruinier--Funke \cite{BF}.
Since then, the theory of mock modular forms has been developed quite broadly and had many applications (see \cite{BFOR} for a review).

The mock theta functions are not the only examples of ``modular adjacent'' objects discovered by Ramanujan.
A prominent example is the $\s$-function appearing in his lost notebook \cite{Ram}
\begin{equation}\label{eq:sigma_qhypergeometric_representation}
    \s(q) := \sum_{n \geq 0} \dfrac{q^{\frac{n(n+1)}{2}}}{\lp -q;q \rp_n}.
\end{equation}
Most strikingly, it was discovered by Andrews--Dyson--Hickerson \cite{ADH} and Cohen \cite{C} that $\s$ forms one of the legs in a three-way relationship involving $q$-hypergeometric series, real-quadratic fields, and Maass forms. This has since been generalized to many other cases; see \cite{BK,BLR,BN,CFLZ,Lo,Lo2,LoOs} and references therein for further examples. The connection to real-quadratic fields is formed by recasting $\s (q)$ as a false-indefinite theta function \cite[Theorem 1]{ADH}
\begin{equation} \label{eq:sigma_theta_representation}
\sigma(q) = \sum_{\substack{n \geq 0 \\ |j| \leq n}} (-1)^{n+j} q^{\frac{n(3n+1)}{2} - j^2} \lp 1 - q^{2n+1} \rp.
\end{equation}
A {\it false-indefinite theta function} is a theta function on a Lorentzian lattice that differs from the well-known indefinite theta functions of Zwegers \cite{ZwThesis} by the insertion of extra sign factors, similar to the relation between false theta functions and classical theta functions.

The connection to Maass forms, on the other hand, was developed by
Cohen~\cite{C}. He showed that the Fourier coefficients of $\s$ along with those of a
``complementary function''
\begin{align*}
\sigma^*(q) := 2 \sum_{n \geq 0} \dfrac{(-1)^n q^{n^2}}{\lp q;q^2 \rp_n}
\end{align*}
form the Fourier coefficients of a Maass form. As in the case of the mock theta functions, it is then natural to search for a modular framework. A major development in this direction came from the mock Maass theta functions of Zwegers \cite{Zw} and the Eichler-type integrals of Lewis--Zagier \cite{LZ} that connects false-indefinite theta functions to such mock Maass theta functions.\footnote{In modern language, $\s$ and $\s^*$ along with their limits to rationals combine into a holomorphic quantum modular form. This was in fact one of the first examples given by Zagier \cite{Zag} in developing the concept of quantum modularity.}
These are not modular in general, but become Maass forms for a class of false-indefinite theta functions, to which $\s$ and $\s^*$ belong.

With this understanding, it is natural to ask what insights we can gain out of this relationship, given the wide-ranging uses of modularity. For instance, one of the major applications of modularity is to the investigation of the asymptotics of the associated Fourier coefficients.
A fundamental example is the work of Hardy--Ramanujan \cite{HR1, HR2} that established the modern Circle Method. Their work examines \textit{partitions}, which are
non-increasing sequences $\lambda = \lp \lambda_1, \lambda_2, \dots, \lambda_\ell \rp$ of positive integers with \textit{size} $|\lambda| := \lambda_1 + \dots + \lambda_\ell$.
An immediate question is to determine and estimating the {\it partition function} $p(n)$ that counts the number of partitions of size $n$.
Hardy and Ramanujan's main result is an asymptotic expansion for $p(n)$, whose first term gives the famous asymptotic
\begin{equation*}
p(n) \sim \dfrac{1}{4\sqrt{3}n} e^{\pi\sqrt{\frac{2n}{3}}} \ \
\mbox{as } n\to\infty.
\end{equation*}
In fact, their work shows much more: by truncating their full asymptotic expansion after $\asymp \sqrt{n}$ terms, the formula of Hardy--Ramanujan can be rounded to the nearest integer in order to obtain the value of $p(n)$. Later, Rademacher \cite{Rad1, Rad2} refined this work to obtain an exact formula for $p(n)$ that is indeed convergent and not merely asymptotic. Rademacher's formula takes the shape
\begin{equation} \label{E: Partition formula}
p(n) = \frac{2\pi}{\lp 24 n -1 \rp^{\frac{3}{4}}} \
\sum_{k\geq1} \frac{A_k (n)}{k} I_{\frac{3}{2}} \! \lp \frac{\pi}{6k} \sqrt{24n-1} \rp,
\end{equation}
where $A_k (n)$ is an appropriate Kloosterman sum and $I_r$ denotes the $I$-Bessel function of order $r$.
Rademacher's work and the follow-up work of Rademacher and Zuckermann \cite{RZ,Zuck} demonstrates that the key point behind the proof is that the generating function for $p(n)$ is (essentially) a modular form of non-positive weight. As a consequence, there are Hardy--Ramanujan--Rademacher type exact formulas for a wide class of partitions. For example, if $r_o (n)$ denotes the number of partitions of $n$ into distinct odd parts, the corresponding generating function
\begin{equation*}
\sum_{n\geq0} r_o (n) q^n = \prod_{n\geq0} \left(1+q^{2n+1}\right)
\end{equation*}
is (basically) a weight zero modular form. This fact leads to an exact formula for $r_o (n)$ of the form
\begin{equation} \label{E: r_o formula}
r_o (n) =
\frac{ \pi}{\sqrt{24n-1}}  \sum_{k\geq1}  \frac{K_k (n)}{k}
I_1 \lp \frac{\pi}{12k} \sqrt{24n-1} \rp
\end{equation}
(see Theorem 5 of \cite{Hag}) with $K_k (n)$ an appropriate Kloosterman sum.
There are now many generalizations of such methods, whereby exact formulas are obtained also for generating functions that combine modular forms with modular-adjacent objects. Exact formulas are known, for example, in contexts involving weakly holomorphic modular forms multiplying mock modular forms \cite{BM}, mock modular forms of depth two \cite{BN3}, and false modular forms \cite{BN2}. In particular, \cite{BN2} gives an exact formula
in a similar combinatorial context
for the number of unimodal sequences. The goal in this paper is to extend these techniques to a new context involving false-indefinite theta functions.

More specifically, we develop the machinery to find high-precision asymptotic formulas by building around the example of partitions with parts separated by parity, which were studied by Andrews \cite{And1,And2}.
A partition $\lambda$ has {\it parts separated by parity} if all of its even parts are larger than all of its odd parts, or if the reverse is true. One may then consider functions counting such partitions, which could be refined with a number of further conditions. For example, in the notation of Andrews, $p_{\mathrm{od}}^{\mathrm{eu}}(n)$ counts the number of partitions of $n$ whose even parts are larger than odd parts, and where odd parts must be distinct and likewise even parts are unrestricted.
Many similar functions can be constructed in this way; the generating functions for these partitions were considered in \cite{And1,And2,BJ}.
We note that of the eight functions that can be constructed from this notation, seven of them have direct connections to modular forms, as well as to false and mock modular objects, where the technology for deriving exact formulas is already developed.
We focus in this paper on the last case, namely $p_{\mathrm{od}}^{\mathrm{eu}}(n)$, which is related to false-indefinite theta functions and Maass forms.

In particular, this relation manifests itself by writing the generating function for $p_{\mathrm{od}}^{\mathrm{eu}}(n)$,
\begin{align*}
F_\od^\eu(q) := \sum_{n \geq 0} p_\od^\eu(n) q^n = 1 + q + q^2 + 2q^3 + 3q^4 + 3q^5 + 4q^6 + 5q^7 + 8q^8 + 8q^9 + \cdots,
\end{align*}
in the form (see \cite{And2, BJ} as well as comments in \cite{BCN})
\begin{equation} \label{Eq: F_od^eu generating function}
F_\od^\eu(q) = \frac{1}{\lp q^2;q^2 \rp_\infty} \lp 1 - \frac{\s (-q)}{2}
+ \frac{\lp -q; -q \rp_\infty}{2} \rp,
\end{equation}
where $\s$ is the function defined in \eqref{eq:sigma_qhypergeometric_representation}. In our considerations, we focus on the part of the
generating function $F_{\mathrm{od}}^{\mathrm{eu}}$ that involves false-indefinite theta functions. More specifically, we decompose
\begin{equation}\label{eq:alpha_definition}
2p_\od^\eu(2n) =: 2p(n) + r_o(2n) + \alpha_0(n) \ \ \ \text{and} \ \ \ 2 p_\od^\eu(2n+1) =: r_o(2n+1) + \alpha_1(n)
\end{equation}
so that the generating functions of $\alpha_0$ and $\alpha_1$ are of the form (for $\t \in \mathbb{H}$ and $q:= e^{2\pi i \t}$)
\begin{align} \label{eq:alpha_generating_function}
\dfrac{u_0(\tau)}{\eta(\tau)} = \sum_{n \geq 0} \alpha_0(n) q^{n - \frac{1}{48}}, \ \ \ \dfrac{u_1(\tau)}{\eta(\tau)} = \sum_{n \geq 0} \alpha_1(n) q^{n + \frac{23}{48}},
\end{align}
where $\eta (\tau) := q^{\frac{1}{24}} \prod_{n \geq 1} (1-q^n)$ is the {\it Dedekind-eta function} and
\begin{align} \label{eq:u0_u1_definition}
u_0(\tau)
:= -\frac{q^\frac{1}{48}}{2}\left(\sigma\!\lp q^{\frac 12} \rp + \sigma\!\lp - q^{\frac 12} \rp\right),
\qquad
u_1(\tau)
:= \frac{q^\frac{1}{48}}{2}\left(\sigma\!\lp q^{\frac 12} \rp - \sigma\!\lp -q^{\frac 12} \rp\right).
\end{align}
Thanks to the exact formulas \eqref{E: Partition formula} and \eqref{E: r_o formula} for $p(n)$ and $r_o(n)$, respectively, to study the asymptotic behavior of $p_\od^\eu(n)$, it is then enough to restrict our attention to $\alpha_j(n)$ and investigate
\begin{align*}
\frac{u_0 (\t)}{\eta (\t)}
&=
q^{-\frac{1}{48}} \lp -1 + q^2 + q^3 + 4 q^4 + 4 q^5 + 9 q^6 + 11 q^7 + 19 q^8 + 23 q^9 + 37 q^{10} + 44 q^{11} + \ldots \rp,\\
\frac{u_1 (\t)}{\eta (\t)}
&=
q^{\frac{23}{48}} \lp 1 + 3 q + 5 q^2 + 9 q^3 + 14 q^4 + 22 q^5 + 31 q^6 + 48 q^7 +
 65 q^8 + 92 q^9 + 126 q^{10} + \ldots \rp .
\end{align*}
We relate $u_j$ to a vector-valued Maass form and use this connection to prove the following precise asymptotics that includes all of the exponentially growing contributions to $\alpha_j(n)$.

\begin{thm} \label{T: Main Theorem}
For $n\in\mathbb N$, $j \in \{0,1\}$, and with $\Delta_0 := - \frac{1}{48}$ and $\Delta_1 := \frac{23}{48}$, we have
\begin{multline*}
\a_j (n) = \frac{2}{\left(n+\DD_j\right)^{\frac{1}{4}}} \sum_{\ell=0}^2
\sum_{k=1}^{\left\lfloor \sqrt{n} \right\rfloor} \frac{1}{k}
\sum_{\substack{0 \leq h < k \\ \gcd (h,k) = 1}}
\psi_{h,k} (j,\ell)  e^{-\frac{2 \pi i}k \lp \frac{h'}{24} + \left(n+\DD_j\right) h \rp}
\\
\times \PV\int_0^{\frac{1}{24}} \Phi_{\ell, \frac{h'}{k}} (t)
\lp \frac{1}{24} - t \rp^{\frac{1}{4}}
I_{\frac{1}{2}} \!\lp \frac{4 \pi}{k} \sqrt{\left(n+\DD_j\right) \lp \frac{1}{24} - t \rp} \rp dt
+
O\!\lp n^{\frac{3}{4}}\rp,
\end{multline*}
where the multiplier $\psi_{h,k}(j,\ell)$ is defined in \eqref{eq:multiplier_system_definition}, $\Phi_{\ell,\frac{h'}{k}}$ is given in \eqref{eq:integral_kernel_definition} with $0 \leq h' < k$ chosen to satisfy $h h' \equiv -1 \pmod{ k}$, and $\PV \int$ denotes the Cauchy principal value integral.
\end{thm}

\begin{remarks}%\hspace{0cm}
\mbox{ }
\begin{enumerate}[leftmargin=*,label=\rm(\arabic*)]
\item There is a third component $u_2$ which is needed to build the full $\mathrm{SL}_2\lp \Z \rp$-orbit of the associated Maass form. However, the false-indefinite theta function $u_2$ does not appear to bear any direct connection to $p_\od^\eu$. It would be interesting to discover a combinatorial interpretation for the analogous object involving $u_2$. A version of this formula also holds for this object; indeed, all three $u_j$ are treated symmetrically in the proof of Theorem \ref{T: Main Theorem} (with $\DD_2 := \frac{11}{12}$).
\item The error term could in principle be made smaller. However, since $u_j$ (almost) transforms with weight 1 and $\eta$ has weight $\frac{1}{2}$, the overall weight is $\frac{1}{2}$ and exact formulas from the Circle Method are not expected. We note that for smaller weights, our methods indeed give exact formulas.
\item Although our results are stated for the specific false-indefinite theta function and the related Maass form under study, the same calculations apply to any similar setup.
\end{enumerate}
\end{remarks}

\noindent We can also extract explicit expressions for the leading exponential term in Theorem \ref{T: Main Theorem} using the local behavior of the integral kernel $\Phi_{\ell,\frac{h'}{k}}$ at $t=0$.

\begin{cor}\label{cor:leading_exponential_first_few_example}
For $n\in\mathbb N$ we have
\begin{align*}
\a_0 (n) &= \frac{e^{4 \pi \sqrt{\frac{n+\DD_0}{24}}}}{(n+\DD_0)^{\frac{3}{2}}}
\lp \frac{\pi}{6 \sqrt{2}} + \frac{71\pi^2 -432}{576 \sqrt{3}} (n+\DD_0)^{-\frac{1}{2}}
+ O\left(n^{-1}\right) \rp,
\\  \notag
\a_1 (n) &=
\frac{e^{4 \pi \sqrt{\frac{n+\DD_1}{24}}}}{n+\DD_1}
\lp \frac{1}{2 \sqrt{3}} +
\frac{23 \pi^2 -144}{288 \sqrt{2}   \pi} (n+\DD_1)^{-\frac{1}{2}}
+ \frac{9745 \pi^2 -19872}{55296 \sqrt{3}} (n+\DD_1)^{-1}
+ O\left(n^{-\frac{3}{2}}\right) \rp,
\\ \notag
\a_2 (n) &=
\frac{e^{4 \pi \sqrt{\frac{n+\DD_2}{24}}}}{n+\DD_2}
\lp \frac{1}{2 \sqrt{3}} +
\frac{25 \pi^2 -72}{144 \sqrt{2}   \pi} (n+\DD_2)^{-\frac{1}{2}}
+ \frac{2929 \pi^2 -10800}{13824 \sqrt{3}} (n+\DD_2)^{-1}
+ O\left(n^{-\frac{3}{2}}\right) \rp .
\end{align*}
\end{cor}

The paper is organized as follows. In Section \ref{sec:prelims}, we review facts on Maass forms, mock Maass theta functions, and false-indefinite theta functions. In Section \ref{sec:parity_partition_false_indefinite}, we study the false-indefinite theta functions $u_0$ and $u_1$ along with their natural companion $u_2$ and give the modular transformations of the associated Maass form.
In Section \ref{sec:obstruction_modularity}, we describe the obstruction to modularity as a Mordell-type integral, which allows one to discern the principal parts near each rational. We use this representation in Section \ref{sec:bounds_nonprincipal_parts} to bound the nonprincipal parts in preparation for
Section \ref{sec:circle_method}, where we apply the Circle Method and prove Theorem \ref{T: Main Theorem}.
In Section \ref{sec:leading_exponential}, we employ Theorem \ref{T: Main Theorem} to show
Corollary \ref{cor:leading_exponential_first_few_example} for the leading exponential term
with its general form given in
Proposition \ref{prop:alpha_n_main_exonential_term}.
In Section \ref{sec:conclusion}, we conclude with final remarks and point out potential future directions.
In Appendix \ref{app:false_indef_estimates}, we give some elementary estimates on the Fourier coefficients of false-indefinite theta functions. Finally, in Appendix \ref{app:taylor_integral_kernel}, we give the Taylor coefficients of the integral kernel at $t=0$.

\section*{Acknowledgements}

The authors thank Walter Bridges for his comments and feedback.
The first and the second author have received funding from the European Research Council (ERC) under the European Union's Horizon 2020 research and innovation programme (grant agreement No. $\hspace{-0.09cm}$101001179). The first and the third author were supported by the SFB/TRR 191 “Symplectic Structure in Geometry, Algebra and Dynamics”, funded by the DFG (Projektnummer 281071066 TRR 191). The views expressed in this article are those of the authors and do not reflect the official policy or position of the U.S. Naval Academy, Department of the Navy, the Department of Defense, or the U.S. Government.

\section{Preliminaries} \label{sec:prelims}
We start our discussion with background on Maass forms, mock Maass theta functions, and their relations to false-indefinite theta functions; see \cite{BN}, where further details can be found.

\subsection{Maass forms}\label{sec:maass_form}
We start by defining Maass forms, whose properties can be found e.g.~in \cite{MR}.
\begin{defn}
A collection $U_0, \dots, U_{N-1}$ of smooth functions $U_j : \H \to \C$ are said to form a {\it vector-valued Maass form} for $\mathrm{SL}_2\lp \Z \rp$ if they satisfy the following properties (with\footnote{
	More generally for any $w \in \IC$, we write $w_1$ for its real part and $w_2$ for its imaginary part.
} $\tau = \tau_1 + i \tau_2 \in \H$):
\begin{enumerate}[leftmargin=*]
\item For $M = \begin{psmallmatrix} a & b \\ c & d \end{psmallmatrix} \in \mathrm{SL}_2\!\lp \Z \rp$
and $\Psi_M$ a suitable multiplier system we have
\begin{align*}%\label{Vector-valued Maass transformation}
U_j\!\lp \dfrac{a\tau + b}{c\tau + d} \rp = \sum_{\ell=0}^{N-1} \Psi_M\!\lp j,\ell \rp U_\ell(\tau).
\end{align*}
\item There exists a constant $\lambda := \frac 14 - \nu^2 \in \C$ such that $\Delta(U_j)=\lambda U_j$ for each $j$, where
\begin{align*}
\Delta := - \tau_2^2\lp \dfrac{\partial^2}{\partial \tau_1^2} + \dfrac{\partial^2}{\partial \tau_2^2} \rp
\end{align*}
is the {\it hyperbolic Laplace operator}.
\item The functions $U_j$ grow at most polynomially towards the cusps.
\end{enumerate}
\end{defn}

From these assumptions it immediately follows that $U_j$ have Fourier expansions
\begin{align*}
U_j(\tau) = \sum_{n \in \Z + \beta_j} a_j\!\lp \tau_2; n \rp e^{2\pi i n \tau_1},
\end{align*}
where $a_j(\tau_2;n)$ has the form
\begin{align*}
a_j(\tau_2;n) = \begin{cases}
d_j(n) \sqrt{\tau_2} K_\nu\!\lp 2\pi |n| \tau_2 \rp & \text{if } n \not = 0, \\
b_j \log(\tau_2) \sqrt{\tau_2} + c_j \sqrt{\tau_2} & \text{if } n = 0 \text{ and } \nu = 0, \\
b_j \tau_2^{\frac 12 - \nu} + c_j \tau_2^{\nu+\frac 12} & \text{if } n = 0 \text{ and } \nu \not = 0,
\end{cases}
\end{align*}
and the coefficients $d_j(n)$ are polynomially bounded in $n$.
Here $K_\nu$ denotes the $K$-Bessel function of order $\nu$.
In this paper, we are only interested in Maass forms on $\mathrm{SL}_2\lp \Z \rp$ with Laplacian eigenvalue $\frac 14$ and $b_j = 0$, so we have Fourier expansions
\begin{align} \label{eq:Uj_Fourier_decomposition}
U_j(\tau) = c_j \sqrt{\tau_2} + \sqrt{\tau_2} \sum_{\substack{n \in \Z + \beta_j \\ n \not = 0}} d_j(n) K_0\!\lp 2\pi |n| \tau_2 \rp e^{2\pi i n \tau_1}.
\end{align}

Following Lewis and Zagier \cite{LZ}, we define the differential one-form
\begin{equation}\label{eq:lewis_zagier_one_form}
\left[ U_j(z), R_\tau(z) \right] := \left(\frac{\del}{\del z} U_j(z) \right)R_\tau(z) dz
+ U_j(z) \left(\frac{\del}{\del \bar{z}} R_\tau(z)\right) d\bar{z},
\end{equation}
where\footnote{Note that throughout we define square-roots using the principal branch of the logarithm.}
\begin{align*}
R_\tau(z) := \dfrac{\sqrt{z_2}}{\sqrt{(z-\tau)(\bar z - \tau)}} .
\end{align*}
This one-form is closed since $U_j$ and $R_\tau$ have Laplacian eigenvalue $\frac{1}{4}$. One can then use this one-form to relate $U_j$ to a family of $q$-series using an Eichler-type integral. More specifically, one can define the holomorphic functions $u_j : \H \to \C$ through
(following \cite{LZ} as well as \cite[Proposition 3.5]{LNR} for more details and including the constant term)
\begin{equation*}
u_j(\tau) := - \dfrac{2}{\pi} \int_\tau^{i\infty} \left[ U_j(z), R_\tau(z) \right] = - \dfrac{c_j}{\pi} + \sum_{\substack{n \in \Z + \beta_j \\ n > 0}} d_j(n) q^n .
\end{equation*}

Crucially $R_\tau (z)$ is also modular covariant. It has weight $(1,0)$ in $(\tau,z)$ including a sign factor appearing in its transformation to keep track of the branch of $R_\tau (z)$ that is exchanged by modular transformations. Together with the modularity of the Maass form $U_j$, this leads to the holomorphic quantum modularity of $u_j$.
More specifically, for $M = \begin{psmallmatrix} a & b \\ c & d \end{psmallmatrix} \in \mathrm{SL}_2\!\lp \Z \rp$  and $\t \in \mathbb{H}$ with $\t_1 \neq -\frac{d}{c}$
we have (see e.g. Remark 2.4 and Proposition 2.5 of \cite{BN} for more details)
\begin{equation*}
u_j\!\lp \dfrac{a\tau + b}{c\tau + d} \rp =
\sgn (c \tau_1 +d)
(c \tau +d)
\sum_{\ell=0}^{N-1} \Psi_M\!\lp j,\ell \rp  \lp u_\ell(\tau) + \mathcal{U}_{\ell,-\frac{d}{c}}(\tau) \rp,
\end{equation*}
where for $j\in\{0,1,\dots,N-1\}$ and $ \varrho \in \Q$ the obstruction to modularity is defined as
\begin{equation}\label{eq:obstruction_modularity}
\mathcal{U}_{j,\varrho}(\tau) := \dfrac{2}{\pi} \int_\varrho^{i\infty} \left[ U_j(z), R_\tau(z) \right]
\end{equation}
with a vertical integration path avoiding the branch cut from $\tau$ to $\bar{\tau}$.

\begin{rem}\label{rem:extension_obstruction}
We can analytically continue $\mathcal{U}_{j,\varrho}$ defined by  \eqref{eq:obstruction_modularity} for $\t_1 > \varrho$ to a holomorphic function
$\mathcal{U}_{j,\varrho}^\#$ on the cut plane $\IC \setminus (- \infty, \varrho]$ by deforming the path of integration and continuing to the second branch of square-root. This yields to the following transformation for $c >0$
\begin{equation*}
u_j\!\lp \dfrac{a\tau + b}{c\tau + d} \rp =
(c \tau +d)
\sum_{\ell=0}^{N-1} \Psi_M\!\lp j,\ell \rp  \lp u_\ell(\tau) + \mathcal{U}_{\ell,-\frac{d}{c}}^\#(\tau) \rp .
\end{equation*}
\end{rem}

\subsection{Mock Maass and false-indefinite theta functions} \label{sec:mock_maass}
We next discuss mock Maass theta functions as defined by Zwegers \cite{Zw}.
These are certain theta functions constructed out of indefinite binary quadratic forms that are non-modular eigenfunctions of the hyperbolic Laplacian with a controllable modular completion that  breaks the Laplacian eigenfunction property.
This set-up is quite analogous to Zwegers' indefinite theta functions and their non-holomorphic modular completions that lead to the theory of mock modular forms developed around mock theta functions and harmonic Maass forms \cite{BFOR,BF,ZwThesis}.
Mock Maass theta functions provide a natural framework to understand and generalize work of Cohen \cite{C} on Ramanujan's $\sigma$-function. This is done by recognizing the mock Maass theta function associated with the $\s$-function as one of the more symmetric cases where the difference between the mock Maass theta function and the modular completion vanishes. Consequently, the resulting mock Maass theta function is a Maass form that is both modular and a Laplacian eigenfunction.
To be more concrete, we consider an even, signature $(1,1)$ binary quadratic form\footnote{Throughout we write vectors in bold letters.} $Q : \IZ^2 \to \IZ$ with $Q(\bm{n}) = \frac{1}{2} \bm{n}^T A \bm{n}$, where $A$ is an integral symmetric matrix with even diagonal entries. We also introduce the corresponding bilinear form $B(\bm{n}, \bm{m}) := \bm{n}^T A \bm{m}$ and extend $Q, B$ to $\IR^2$.
The set of $\bm{c} \in \R^2$ with $Q(\bm{c}) = -1$ breaks into two connected components, which is also so for the set of $\bm{c} \in \R^2$ with $Q(\bm{c}) = 1$.
So if we pick two vectors $\bm{c_0}, \bm{c_0^\perp} \in \IR^2$ with $Q(\bm{c_0}) = -1$, $Q(\bm{c_0^\perp}) = 1$, and $B(\bm{c_0},\bm{c_0^\perp}) = 0$, the connected components $C_Q$ and $C_Q^\perp$ to which $\bm{c_0}$ and $\bm{c_0^\perp}$ belong to, respectively, are characterized as
\begin{equation*}
C_Q = \left\{ \bm{c} \in \R^2 : Q(\bm{c}) = -1 \text{ and } B(\bm{c}, \bm{c_0}) < 0 \right\},\quad
C_Q^\perp = \left\{ \bm{c} \in \R^2 : Q(\bm{c}) = 1 \text{ and } B\left(\bm{c}, \bm{c_0^\perp}\right) > 0 \right\}.
\end{equation*}
Finally, to parametrize the sets $C_Q$ and $C_Q^\perp$ we select a base change matrix $P \in \mathrm{GL}_2 (\mathbb{R})$ such that\footnote{
For convenience we use the reference quadratic form $x_1^2 - x_2^2$ as in \cite{BN}. See the comments there on how this relates to \cite{Zw} and why the final condition involving $C_Q^\perp$ should be added to the first two conditions imposed by \cite{Zw}.
}
\begin{equation}\label{eq:false_indef_P_choice}
A = P^T \mat{2 & 0 \\ 0 & -2} P, \qquad
P^{-1} \mat{0 \\ 1} \in C_Q, \qquad
P^{-1} \mat{1 \\ 0} \in C_Q^\perp
\end{equation}
and then parametrize $C_Q$ and $C_Q^\perp$ with $t \in \mathbb{R}$ by defining
\begin{equation}\label{eq:ct_definition}
\bm{c(t)} := P^{-1} \mat{\sinh (t) \\ \cosh (t)} \in C_Q
\quad \mbox{ and } \quad
\bm{c^\perp (t)} := P^{-1} \mat{\cosh (t) \\ \sinh (t)} \in C_Q^\perp.
\end{equation}
Using this setup, we recall Zwegers' definition of a mock Maass theta function (also see the comments in \cite{BN} on the constant term).
\begin{defn}\label{defn:mock_maass_theta_fnc}
For $j \in \{1,2\}$ let $\bm{c_j} = \bm{c (t_j)}\in C_Q$ and let $\bm{c_j^\perp} := \bm{c^\perp (t_j)}$.
For any $\bm{\mu}\in A^{-1}\mathbb{Z}^2/\mathbb{Z}^2$ we define the {\it mock Maass theta function} $\mT_{\bm\mu}$ by
\begin{align*}%\label{eq:mT_definition}
\mT_{\bm\mu}(\t) &:=
\sgn(t_2-t_1) \frac{\sqrt{\t_2}}{2} \sum_{\substack{\bm n \in \Z^2 + \bm\mu \\ \bm{n} \neq \bm{0}}}
\Big( 1- \sgn(B(\bm n,\bm{c_1})) \sgn (B (\bm n,\bm{c_2})) \Big)
K_0(2\pi Q(\bm n)\t_2) e^{2\pi i Q(\bm n)\t_1}
\\
&\!\! +
\sgn(t_2-t_1) \frac{\sqrt{\t_2}}{2} \sum_{\substack{\bm n \in \Z^2 + \bm\mu \\ \bm{n} \neq \bm{0}}}
\left( 1- \sgn\!\left(B\!\left(\bm n,\bm{c_1^\perp}\right)\right) \sgn \!\left(B \!\left(\bm n,\bm{c_2^\perp}\right)\right) \right)
K_0\left(-2\pi Q(\bm n)\t_2\right) e^{2\pi i Q(\bm n)\t_1} \\
&\!\!
+ (t_2-t_1) \sqrt{\t_2} \d_{\bm{\mu} \in \Z^2}.
\end{align*}
\end{defn}

The differential equations satisfied by $K_0$ imply immediately that $\mT_{\bm\mu}$ is an eigenfunction of $\Delta$ with eigenvalue $\frac 14$. Then, as described in Subsection \ref{sec:maass_form}, we can use the closed one-form defined in \eqref{eq:lewis_zagier_one_form} to relate the mock Maass theta function $\mT_{\bm\mu}$ to the false-indefinite theta function
\begin{align*}
\vartheta_{\bm\mu}(\tau) := \frac{t_1-t_2}{\pi}\delta_{\bm\mu\in\Z^2} + \dfrac{\sgn(t_2-t_1) }{2} \sum_{\substack{\nn\in\Z^2+\bm\mu \\ \nn\not={\bf 0}}} \lp 1 - \sgn\lp B(\nn,\bm{c_1}) \rp \sgn\lp B(\nn,\bm{c_2}) \rp \rp q^{Q(\nn)}
\end{align*}
through the Eichler-type integral
\begin{align}\label{eq:eichler_type_integral}
\vartheta_{\bm\mu}(\tau) = - \dfrac{2}{\pi} \int_\tau^{i\infty} \left[ \mT_{\bm\mu}(z), R_\tau(z) \right].
\end{align}

In general, mock Maass theta functions are not modular objects. Remarkably, in \cite{Zw} Zwegers
showed that there is a {modular completion} for the mock Maass theta function $\mT_{\bm\mu}$ given by
\begin{align*}
\widehat\mT_{\bm\mu}(\tau) := \sqrt{\tau_2} \sum_{\nn \in \Z^2 + \bm\mu} q^{Q(\nn)} \int_{t_1}^{t_2} e^{-\pi B\lp \nn, \bm{c}(\bm{t}) \rp^2 \tau_2} dt.
\end{align*}
According to Theorem 2.6 and Lemma 4.1 of \cite{Zw}, this completion modifies $\mT_{\bm\mu}$ as
\begin{equation} \label{Mock Maass Error to Modularity}
\widehat\mT_{\bm\mu} = \mT_{\bm\mu}
+ \varphi_{\bm\mu}^{[\bm{c_1}]}  - \varphi_{\bm\mu}^{[\bm{c_2}]},
\end{equation}
where the \textit{shadow contributions} $\varphi_{\bm\mu}^{[\bm{c_j}]}$ are given by
\begin{align} \label{eq:shadow definition}
\varphi_{\bm\mu}^{[\bm{c_j}]}\lp \tau \rp &:= \sqrt{\tau_2} \sum_{\nn\in \Z^2 + \bm\mu} \alpha_{t_j}\!\lp \nn \sqrt{\tau_2} \rp q^{Q(\nn)},\\
\intertext{and}
\nonumber\alpha_{t_j}\!\lp \bm{x} \rp &:= \begin{cases}
\int_{t_j}^\infty e^{-\pi B\lp \bm{x}, \bm{c}(t) \rp^2} dt & \text{if }
B\!\left(\bm{x}, \bm{c_j}\right) B\!\left(\bm{x},\bm{c_j^\perp}\right) > 0, \vspace{.05cm}\\
- \int_{-\infty}^{t_j} e^{-\pi B\lp \bm{x}, \bm{c}(t) \rp^2} dt & \text{if }
B\!\left(\bm{x}, \bm{c_j}\right) B\!\left(\bm{x},\bm{c_j^\perp}\right) < 0, \vspace{.05cm}\\
0 & \text{if } B\!\left(\bm{x}, \bm{c_j}\right) B\!\left(\bm{x},\bm{c_j^\perp}\right) = 0.
\end{cases}
\end{align}
Here we assume that $Q(\bm{n}) = 0$ has no solutions on $\mathbb{Z}^2 + \bm{\mu}$ except for $\bm{n} = \bm{0}$.
Due to these additions, the completion does not in general maintain the eigenvalue property possessed by $\mT_{\bm\mu}$, but in contrast it automatically satisfies a modular transformation law. In particular, we have \begin{equation}\label{eq:mock_maass_theta_transformation}
\widehat{\mT}_{\bm\mu} \lp \frac{a \t + b}{c \t + d} \rp
=
\sum_{\bm{\nu} \in A^{-1} \Z^2/ \Z^2}
\psi_{M,Q} (\bm{\mu}, \bm{\nu}) \widehat{\mT}_{\bm\nu} (\t)
\quad \mbox{for } \mat{a & b \\ c & d} \in \SL_2 (\Z),
\end{equation}
where
\begin{equation}\label{eq:theta_multiplier_system}
\psi_{M,Q} (\bm{\mu}, \bm{\nu})
:=
\begin{cases}
e^{2\pi i a b Q(\bm{\mu})} \d_{\bm{\mu}, \sgn (d) \bm{\nu}}
\qquad &\mbox{if } c = 0, \vspace{0.3cm} \\
\frac{1}{|c| \sqrt{|\det (A)|}} \displaystyle\sum_{\bm{m} \in \IZ^2 / c \IZ^2}
e^{\frac{2\pi i}{c} \lp a Q(\bm{m}+\bm{\mu}) - B(\bm{m}+\bm{\mu},\bm{\nu}) + d Q(\bm{\nu})  \rp}
\qquad &\mbox{if } c \neq 0,
\end{cases}
\end{equation}
where $\delta_{\bm{\mu}, \bm{\nu}}:=1$ if $\bm{\mu} = \bm{\nu}$ in $A^{-1} \Z^2/ \Z^2$ and 0 otherwise. We also require the following lemma from \cite{Zw} to determine whether a linear combination of false-indefinite theta functions is a Maass form. 

\begin{lemma} \label{L:Zwegers Lemma}
Let $Q$ be the even, signature $(1,1)$ quadratic form $Q(\bm{n}) = \frac{1}{2} \bm{n}^T A \bm{n}$ and $\varphi_{\bm{\mu}}^{[\bm{c}]}$ be as in \eqref{eq:shadow definition} for $\bm{\mu} \in \Q^2$ and $\bm{c} \in C_Q$. If $\gamma \in \mathrm{SL}_2\lp \Z \rp$ with $\gamma^T A \gamma = A$ and $\gamma C_Q = C_Q$, then
\begin{align*}
\varphi_{\gamma \bm{\mu}}^{[\gamma \bm{c}]} = \varphi_{\bm{\mu}}^{[\bm{c}]}.
\end{align*}
\end{lemma}

\section{False-indefinite theta functions and $p_\od^\eu(n)$}\label{sec:parity_partition_false_indefinite}

We next consider $p_\od^\eu$ and give details on its relationship to false-indefinite theta functions and Maass forms. Recalling the discussion of the generating function \eqref{Eq: F_od^eu generating function}, we only need to focus on the coefficients $\alpha_0$ and $\alpha_1$ in \eqref{eq:alpha_definition}
and their generating functions described in equations
\eqref{eq:alpha_generating_function} and \eqref{eq:u0_u1_definition}.
We start our analysis with \eqref{eq:sigma_theta_representation} and write it in the following format to clarify its nature as a false-indefinite theta functions (changing variables $n \mapsto -n-1$ in the term associated to $q^{2n+1}$)
\begin{equation*}
\sigma(q) = \frac{1}{2} \sum_{n,j \in \IZ}
\left( 1 + \sgn \lp n+j + \frac{1}{6} \rp \sgn \lp n-j + \frac{1}{6} \rp \right)
(-1)^{n+j} q^{\frac{n (3n+1)}2- j^2}.
\end{equation*}
Inserting this expression in \eqref{eq:u0_u1_definition}, we then find
\begin{align*}
u_0 (\t) &= - \frac{1}{2} \sum_{\a,\b \in \{0,1\}} (-1)^{\a+\b}
\sum_{\bm{n} \in \IZ^2 + \lp \frac{\a}{2}+\frac{\b}{4}+\frac{1}{24}, \frac{\a}{2} \rp}
\big( 1 + \sgn (2n_1+n_2) \sgn (2n_1-n_2) \big) q^{12n_1^2-2n_2^2},\\
u_1 (\t) &= - \frac{1}{2} \sum_{\a,\b \in \{0,1\}} (-1)^{\a+\b}
\sum_{\bm{n} \in \IZ^2 + \lp \frac{\a}{2}+\frac{\b}{4}+\frac{1}{24}, \frac{1-\a}{2} \rp}
\big( 1 + \sgn (2n_1+n_2)  \sgn (2n_1-n_2) \big) q^{12n_1^2-2n_2^2} .
\end{align*}
Now, $u_0$ and $u_1$ can be expressed in terms of false-indefinite theta functions in the terminology of Subsection \ref{sec:mock_maass} with\footnote{
Here we let $C_Q$ to be the set of vectors $\bm{c} \in \IR^2$ with $Q(\bm{c}) = -1$ that have their second components positive and correspondingly let $C_Q^\perp$ to be the component of $Q(\bm{c}) = 1$ vectors with the first component positive.
}
\begin{equation*}
A = \mat{24 & 0 \\ 0 & -4}, \qquad
\bm{c_1} = \frac{1}{\sqrt{6}} \mat{-1 \\ 3}, \quad \mbox{and} \quad
\bm{c_2} = \frac{1}{\sqrt{6}} \mat{1 \\ 3}
\end{equation*}
satisfying
\begin{equation*}
Q(\bm{c_1}) = Q(\bm{c_2}) = -1, \ \
B(\bm{c_1}, \bm{c_2}) = -10, \ \
B(\bm n,\bm{c_1}) = -2 \sqrt{6} (2n_1 + n_2), \ \
B (\bm n,\bm{c_2}) = 2 \sqrt{6} (2n_1 - n_2).
\end{equation*}
Then \eqref{eq:false_indef_P_choice} is satisfied with $P := \pmat{2\sqrt3 & 0 \\ 0 & \sqrt{2}}$
and find from equation \eqref{eq:ct_definition} the parameters
\begin{equation*}
t_1 = \log\lp \sqrt{3} - \sqrt{2} \rp, \qquad t_2 = \log\lp \sqrt{3} + \sqrt{2} \rp .
\end{equation*}
The resulting false-indefinite theta functions are (with $24\mu_1, 4\mu_2 \in \IZ$)
\begin{align*}%\label{eq:f_false_indefinite_definition}
f_{\bm\mu}(\t) &:=
- \frac{\mathrm{arccosh} (5)}{\pi}  \d_{\bm{\mu} \in \Z^2}
+\frac{1}{2} \sum_{\substack{\bm{n} \in \IZ^2 +\bm{\mu} \\ \bm{n} \neq \bm{0} }}
\big( 1 + \sgn (2n_1+n_2) \sgn (2n_1-n_2) \big) q^{12n_1^2-2n_2^2} .
\end{align*}
This produces the following mock Maass theta functions
(and modular completions $\wh{F}_{\bm\mu}$ by \eqref{Mock Maass Error to Modularity})
\begin{align}
F_{\bm\mu}(\t) &:=
\frac{\sqrt{\t_2}}{2} \hspace{-0.2cm} \sum_{\substack{\bm n \in \IZ^2 + \bm\mu \\ \bm{n} \neq \bm{0}}} \hspace{-0.2cm}
\left( 1 + \sgn (2n_1+n_2) \sgn (2n_1-n_2) \right)
K_0(2\pi \!\left(12n_1^2-2n_2^2\right)\! \t_2)
e^{2\pi i \left(12n_1^2-2n_2^2\right) \t_1}
\notag \\ &\qquad
+
\frac{\sqrt{\t_2}}{2} \hspace{-0.2cm} \sum_{\substack{\bm n \in \Z^2 + \bm\mu \\ \bm{n} \neq \bm{0}}} \hspace{-0.2cm}
\left( 1 - \sgn (3n_1+n_2) \sgn (3n_1-n_2) \right)
K_0(-2\pi \left(12n_1^2-2n_2^2\right) \t_2)
e^{2\pi i \left(12n_1^2-2n_2^2\right) \t_1}
\notag \\ &\qquad \qquad
+ \mathrm{arccosh} (5) \sqrt{\t_2} \d_{\bm{\mu} \in \Z^2} ,
\label{eq:Fmu_definition}
\end{align}
which by equation \eqref{eq:eichler_type_integral} satisfy
\begin{align*}
f_{\bm\mu}(\t)  = - \dfrac{2}{\pi} \int_\tau^{i\infty} \left[ F_{\bm\mu}(z), R_\tau(z) \right] .
\end{align*}
With these definitions in hand, for $j \in \{0,1\}$ we have
\begin{equation}\label{eq:u0_u1_false_indefinite_description}
u_j (\t)=\frac{1}{2} \lp \sum_{\bm{\mu} \in \mathcal{S}^+_j}  f_{\bm\mu}(\t)
-  \sum_{\bm{\mu} \in \mathcal{S}^-_j} f_{\bm\mu}(\t) \rp,
\end{equation}
where
\begin{align*}
\mathcal{S}^+_0 \!:=\!\! \left\{\!
\lp \frac{7}{24}, 0 \rp,
\lp \frac{17}{24}, 0 \rp,
\lp \frac{11}{24}, \frac{1}{2} \rp,
\lp \frac{13}{24}, \frac{1}{2} \rp
\!\right\},
\ 
\mathcal{S}^-_0 \!:=\!\! \left\{\!
\lp \frac{1}{24}, 0 \rp,
\lp \frac{23}{24}, 0 \rp,
\lp \frac{5}{24}, \frac{1}{2} \rp,
\lp \frac{19}{24}, \frac{1}{2} \rp
\!\right\}\! , \\
\mathcal{S}^+_1 \!:=\!\! \left\{\!
\lp \frac{11}{24}, 0 \rp,
\lp \frac{13}{24}, 0 \rp,
\lp \frac{7}{24}, \frac{1}{2} \rp,
\lp \frac{17}{24}, \frac{1}{2} \rp
\!\right\}\! ,
\ 
\mathcal{S}^-_1 \!:=\!\! \left\{ \!
\lp \frac{5}{24}, 0 \rp,
\lp \frac{19}{24}, 0 \rp,
\lp \frac{1}{24}, \frac{1}{2} \rp,
\lp \frac{23}{24}, \frac{1}{2} \rp
\!\right\}\! .
\end{align*}
We also extend the relation \eqref{eq:u0_u1_false_indefinite_description} to a third function with $j=2$ by setting
\begin{equation*}
\mathcal{S}^+_2 \!:=\!\! \left\{\!
\lp \frac{5}{12}, \frac{1}{4} \rp,
\lp \frac{7}{12}, \frac{1}{4} \rp,
\lp \frac{5}{12}, \frac{3}{4} \rp,
\lp \frac{7}{12}, \frac{3}{4} \rp
\!\right\} \! ,
\ 
\mathcal{S}^-_2 \!:=\!\! \left\{\!
\lp \frac{1}{12}, \frac{1}{4} \rp,
\lp \frac{11}{12}, \frac{1}{4} \rp,
\lp \frac{1}{12}, \frac{3}{4} \rp,
\lp \frac{11}{12}, \frac{3}{4} \rp
\!\right\}\! .
\end{equation*}
We see below that $u_2$ is in the $\mathrm{SL}_2(\Z)$-orbit of $u_0$ and $u_1$ and it completes these functions to a vector-valued modular object under the full modular group.

For $j \in \{0,1,2\}$ we also define the corresponding functions
\begin{equation}\label{eq:Uj_definition}
U_j :=\frac{1}{2} \lp \sum_{\bm{\mu} \in \mathcal{S}^+_j}  F_{\bm\mu}
-  \sum_{\bm{\mu} \in \mathcal{S}^-_j} F_{\bm\mu} \rp
\ \mbox{ and } \
\wh{U}_j :=\frac{1}{2} \lp \sum_{\bm{\mu} \in \mathcal{S}^+_j}  \wh{F}_{\bm\mu}
-  \sum_{\bm{\mu} \in \mathcal{S}^-_j} \wh{F}_{\bm\mu} \rp .
\end{equation}
To connect $p_\od^\eu$ and Maass forms, we show that $U_j=\widehat U_j$, which yields
the modular behavior of $u_j$.
\begin{prop} \label{P: U_j and u_j transformation}
For $j \in \{0,1,2\}$ we have $U_j = \wh{U}_j$. For any
$M = \begin{psmallmatrix} a & b \\ c & d \end{psmallmatrix} \in \SL_2 (\IZ)$, we have
\begin{align*}
U_j\!\lp \dfrac{a\tau + b}{c\tau + d} \rp = \sum_{\ell=0}^2 \Psi_{M}\!\lp j,\ell \rp U_\ell(\tau)
\end{align*}
and for $\t_1 \neq -\frac{d}{c}$
\begin{align*}
u_j\!\lp \dfrac{a\tau + b}{c\tau + d} \rp =
\sgn (c \tau_1 +d) \lp c\tau + d \rp \sum_{\ell=0}^2 \Psi_{M}\!\lp j,\ell \rp \lp u_\ell\!\lp \tau \rp + \mathcal U_{\ell, - \frac dc}\!\lp \tau \rp \rp,
\end{align*}
where $\mathcal{U}_{j, - \frac dc}$ is defined as in \eqref{eq:obstruction_modularity} and the multiplier system $\Psi_{M}$ is given by
\begin{equation}\label{eq:U_modular_transformation_multiplier}
\Psi_{M}(j,\ell)
=
\sum_{\bm{\mu} \in \mathcal{S}^+_j} \psi_{M,Q} (\bm{\mu}, \bm{\nu})
-  \sum_{\bm{\mu} \in \mathcal{S}^-_j} \psi_{M,Q} (\bm{\mu}, \bm{\nu})
\ \ \mbox{with } \bm{\nu} \in \mathcal{S}^+_\ell
\mbox{ and } Q(\bm{n}) = 12n_1^2-2n_2^2.
\end{equation}
Here the multiplier system $\psi_{M,Q}(\bm{\mu}, \bm{\nu})$ is defined in \eqref{eq:theta_multiplier_system}.
\end{prop}
\begin{remark}
For modular translation and inversion, i.e., $T = \begin{psmallmatrix} 1 & 1 \\ 0 & 1 \end{psmallmatrix}$ and $S = \begin{psmallmatrix} 0 & -1 \\ 1 & 0 \end{psmallmatrix}$,
the multiplier system is given by\footnote{
Here the entry of the matrix $\Psi_M$ at position $(j+1,k+1)$ is given by $\Psi_M (j,k)$.
}
\begin{equation}\label{eq:U_multiplier_T_S}
\Psi_{T} = \mat{
\z_{48} & 0 & 0 \\
0 & \z_{48}^{25} & 0 \\
0 & 0 & \z_{24}^{23}
}
\andd
\Psi_{S} = \frac{1}{2} \mat{
1 & 1 & \sqrt{2} \\
1 & 1 & -\sqrt{2} \\
\sqrt{2} & -\sqrt{2} & 0
},
\end{equation}
where here and throughout we define $\z_N := e^{\frac{2\pi i}{N}}$.
\end{remark}
\begin{proof}[Proof of Proposition \ref{P: U_j and u_j transformation}]
The matrix $\gamma = \begin{psmallmatrix} 5 & 2 \\ 12 & 5 \end{psmallmatrix}$ is in $\mathrm{SL}_2(\Z)$ with $\gamma^T A \gamma = A$ and $\gamma \bm{c_1} = \bm{c_2}$ (so in particular $\g C_Q = C_Q$).
It is also straightforward to show that $\gamma$ preserves each $\mathcal S_j^\pm$ (mod $1$).
Therefore by Lemma \ref{L:Zwegers Lemma} we have $U_j = \widehat U_j$. Since $U_j$ inherits modular transformation laws from $\widehat U_j$, which consequently lead to the modular transformations of $u_j$ according to the discussion of Subsection \ref{sec:maass_form}, we only need to compute the transformation for $\widehat U_j$ following the transformations of $\wh{F}_{\bm{\mu}}$ from equations \eqref{eq:mock_maass_theta_transformation} and \eqref{eq:theta_multiplier_system}.
We next compute the translation and inversion transformations of $\widehat U_j$ and show that $\widehat U_j$ transform among themselves with multiplier system given by \eqref{eq:U_multiplier_T_S}. The modular translation property immediately follows from
\begin{equation*}%\label{eq:T_transformation_phases}
Q (\bm{\mu}) \equiv
\begin{cases}
\frac{1}{48} \pmod{1} \quad
& \mbox{if } \bm{\mu} \in \mathcal{S}^+_0 \cup \mathcal{S}^-_0,
\\
\frac{25}{48} \pmod{1} \quad
& \mbox{if } \bm{\mu} \in \mathcal{S}^+_1 \cup \mathcal{S}^-_1,
\\
\frac{23}{24} \pmod{1} \quad
& \mbox{if } \bm{\mu} \in \mathcal{S}^+_2 \cup \mathcal{S}^-_2.
\end{cases}
\end{equation*}
For the modular inversion property, we start with the modular transformations of the completed mock Maass theta functions given in \eqref{eq:mock_maass_theta_transformation}, which yields that for any $24\mu_1, 4\mu_2 \in \IZ$ we have
\begin{equation*}
\wh{F}_{\bm{\mu}} \!\lp -\frac{1}{\t} \rp
=
\sum_{\bm{\nu} \in A^{-1} \Z^2/ \Z^2} \psi_{S,Q} (\bm{\mu}, \bm{\nu})
\wh{F}_{\bm\nu} (\t)
\where
\psi_{S,Q} (\bm{\mu}, \bm{\nu})
= \frac{1}{4 \sqrt{6}}  e^{-2\pi i (24 \mu_1 \nu_1 - 4 \mu_2 \nu_2)}.
\end{equation*}
Then we have
\begin{equation*}
\wh{U}_j \! \lp -\frac{1}{\t} \rp
=
\frac{1}{2} \sum_{\bm{\nu} \in A^{-1} \Z^2/ \Z^2}
\lp \sum_{\bm{\mu} \in \mathcal{S}^+_j} \psi_{S,Q} (\bm{\mu}, \bm{\nu})
-  \sum_{\bm{\mu} \in \mathcal{S}^-_j} \psi_{S,Q} (\bm{\mu}, \bm{\nu})  \rp
\wh{F}_{\bm\nu} (\t)
\end{equation*}
and the inversion transformation for $\wh{U}_j$ as described by the multiplier system in \eqref{eq:U_multiplier_T_S} follows from
\begin{align*}
\sum_{\bm{\mu} \in \mathcal{S}^+_0} \psi_{S,Q} (\bm{\mu}, \bm{\nu})
-  \sum_{\bm{\mu} \in \mathcal{S}^-_0} \psi_{S,Q} (\bm{\mu}, \bm{\nu})
&=
\begin{cases}
\pm \frac{1}{2} &\quad \mbox{if } \bm{\nu} \in \mathcal{S}_0^\pm,
\\
\pm \frac{1}{2} &\quad \mbox{if } \bm{\nu} \in \mathcal{S}_1^\pm,
\\
\pm \frac{1}{\sqrt{2}} &\quad \mbox{if } \bm{\nu} \in \mathcal{S}_2^\pm,
\\
0 &\quad \mbox{otherwise},
\end{cases}\\
\sum_{\bm{\mu} \in \mathcal{S}^+_1} \psi_{S,Q} (\bm{\mu}, \bm{\nu})
-  \sum_{\bm{\mu} \in \mathcal{S}^-_1} \psi_{S,Q} (\bm{\mu}, \bm{\nu})
&=
\begin{cases}
\pm \frac{1}{2} &\quad \mbox{if } \bm{\nu} \in \mathcal{S}_0^\pm,
\\
\pm \frac{1}{2} &\quad \mbox{if } \bm{\nu} \in \mathcal{S}_1^\pm,
\\
\mp \frac{1}{\sqrt{2}} &\quad \mbox{if } \bm{\nu} \in \mathcal{S}_2^\pm,
\\
0 &\quad \mbox{otherwise},
\end{cases}\\
\sum_{\bm{\mu} \in \mathcal{S}^+_2} \psi_{S,Q} (\bm{\mu}, \bm{\nu})
-  \sum_{\bm{\mu} \in \mathcal{S}^-_2} \psi_{S,Q} (\bm{\mu}, \bm{\nu})
&=
\begin{cases}
\pm \frac{1}{\sqrt{2}} &\quad \mbox{if } \bm{\nu} \in \mathcal{S}_0^\pm,
\\
\mp \frac{1}{\sqrt{2}} &\quad \mbox{if } \bm{\nu} \in \mathcal{S}_1^\pm,
\\
0 &\quad \mbox{otherwise}.
\end{cases}
\end{align*}
The general transformation with the multiplier system \eqref{eq:U_modular_transformation_multiplier} follows from $S$ and $T$.
\end{proof}

\section{The Obstruction to Modularity}\label{sec:obstruction_modularity}
In this section, we derive ``Mordell-type'' integral representations\footnote{This terminology aligns with the analogous integral representation in the world of mock theta functions.} for the obstruction to modularity (for $c>0$ and $\t_1 > -\frac{d}{c}$)
\begin{equation}\label{eq:obstruction_modularity_def}
\mathcal U_{j, -\frac dc}\lp \tau \rp := \dfrac{2}{\pi} \int_{-\frac dc}^{i\infty} \left[ U_j(z), R_\tau(z) \right],
\end{equation}
which appear in Proposition \ref{P: U_j and u_j transformation}, and its extension to the whole upper half-plane (see Remark \ref{rem:extension_obstruction}).
These are crucial for determining the contribution of the principal parts to the asymptotic growth of $\a_j$ (and hence of $p_\od^\eu$), as well as precise bounds on the error term. 
For this calculation, we follow \eqref{eq:Uj_Fourier_decomposition} and denote
by $d_j$ the Fourier coefficients in the expansion
\begin{align} \label{eq:Uj_general_form}
U_j(z) = \sqrt{z_2} \sum_{n \in \Z + \beta_j} d_j(n) K_0\!\lp 2\pi |n| z_2 \rp e^{2\pi i n z_1},
\end{align}
where $j \in \{0,1,2\}$ and $\b_0 := \frac{1}{48}$, $\b_1 := \frac{25}{48}$, $\b_2 := \frac{23}{24}$ (see Proposition \ref{P: U_j and u_j transformation} and \eqref{eq:U_multiplier_T_S}). In particular, there is no constant term in \eqref{eq:Uj_general_form}. The calculation of the Mordell-type representation for $\mathcal{U}_{j,-\frac dc}$
(or rather for the combination $u_j + \mathcal{U}_{j,-\frac dc}$) is lengthy, and is broken down into several steps.

\subsection{Expanding the obstruction to modularity}
We first insert the Fourier expansion \eqref{eq:Uj_general_form} in \eqref{eq:obstruction_modularity_def} for the obstruction to modularity and then simplify it.

\begin{lemma} \label{L: First Mordell Lemma}
Let $-\frac dc \in \Q$ and $\t \in \H$ be such that $\t_1 > -\frac dc$. Then for
$j \in \{0,1,2\}$ we have
\begin{align*}
\mathcal U_{j,- \frac dc}\lp \t \rp = - \dfrac{1}{\pi} \int_0^\infty \sum_{n \in \Z+\beta_j} d_j(n) e^{- \frac{2\pi i d n}{c}} \dfrac{t K_0\lp 2\pi |n| t \rp}{\sqrt{t^2 + \lp \t + \frac dc \rp^2}} \lp 2\pi n + \dfrac{i\lp \t + \frac dc \rp}{t^2 + \lp \t + \frac dc \rp^2} \rp dt.
\end{align*}
\end{lemma}
\begin{proof}
We start by writing (see e.g. proof of Proposition 2.3 of \cite{BN})
\begin{equation*}
\calU_{j,-\frac dc}(\t) = \frac{1}{2\pi}\int_0^\infty
\left(4it U_j'\!\left(-\frac dc+it\right)-\frac{t+i\left(\t+\frac dc\right)}{t-i\left(\t+\frac dc\right)} U_j\!\left(-\frac dc+it\right)\right) \frac{dt}{\sqrt{t} \sqrt{t^2+\left(\t+\frac dc\right)^2}}.
\end{equation*}
Now we use the Fourier expansion of $U_j$ given in \eqref{eq:Uj_general_form} and the corresponding Fourier expansion
\begin{equation*}
-iU_j'(z) = \frac1{4\sqrt{z_2}}\sum_{n\in\Z+\b_j} d_j(n)\big((4\pi n z_2-1)K_0(2\pi|n|z_2)+4\pi|n|z_2K_1(2\pi|n|z_2)\big)e^{2\pi inz_1},
\end{equation*}
derived employing $K_0'(x)=-K_1(x)$. This yields
\begin{align*}
\calU_{j,-\frac dc}(\t)
&= -\frac1\pi \int_0^\infty \sum_{n\in\Z+\b_j} d_j(n)e^{-\frac{2\pi i d n}{c}} \frac{1}{\sqrt{t^2+\left(\t+\frac dc\right)^2}}\\
&\hspace{4cm}\times \left(\left(2\pi nt+\frac{i\left(\t+\frac dc\right)}{t-i\left(\t+\frac dc\right)}\right)K_0(2\pi|n|t)+2\pi|n|tK_1(2\pi|n|t)\right) dt,
\end{align*}
which we rewrite as
\begin{align*}
\calU_{j,-\frac dc}(\t) &= -\frac{1}{\pi}
\int_0^\infty
\sum_{n\in\Z+\b_j} d_j(n)e^{-\frac{2\pi i d n}{c}}
\\
&\hspace{1cm}\times
\lp \frac{t}{\sqrt{t^2+\left(\t+\frac dc\right)^2}}
\left(2\pi n+\frac{i\left(\t+\frac dc\right)}{t^2+\left(\t+\frac dc\right)^2}\rp
K_0(2\pi|n|t)
-
\frac{\del}{\del t}
\frac{tK_0 (2\pi|n|t)}{\sqrt{t^2+\left(\t+\frac dc\right)^2}}\rp dt.
\end{align*}
For the second term, the exponential decay of the $K$-Bessel function leads to the uniform convergence of the series on compact subsets of $t>0$ (with or without the derivative). So we can interchange the summation and differentiation and rewrite using \eqref{eq:Uj_general_form}
\begin{multline*}
\calU_{j,-\frac dc}(\t)=-\frac{1}{\pi}
\int_0^\infty \lp
\sum_{n\in\Z+\b_j} d_j(n)e^{-\frac{2\pi i d n}{c}}
\frac{t K_0(2\pi|n|t)}{\sqrt{t^2+\left(\t+\frac dc\right)^2}}
\left(2\pi n+\frac{i\left(\t+\frac dc\right)}{t^2+\left(\t+\frac dc\right)^2}\rp
 \right.
\\
\left. -
\frac{\del}{\del t} \lp U_j \lp - \frac{d}{c} + i t \rp
\sqrt{\frac{t}{t^2+\left(\t+\frac dc\right)^2}} \rp \rp dt.
\end{multline*}
Moreover, we can separate the integral of the total derivative term and explicitly evaluate it to zero thanks to the exponential decay of $U_j(-\frac{d}{c}+it)$  as $t \to \infty$ and as $t \to 0^+$.
The exponential decay towards $t \to \infty$ follows from the absence of constant terms noted in equation \eqref{eq:Uj_general_form}.
Then the modular transformations in Proposition \ref{P: U_j and u_j transformation} imply the exponential decay as $t \to 0^+$.
\end{proof}

We ultimately want to interchange the order of the sum and integral in the remaining term. The problem is that the exponential decay of $U_j \lsp -\frac dc + it \rsp$ and its first derivative as $t \to 0^+$, although in fact true, cannot be seen in the Fourier expansion.
Therefore the interchange of the sum and integral is not justified, and our next
step is to regularize the integral and to separate the component of the integral that is causing the problem. For this, let
\begin{align} \label{eq:KK Definition}
	\KK(x) := x  K_1 (x), \quad \KK_{\t,\frac dc}(n) :=
	\int_0^\infty \frac{t\lp i \lp \t + \frac{d}{c} \rp  K_0 (2 \pi |n| t) - \sgn (n) t K_1 (2 \pi |n| t) \rp}{\lp t^2 +\left(\t+\frac dc\right)^2 \rp^{\frac{3}{2}}}
	dt.
\end{align}

\begin{lemma} \label{L: Second Mordell Lemma}
Let $-\frac dc \in \Q$ and $\t \in \H$ be such that $\t_1 > -\frac dc$. Then for $j \in \{0,1,2\}$ we have
\begin{multline*}
\calU_{j,-\frac dc}(\t)
=
-\frac{1}{2\pi^2 \lp \t + \frac{d}{c} \rp} \lim_{\d \to 0^+}
\sum_{n\in\Z+\b_j} \frac{d_j(n)e^{-\frac{2\pi i d n}{c}}}{n}
\KK(2 \pi |n| \d)
- \frac{1}{\pi} \sum_{n\in\Z+\b_j} d_j(n)e^{-\frac{2\pi i d n}{c}}
\KK_{\t,\frac dc}(n).
\end{multline*}
\end{lemma}
\begin{proof}
We start by rewriting the expression for $\calU_{j,-\frac dc}$ given in Lemma \ref{L: First Mordell Lemma} as
\begin{equation*}
\calU_{j,-\frac dc}(\t)
=
-\frac{1}{\pi}
\lim_{\d \to 0^+}
\int_\d^\infty
\sum_{n\in\Z+\b_j} d_j(n)e^{-\frac{2\pi i d n}{c}}
\frac{tK_0(2\pi|n|t)}{\sqrt{t^2+\left(\t+\frac dc\right)^2}}
\left(2\pi n+\frac{i\left(\t+\frac dc\right)}{t^2+\left(\t+\frac dc\right)^2}\rp
dt.
\end{equation*}
Since for any $x>0$ we have (see e.g. 10.37.1 and 10.39.2 of \cite{NIST}),
\begin{equation*}%\label{eq:Bessel_K0_bounds}
0 \leq K_0 (x) < K_{\frac{1}{2}} (x) = \sqrt{\frac{\pi}{2x}} e^{-x},
\end{equation*}
for $t > \d$ we can estimate
\begin{equation*}
0 \leq K_0(2\pi|n|t) < \frac{e^{-2 \pi |n| t}}{2\sqrt{|n| t}}
\leq \frac{e^{- \pi \min\{\b_j, 1-\b_j\} t} e^{- \pi |n| \d}}{2\sqrt{|n| t}} ,
\end{equation*}
where for the second bound we use $|n| \geq \min \{\b_j, 1-\b_j\}$ as $\b_j \in (0,1)$. Thus the combined integral and sum is absolutely convergent and we can switch the order for $\d > 0$. So we have
\begin{equation*}
\calU_{j,-\frac dc}(\t)
=
-\frac{1}{\pi} \lim_{\d \to 0^+} \sum_{n\in\Z+\b_j} d_j(n)e^{-\frac{2\pi i d n}{c}}
\int_\d^\infty
\frac{tK_0(2\pi|n|t)}{\sqrt{t^2+\left(\t+\frac dc\right)^2}}
\left(2\pi n+\frac{i\left(\t+\frac dc\right)}{t^2+\left(\t+\frac dc\right)^2}\rp
dt.
\end{equation*}
Integrating the term with $2 \pi n$ by parts using that $\frac{\del}{\del t} (- t K_1 (Ct)) = Ct K_0 (Ct)$ for $C>0$, we find
\begin{multline*}
\calU_{j,-\frac dc}(\t)
=
-\frac{1}{\pi} \lim_{\d \to 0^+} \sum_{n\in\Z+\b_j} d_j(n)e^{-\frac{2\pi i d n}{c}}
\lp \sgn (n) \frac{\d K_1 (2 \pi |n| \d)}{\sqrt{\d^2+\left(\t+\frac dc\right)^2}}
\right.
\\
\left.
+ \int_\d^\infty \frac{t}{\lp t^2 +\left(\t+\frac dc\right)^2 \rp^{\frac{3}{2}}}
\lp i \lp \t + \frac{d}{c} \rp K_0 (2 \pi |n| t) - \sgn (n) t K_1  (2 \pi |n| t) \rp
dt
\rp .
\end{multline*}
Due to the exponential decay of $K_1 (2 \pi |n| \d)$ as $n \to \pm \infty$, the sum on $n$ can be separated to yield
\begin{multline}\label{eq:calU_integral_twosums}
\calU_{j,-\frac dc}(\t)
=
-\frac{1}{\pi} \lim_{\d \to 0^+}
\lp
\sum_{n\in\Z+\b_j} d_j(n)e^{-\frac{2\pi i d n}{c}}
\sgn (n) \frac{\d  K_1 (2 \pi |n| \d)}{\sqrt{\d^2+\left(\t+\frac dc\right)^2}}
\right.
\\
\left.
+ \sum_{n\in\Z+\b_j} \! d_j(n)e^{-\frac{2\pi i d n}{c}} \!
\int_\d^\infty \frac{t \lp i \lp \t + \frac{d}{c} \rp K_0 (2 \pi |n| t) - \sgn (n) t K_1  (2 \pi |n| t) \rp}{\lp t^2 +\left(\t+\frac dc\right)^2 \rp^{\frac{3}{2}}}
dt
\rp .
\end{multline}
We now focus on the contribution of the second line and make the change of variables as $t \mapsto \frac{t}{|n|}$ to rewrite the integral as
\begin{equation}\label{eq:calU_integral_convergent_delta_part}
\frac{1}{n^2} \int_{\d |n|}^\infty
\frac{t}{\lp \frac{t^2}{n^2} +\left(\t+\frac dc\right)^2 \rp^{\frac{3}{2}}}
\lp i \lp \t + \frac{d}{c} \rp K_0 (2 \pi t) -\frac{t}{n} K_1  (2 \pi t) \rp
dt .
\end{equation}
Because we assume $\t_1 > -\frac{d}{c}$, we can bound
\begin{equation*}
\frac{1}{\left| \frac{t^2}{n^2} +\left(\t+\frac dc\right)^2 \right|}
=
\frac{1}{\left| \lp \frac{t^2}{n^2} + \lp \t_1 + \frac{d}{c} \rp^2 - \t_2^2 \rp
+2i \t_2 \lp \t_1 + \frac{d}{c} \rp \right|}
\leq
\frac{1}{2 \t_2 \lp \t_1 + \frac{d}{c} \rp}.
\end{equation*}
Thus for $\d \geq 0$, we have the estimates
\begin{align*}
\left| \int_{\d |n|}^\infty
\frac{i t \lp \t + \frac{d}{c} \rp K_0 (2 \pi t)}{\lp \frac{t^2}{n^2} +\left(\t+\frac dc\right)^2 \rp^{\frac{3}{2}}} dt
\right|
&\leq
\frac{\left| \t + \frac{d}{c} \right|}{\lp 2 \t_2 \lp \t_1 + \frac{d}{c} \rp \rp^{\frac{3}{2}}}
\int_0^\infty tK_0 (2 \pi t) dt,\\
\left| \int_{\d |n|}^\infty
\frac{t^2 K_1 (2 \pi t)}{\lp \frac{t^2}{n^2} +\left(\t+\frac dc\right)^2 \rp^{\frac{3}{2}}} dt
\right|
&\leq
\frac{1}{\lp 2 \t_2 \lp \t_1 + \frac{d}{c} \rp \rp^{\frac{3}{2}}}
\int_0^\infty t^2K_1 (2 \pi t) dt .
\end{align*}
Therefore, the integral \eqref{eq:calU_integral_convergent_delta_part} is uniformly $\ll \frac{1}{n^2}$ for $\d \geq 0$. Together with the bound $|d_j (n)| \ll \sqrt{|n|}$ from Corollary \ref{cor:dj_bound}, we then find that the sum over $n$ in the second line of \eqref{eq:calU_integral_twosums} defines a function of $\d$ that is continuous for $\d \geq 0$. So the limit $\d \to 0^+$ can be separately taken for this term by simply setting $\d=0$.
The desired result then holds thanks to the fact that for $\t_1 > - \frac{d}{c}$ we have
\begin{equation*}
\lim_{\d \to 0^+} \frac{1}{\sqrt{\d^2+\left(\t+\frac dc\right)^2}}
= \frac{1}{\t+\frac dc}.
\qedhere
\end{equation*}
\end{proof}

Our next goal is to let $\d \to 0^+$ in the first term of Lemma \ref{L: Second Mordell Lemma}. We see below that this limit is obtained by setting $\d=0$ (noting that $\lim_{x \to 0^+} \KK (x) = 1$) and taking the sum symmetrically. Here by a
{\it symmetric sum} over a discrete set $S$, we mean
\begin{align*}
\psum_{n \in S} := \lim_{X \to \infty} \sum_{\substack{n \in S \\ |n| \leq X}}.
\end{align*}
We first show that the first term of Lemma \ref{L: Second Mordell Lemma} is convergent with $\d=0$.

\begin{lem} \label{lem:1_over_n_term_convergence}
The series $\displaystyle\psum_{n\in\Z+\b_j} \frac{d_j(n) e^{-\frac{2\pi i d n}{c}}}{n}$ is convergent for any $c \in \IN$, $d \in \IZ$, and $j \in \{0,1,2\}$.
\end{lem}
\begin{proof}
For $M \in \IN$, we consider the partial sums
\begin{equation}\label{eq:1_over_n_term_partial_sum_M}
\sum_{\substack{n \in \IZ+\b_j \\ -c(M+1)+\b_j \leq n \leq c(M+1)-1+\b_j}}
\frac{d_j (n)  e^{-\frac{2\pi idn}c}}n.
\end{equation}
By Corollary \ref{cor:dj_bound}, the summand vanishes as $n \to \pm \infty$ and up to $O(c)$ terms on the boundary, the limit $M \to \infty$ of the partial sum \eqref{eq:1_over_n_term_partial_sum_M}, if it exists, is equivalent to the convergence of the symmetric sum in this lemma and the value of the limit $M \to \infty$ gives its value.
We start by changing variables as $n = cm+r+\b_j$ with $m \in \{-M-1,\ldots,M \}$ and $r \in \{0,1,\ldots,c-1\}$ to rewrite \eqref{eq:1_over_n_term_partial_sum_M} as
\begin{multline}
\frac{1}{c} \sum_{r=0}^{c-1} e^{-\frac{2\pi i d}{c} (r+\b_j)}
\sum_{m=-M-1}^M \frac{d_j (cm+r+\b_j)}{m + \frac{r+\b_j}{c}}
\\
=
\frac{1}{c} \sum_{r=0}^{c-1} e^{-\frac{2\pi i d}{c} (r+\b_j)}
\sum_{m=0}^M \lp \frac{d_j (cm+r+\b_j)}{m + \frac{r+\b_j}{c}}
- \frac{d_j (-c(m+1)+r+\b_j)}{m + \frac{c-r-\b_j}{c}} \rp,
\label{eq:symmetric_sum_rewriting}
\end{multline}
where we change $m \mapsto -m-1$ for the sum over $-M-1\le m\le-1$.
Now we use summation by parts to the sum over $m$ and obtain
\begin{align}
&\sum_{m=0}^{M-1} \lp
\frac{\displaystyle \sum_{n=0}^m d_j (cn+r+\b_j)}{\lp m + \frac{r+\b_j}{c} \rp \lp m + 1 + \frac{r+\b_j}{c} \rp}
- \frac{\displaystyle \sum_{n=0}^m d_j (-c(n+1)+r+\b_j)}{\lp m + \frac{c-r-\b_j}{c} \rp \lp m + 1 + \frac{c-r-\b_j}{c} \rp} \rp
\notag \\ & \qquad \qquad
+ \frac{1}{M + \frac{r+\b_j}{c}} \sum_{n=0}^M d_j (cn+r+\b_j)
- \frac{1}{M + \frac{c-r-\b_j}{c}} \sum_{n=0}^M d_j (-c(n+1)+r+\b_j) .
\label{eq:dj_over_n_partial_sum_summation_by_parts}
\end{align}
By Corollary \ref{cor:dj_partial_sum_bound}, for $m \geq c$ we have (note that $0< \frac{r+\b_j}{c} < 1$)
\begin{align}
\sum_{n=0}^m d_j (cn+r+\b_j) &= c \CA_{j,r,c} \lp m + \frac{r+\b_j}{c} \rp
+ O \lp c^{\frac{3}{2}} \sqrt{m} \rp,
\label{eq:dj_partial_sum_bound1}
\\
\sum_{n=0}^md_j (-c(n+1)+r+\b_j) &= c \CA_{j,r,c} \lp m + \frac{c-r-\b_j}{c} \rp
+ O \lp c^{\frac{3}{2}} \sqrt{m} \rp,
\label{eq:dj_partial_sum_bound2}
\end{align}
with the implied constants on the error terms independent of $j$ and $r$.
In particular, for $M \geq c$ the second line of equation \eqref{eq:dj_over_n_partial_sum_summation_by_parts} is bounded as
\begin{equation*}
\lp c \CA_{j,r,c} + \frac{O\!\lp c^{\frac{3}{2}} \sqrt{M} \rp}{M + \frac{r+\b_j}{c}} \rp
-
\lp c \CA_{j,r,c} + \frac{O\!\lp c^{\frac{3}{2}} \sqrt{M} \rp}{M + \frac{c-r-\b_j}{c}} \rp \to 0 \quad \text{as $M\to\infty$}.
\end{equation*}

We next consider the summand on the first line of
\eqref{eq:dj_over_n_partial_sum_summation_by_parts} for $m \geq c$ and estimate it as
\begin{equation*}
\resizebox{1\hsize}{!}{$\dfrac{c \CA_{j,r,c} \lp 1 - 2 \dfrac{r+\b_j}{c} \rp}{\lp m + 1 + \dfrac{r+\b_j}{c} \rp \lp m + 1 + \dfrac{c-r-\b_j}{c} \rp}
+ \dfrac{O\!\lp c^{\frac{3}{2}} \sqrt{m} \rp}{\lp m + \dfrac{r+\b_j}{c} \rp \lp m + 1 + \dfrac{r+\b_j}{c} \rp}
+ \dfrac{O\!\lp c^{\frac{3}{2}} \sqrt{m} \rp}{\lp m + \dfrac{c-r-\b_j}{c} \rp \lp m + 1 + \dfrac{c-r-\b_j}{c} \rp} .$}
\end{equation*}
So overall the summand is $\ll m^{-\frac{3}{2}}$ as $m \to \infty$ and hence the sum on the first line of equation \eqref{eq:dj_over_n_partial_sum_summation_by_parts} is convergent as $M \to \infty$. This proves the required convergence.
\end{proof}

Given the convergence of the symmetric sum of Lemma \ref{lem:1_over_n_term_convergence}, we now take the limit $\d \to 0^+$ in the first term of Lemma \ref{L: Second Mordell Lemma} and show that this symmetric sum is the resulting limit.
\begin{lem}\label{lem:symmetric_sum_limit}
For $c \in \IN$, $d \in \IZ$, and $j \in \{0,1,2\}$ we have
\begin{equation*}
\lim_{\d \to 0^+}
\sum_{n\in\Z+\b_j} \frac{d_j(n)e^{-\frac{2\pi i d n}{c}}}{n}
\KK (2 \pi |n| \d)
=
\psum_{n\in\Z+\b_j} \frac{d_j(n)e^{-\frac{2\pi i d n}{c}}}{n} .
\end{equation*}
\end{lem}
\begin{proof}
We define $b(n) := \frac{d_j(n)}{n} e^{-\frac{2\pi i d n}{c}}$ for $n \in \IZ + \b_j$ and
\begin{equation*}
B_{\delta} := \sum_{n\in\Z+\b_j} b(n)  \KK (2 \pi |n| \d),
\qquad
B := \psum_{n\in\Z+\b_j} b(n),
\andd
B(x) := \sum_{\substack{n \in \IZ + \b_j \\ |n| \leq x}} b(n).
\end{equation*}
Our goal is to prove that $\lim_{\d \to 0^+} B_{\d} = B$.
The first step towards that is the identity %{\bf KB: Only agree if n=}
\begin{equation}\label{eq:s_delta_summation_by_parts}
B_{\d} = 2 \pi \d \int_0^\infty B(x)  \KK_0 (2 \pi \d x)  dx,
\end{equation}
where $\KK_0 (x) := - \KK' (x) = x  K_0 (x)$.
To prove \eqref{eq:s_delta_summation_by_parts}, we split $B(x) = B^{[1]} (x) + B^{[2]} (x)$, where
\begin{equation*}
B^{[1]} (x) :=  \sum_{\substack{n \in \IZ + \b_j \\ 0 < n \leq x}} b(n)
\andd
B^{[2]} (x) :=  \sum_{\substack{n \in \IZ + \b_j \\ -x \leq n < 0}} b(n) .
\end{equation*}
By Corollary \ref{cor:dj_bound}, we have
$B^{[1]} (x), B^{[2]} (x) \ll \sqrt{x}$, so we can split the right-hand side of \eqref{eq:s_delta_summation_by_parts} as
\begin{equation}\label{eq:s_delta_summation_by_parts_two_integrals}
2 \pi \d \int_0^\infty B^{[1]}(x)  \KK_0 (2 \pi \d x)  dx
+
2 \pi \d \int_0^\infty B^{[2]}(x)  \KK_0 (2 \pi \d x)  dx.
\end{equation}
Both integrals converge due to the exponential decay of $\KK_0 (x)$ as $x \to \infty$.
Focusing on the first integral of \eqref{eq:s_delta_summation_by_parts_two_integrals}, we note that $B^{[1]} (x)$ is zero near $x=0$ and has jump discontinuities at $\b_j + \mathbb{N}_0$ (recall $0<\b_j<1$). Noting $B^{[1]} (n+\b_j-1) = B^{[1]} (n+\b_j) - b(n+\b_j)$ we rewrite this integral as
\begin{multline*}
2 \pi \d \sum_{n\geq1} B^{[1]} (n+\b_j-1) \int_{n+\b_j-1}^{n+\b_j} \KK_0 (2 \pi \d x)  dx
\\ \quad
=
\sum_{n\geq1}  \Bigg( B^{[1]} (n+\b_j-1) \KK(2 \pi \d (n+\b_j-1))
- B^{[1]} (n+\b_j)  \KK(2 \pi \d (n+\b_j))
+ b(n+\b_j) \KK(2 \pi \d (n+\b_j))  \Bigg).
\end{multline*}
The first two terms telescope to yield $B^{[1]} (\b_j) \KK(2 \pi \d \b_j)$. Together with the third term we then find
\begin{align*}
2\pi \delta \int_0^\infty B^{[1]}(x) \mathcal{K}_0 \!\lp 2\pi \delta x \rp dx = \sum_{\substack{n \in \Z + \b_j \\ n > 0}} b(n) \mathcal{K}\!\lp 2\pi |n| \delta \rp .
\end{align*}
The analogous result for $B^{[2]} (x)$ then proves the identity \eqref{eq:s_delta_summation_by_parts}.

Now note that because $\lim_{x \to 0^+} \KK (x) = 1$ and $\lim_{x \to \infty} \KK (x) = 0$, we have
\begin{equation*}
2 \pi \d \int_0^\infty \KK_0 (2 \pi \d x)  dx = 1 .
\end{equation*}
Hence by equation \eqref{eq:s_delta_summation_by_parts}
\begin{equation*}
B_{\d} - B
= 2 \pi \d \int_0^\infty (B(x)-B) \KK_0 (2 \pi \d x)  dx .
\end{equation*}
Since $B = \lim_{x \to \infty} B(x)$ for any $\e > 0$ we can choose $X_\e > 0$ such that
$|B(x) -B| < \e$ for all $x > X_\e$. Then we have (noting that $\KK_0 (x) >0$ for $x>0$)
\begin{equation*}
|B_{\d} - B| <  2 \pi \d \int_0^{X_\e} |B(x)-B|  \KK_0 (2 \pi \d x)  dx
+ 2 \pi \d \e \int_{X_\e}^\infty \KK_0 (2 \pi \d x)  dx.
\end{equation*}
Letting $\KK_{0,\max}:=\max\{\mathcal{K}_0(x): x>0\}$ (note $\KK_0 (x) \to 0$ as $x \to 0^+$ and as $x \to \infty$) and bounding
\begin{equation*}
2 \pi \d \int_{X_\e}^\infty \KK_0 (2 \pi \d x)  dx
\leq 2 \pi \d \int_0^\infty \KK_0 (2 \pi \d x)  dx = 1,
\end{equation*}
we find that
\begin{equation*}
|B_{\d} - B| < 2 \pi \d \KK_{0,\max} \int_0^{X_\e} |B(x)-B| dx + \e .
\end{equation*}
If $\d$ is sufficiently small, then the first term is smaller than $\e$, which concludes our proof.
\end{proof}

For use in Section \ref{sec:bounds_nonprincipal_parts}, we next examine how large these symmetric sums can get.
\nolisttopbreak
\begin{lem} \label{lem:1_over_n_term_bound}
For $c \in \IN$, $d \in \IZ$, $j \in \{0,1,2\}$ we have, with the implied constant independent of $d$,
\begin{equation*}
\psum_{n\in\Z+\b_j} \frac{d_j(n)e^{-\frac{2\pi i d n}{c}}}{n} \ll c.
\end{equation*}
\end{lem}
\begin{proof}
We follow equation \eqref{eq:symmetric_sum_rewriting} and start with
\begin{equation*}
\psum_{n\in\Z+\b_j} \frac{d_j(n)e^{-\frac{2\pi i d n}{c}}}{n}
=
\frac{1}{c} \sum_{r=0}^{c-1} e^{-\frac{2\pi id}{c} (r+\b_j)}
\sum_{m\ge0} \lp \frac{d_j (cm+r+\b_j)}{m + \frac{r+\b_j}{c}}
- \frac{d_j (-c(m+1)+r+\b_j)}{m + \frac{c-r-\b_j}{c}} \rp .
\end{equation*}
We split the sum over $m$ into a sum over $0 \leq m \leq c$ and a sum over $m \geq c+1$.
Applying summation by parts on partial sums of the contribution from $m \geq c+1$
and noting the bounds in \eqref{eq:dj_partial_sum_bound1} and \eqref{eq:dj_partial_sum_bound2} we rewrite this as
\begin{equation*}
\frac{1}{c} \sum_{r=0}^{c-1} e^{-\frac{2\pi i d}{c} (r+\b_j)}
\lp I_1 + I_2 + I_3 + I_4 + I_5 + I_6 + I_7 \rp,
\end{equation*}
where
\begin{align*}
I_1 &:= \sum_{m=0}^c \frac{d_j (cm+r+\b_j)}{m + \frac{r+\b_j}{c}}, \quad I_2 := - \sum_{m=0}^c \frac{d_j (-c(m+1)+r+\b_j)}{m + \frac{c-r-\b_j}{c}}, \quad I_3 := - \frac{\sum_{n=0}^c d_j (cn+r+\b_j)}{c +1 + \frac{r+\b_j}{c}}
\\
I_4 &:= \frac{\sum_{n=0}^c d_j (-c(n+1)+r+\b_j)}{c + 1 + \frac{c-r-\b_j}{c}}, \quad I_5 := \sum_{m\ge c+1} \frac{\sum_{n=0}^m d_j (cn+r+\b_j) - c \CA_{j,r,c} \lp m + \frac{r+\b_j}{c} \rp}{\lp m + \frac{r+\b_j}{c} \rp \lp m +1 + \frac{r+\b_j}{c} \rp},
\\
I_6 &:= - \sum_{m\ge c+1} \frac{\sum_{n=0}^m d_j (-c(n+1)+r+\b_j) - c \CA_{j,r,c} \lp m + \frac{c-r-\b_j}{c} \rp}{\lp m + \frac{c-r-\b_j}{c} \rp \lp m + 1 + \frac{c-r-\b_j}{c} \rp},
\\
I_7 &:= c  \CA_{j,r,c} \lp 1 - 2 \frac{r+\b_j}{c} \rp
\sum_{m\ge c+1} \frac{1}{\lp m +1 + \frac{r+\b_j}{c} \rp \lp m + 1 + \frac{c-r-\b_j}{c} \rp}
\end{align*}
with the involved series convergent.
The lemma follows if we show that $I_j=O(c)$ for $j\in\{1,\dots,7\}$ with an implied constant independent of $r$.
For $I_1$, $I_2$, this follows from Corollary~\ref{cor:dj_bound}.
For $I_3$, $I_4$, we use \eqref{eq:dj_partial_sum_bound1} and \eqref{eq:dj_partial_sum_bound2}, respectively, with $m=c$ and note that $|\CA_{j,r,c}|\ll1$.
The bounds on $I_5$, $I_6$ also follow from \eqref{eq:dj_partial_sum_bound1} and \eqref{eq:dj_partial_sum_bound2}, giving
\begin{equation*}
|I_{5}|, |I_{6}| \ll  c^{\frac{3}{2}} \sum_{m\ge c+1} m^{-\frac{3}{2}} \ll c .
\end{equation*}
Finally we bound $|I_7|\ll 1$.
\end{proof}
We now arrive at the following midway point to our desired Mordell-type representation.

\begin{prop} \label{P: Mordell Representation Step 1}
Let $- \frac dc \in \Q$ and $\t \in \H$ be such that $\t_1 > -\frac dc$. Then for $j \in \{ 0,1,2 \}$ we have
\begin{equation*}
\calU_{j,-\frac dc}(\t)
=
-\frac{1}{2\pi^2 \lp \t + \frac{d}{c} \rp} \
\psum_{n\in\Z+\b_j} \frac{d_j(n)e^{-\frac{2\pi i d n}{c}}}{n}
- \frac{1}{\pi} \sum_{n\in\Z+\b_j} d_j(n)e^{-\frac{2\pi i d n}{c}}
\KK_{\t,\frac dc}(n).
\end{equation*}
%where $\KK_{\t,\frac dc}(n)$ is defined in \eqref{eq:KK Definition}.
\end{prop}

\subsection{Mordell type representation of the modular transformation}
We next rewrite $\KK_{\t,\frac dc}(n)$. In Proposition \ref{P: Mordell Representation Step 1} we restrict to $\t_1 > -\frac dc$ because we are using the principal value of the square roots appearing in the definition of $\KK_{\t,\frac dc}(n)$. Our rewriting below naturally continues the obstruction to modularity to the entire $\mathbb{H}$ as stated in Remark \ref{rem:extension_obstruction}. Here we require the {\it sine and cosine integrals}
\begin{align*}
\Si(w) := \int_0^w \frac{\sin(t)}{t} dt \ \ \mbox{ for } w \in  \IC,
\qquad \Ci(w) := - \int_w^\infty \frac{\cos(t)}{t} dt
\ \ \mbox{ for } w \in \IC \setminus (-\infty,0].
\end{align*}
Then for $w \in \IC \setminus (-\infty,0]$, we set
\begin{align*}
f(w) := \Ci(w) \sin(w) + \cos(w)\lp \dfrac{\pi}{2} - \Si(w) \rp, \ \ \
g(w) := - \Ci(w) \cos(w) + \sin(w) \lp \dfrac{\pi}{2} - \Si(w) \rp.
\end{align*}

\begin{lem}\label{lem:modularity_obstruction_with_f_g}
For $\t \in \IH$ with $\t_1 > - \frac{d}{c}$ and $n \neq 0$ we have
\begin{align*}
\KK_{\t,\frac dc}(n) = \sgn(n) f\!\lp 2\pi |n| \!\lp \t + \frac dc \rp \rp
+ i g\!\lp 2\pi |n| \!\lp \t + \frac dc \rp \rp - \dfrac{1}{2\pi n \lp \t + \frac dc \rp}.
\end{align*}
\end{lem}
\begin{proof}
Using \eqref{eq:KK Definition} and making the change of variables $u = 2 \pi |n| t$ we have, with $w:=2\pi|n|(\t+\frac dc)$,
\begin{equation*}
\KK_{\t,\frac dc}(n) =
iw \int_0^\infty \frac{u K_0(u)}{\lp u^2 + w^2 \rp^{\frac32}} du - \sgn(n) \int_0^\infty \dfrac{u^2 K_1(u)}{\lp u^2 + w^2 \rp^{\frac32}} du,
\end{equation*}
Such integrals can be expressed in a couple of different ways (see Section 3 of \cite{LZ}). We start with 6.565.7 of \cite{GR}, which gives (for $\re (a) >0$, $\re (b) >0$, and $\re (\nu) > -1$)
\begin{equation*}%\label{eq:Bessel_integral}
\int_0^\infty x^{1+\nu} \left(x^2 + a^2\right)^\mu K_\nu (bx) dx
=
2^\nu \Gamma (\nu+1) a^{\nu + \mu+1} b^{-1-\mu} S_{\mu-\nu,\mu+\nu+1} (ab),
\end{equation*}
where the Lommel function $S_{\mu, \nu}$ is defined in 8.570.2 of \cite{GR}
(with the definition extending to $\mu \pm \nu$ a negative odd integer through its limiting value). In particular, using 8.575.1 of \cite{GR},
\begin{equation*}
w\int_0^\infty \frac{uK_0(u)}{\lp u^2 + w^2 \rp^{\frac32}} du
=
\sqrt{w} S_{- \frac{3}{2}, -\frac{1}{2}} (w)
\andd
\int_0^\infty \frac{u^2 K_1(u)}{\lp u^2 + w^2 \rp^{\frac32}} du
= \frac{1}{w} - \sqrt{w} S_{-\frac{1}{2}, \frac{1}{2}}(w),
\end{equation*}
These two Lommel functions (for which one needs the limiting value of the definition above) can be simplified using Watson's treatise of Bessel functions (see the first displayed equation after equation (2) in Subsection 10.73 of \cite{Watson}), which expresses $S_{\nu-1, \nu}$ for $\nu\not\in-\mathbb N_0$ as
\begin{equation*}%\label{eqn:S}
S_{\nu-1, \nu} (w) =
\frac{1}{2 \nu} \lb w^\nu \log (w)
- \frac{\del }{\del \mu}S_{\mu,\nu} (w) \rb_{\mu = \nu+1}.
\end{equation*}
Plugging in 8.570.2 of \cite{GR}, the lemma follows after a lengthy but straightforward calculation.
\end{proof}
Lemma \ref{lem:modularity_obstruction_with_f_g} is useful because the definition of the trigonometric integrals over the cut-plane
analytically continues $\KK_{\t, \frac dc}$ to all $\t \in\mathbb H$. We use this to find a Mordell-type representation of this term that is valid on $\mathbb H$.

\begin{lem}\label{lem:f_g_trig_integral_integral_rep}
For $w \in \mathbb{H}$ we have
\begin{equation*}%\label{eq:f_integral_rep}\label{eq:g_integral_rep}
f(w) = i\PV\int_0^\infty \frac{e^{iwt}}{t^2-1} dt + \frac{\pi}{2} e^{iw},
\qquad
g(w) = \PV\int_0^\infty \frac{te^{iwt}}{t^2-1} dt - \frac{\pi i}{2}e^{iw}.
\end{equation*}
\end{lem}
\begin{proof}
If $\re(w)>0$, then by 5.2.12 and 5.2.13 of \cite{AS} we have
\begin{equation*}
f(w) = \int_0^\infty \frac{e^{-wt}}{t^2+1} dt
\quad \mbox{and} \quad
g(w) = \int_0^\infty \frac{te^{-wt}}{t^2+1} dt.
\end{equation*}

\begin{figure}[h!]
\vspace{-0pt}
\centering
\includegraphics[scale=0.25]{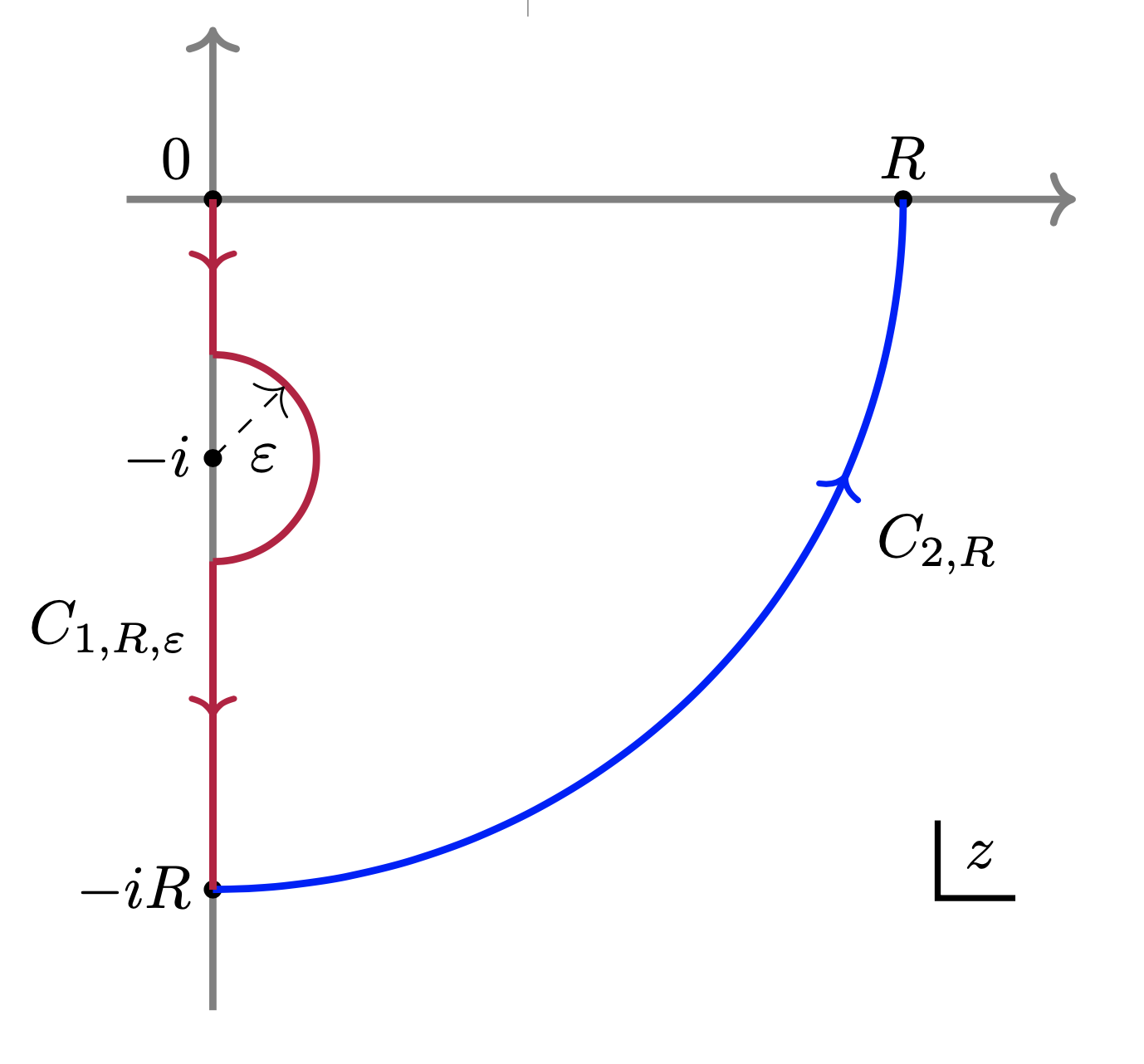}
\vspace{-15pt}
\caption{}
\label{fig:contour}
\vspace{-10pt}
\end{figure}

\noindent We first focus on $f$ and assume $\re(w), \im(w)>0$. The integrand is a meromorphic function of $t$ with poles at $t=\pm i$, we can deform the path of integration as shown in Figure \ref{fig:contour} and rewrite
\begin{align*}
f(w) = \lim_{R \to \infty} \left( \int_{C_{1,R,\e}} \dfrac{e^{-wz}}{z^2+1} dz + \int_{C_{2,R}} \dfrac{e^{-wz}}{z^2+1} dz \right),
\end{align*}
for $0< \ve < 1$. The integral over $C_{2,R}$ tends to zero as $R \to \infty$ using $\re(w), \im(w)>0$. Letting $\e \to 0^+$ for the integral over $C_{1,R,\e}$ gives the expression in the lemma with the second term coming from the simple pole at $z=-i$. Crucially, this expression is holomorphic for $\im(w)>0$ and hence coincides with $f(w)$ there. The argument for $g$ is exactly the same.\qedhere
\end{proof}

We are now ready to continue the expression from Lemma \ref{L: Second Mordell Lemma} for $\calU_{j,\varrho}(\t)$ to the entire $\H$ and to derive a Mordell-type expression for $\mathcal{U}_{j,\varrho}^\#$ as discussed in Remark \ref{rem:extension_obstruction}. More precisely, we find such an expression for (with $\t \in \IH$)
\begin{equation*}
\mathcal{I}_{\ell, -\frac dc}(\t) :=
u_\ell (\t) + \mathcal{U}_{\ell,-\frac dc}^\# (\t) .
\end{equation*}

\begin{thm}\label{prop:uj_modular_mordell_representation}
For $M = \pmat{a & b \\ c & d} \in \SL_2 (\IZ)$ with $c > 0$ and for $\t \in \IH$, we have
\begin{equation*}
u_j \left( \frac{a \t+ b}{c \t +d} \right) =
(c \t +d) \sum_{\ell=0}^2 \Psi_M (j,\ell) \mathcal{I}_{\ell, -\frac dc}(\t),
\end{equation*}
where
\begin{equation*}
\mathcal{I}_{\ell, -\frac dc}(\t)
=
\frac{1}{\pi i} \psum_{n\in\Z+\b_\ell} d_\ell(n)e^{-\frac{2\pi idn}c}
\PV\int_0^\infty \frac{e^{2 \pi i \lp \t + \frac{d}{c} \rp t}}{t-n} dt .
\end{equation*}
\end{thm}

\begin{proof}
We first plug Lemma \ref{lem:f_g_trig_integral_integral_rep} into Lemma \ref{lem:modularity_obstruction_with_f_g}
(note that $\beta_\ell \not \in \mathbb{Z}$ so that $n \neq 0$).
Then we insert the result in Proposition \ref{P: Mordell Representation Step 1}
while noting the modular transformations in Proposition~\ref{P: U_j and u_j transformation}
to find the claimed transformation as well as the identity for $\mathcal{I}_{\ell, -\frac dc}(\t)$ if $\t \in \mathbb{H}$ with $\t_1 > - \frac{d}{c}$ (note that $\sgn (c \t_1 + d) = 1$ in this case).
Therefore, the theorem statement follows if we can show that the stated expression for $\mathcal{I}_{\ell, -\frac dc}$ defines a holomorphic function on $\IH$.

The summands here are already holomorphic on $\mathbb{H}$ as follows from Lemma \ref{lem:f_g_trig_integral_integral_rep}. So our goal is to show that the involved infinite sums are convergent for any $\t \in \mathbb{H}$ and they yield a holomorphic function there.
We start our analysis by splitting $\frac{1}{t-n}=( \frac{1}{t-n} + \frac{1}{n} ) - \frac{1}{n}$ to rewrite
\begin{equation*}
\mathcal{I}_{\ell, -\frac dc}(\t)
=
\frac{1}{\pi i} \psum_{n\in\Z+\b_\ell} \frac{d_\ell(n)e^{-\frac{2\pi idn}c}}{n}
\PV\int_0^\infty
\frac{t e^{2 \pi i \lp \t + \frac{d}{c} \rp t}}{t-n} dt
-
\frac{1}{2\pi^2\left(\t + \frac{d}{c}\right)}
\psum_{n\in\Z+\b_\ell} \frac{d_\ell(n)e^{-\frac{2\pi idn}c}}{n}
\end{equation*}
with the convergence of the sums on the right-hand side proving the same for the sum in  $\mathcal{I}_{\ell, -\frac dc}$.
The second term is holomorphic for $\t \in \IH$ with Lemma \ref{lem:1_over_n_term_convergence} justifying the convergence of the involved sum. So we prove convergence of the first sum for $\t \in \IH$ and the holomorphy of the resulting function by showing that the sum over $n$ is absolutely and uniformly convergent on compact subsets of $\IH$.
We start by bounding the integral appearing in this term. If $n<0$, then we have
\begin{equation*}
\left| \int_0^\infty  \frac{te^{2 \pi i \lp \t + \frac{d}{c} \rp t}}{t-n} dt \right|
\leq \frac{1}{|n|} \int_0^\infty t e^{- 2 \pi \t_2 t} dt
= \frac{1}{4 \pi^2 \t_2^2 |n|}.
\end{equation*}
If $n>0$, then we distinguish whether
$\t_1 \geq -\frac{d}{c}$ or $\t_1 < - \frac{d}{c}$.
If $\t_1 \geq -\frac{d}{c}$, then we rotate the path of integration to $e^{\frac{\pi i}{4}} \IR^+$ while picking up the half residue from the pole at $t=n$ to write
\begin{equation*}
\PV\int_0^\infty \frac{te^{2 \pi i \lp \t + \frac{d}{c} \rp t}}{t-n} dt
=
\pi i n e^{2 \pi i n \lp \t + \frac{d}{c} \rp}
+
\int_{e^{\frac{\pi i}{4}} \IR^+}  \frac{ze^{2 \pi i \lp \t + \frac{d}{c} \rp z}}{z-n} dz .
\end{equation*}
For $z = e^{\frac{\pi i}{4}} t$ with $t \geq 0$, we then note
\begin{equation*}
\left| e^{2 \pi i \lp \t + \frac{d}{c} \rp z} \right|
=
e^{- \sqrt{2} \pi \lp \t_2 + \t_1 + \frac{d}{c} \rp t} \leq e^{- \sqrt{2} \pi \t_2 t}
\quad \mbox{and} \quad
\left| \frac{z}{z-n} \right| \leq \frac{\sqrt{2}t}{n}
\end{equation*}
to bound
\begin{equation*}
\left| \PV\int_0^\infty \frac{te^{2 \pi i \lp \t + \frac{d}{c} \rp t}}{t-n} dt \right|
\leq \pi n e^{-2 \pi n \t_2} + \frac{1}{\sqrt{2} \pi^2 \t_2^2 n} .
\end{equation*}
The same bound also holds for $\t_1 < -\frac{d}{c}$ as can be seen by similarly rotating the path of integration to $e^{-\frac{\pi i}{4}} \IR^+$. So we have
\begin{equation*}
\sum_{n\in\Z+\b_\ell} \left| \frac{d_\ell(n)e^{-\frac{2\pi idn}c}}{n}
\PV\int_0^\infty e^{2 \pi i \lp \t + \frac{d}{c} \rp t}
\frac{t}{t-n} dt \right|
\leq
\sum_{n\in\Z+\b_\ell} \frac{|d_\ell (n)|}{|n|}
\lp \pi n e^{-2 \pi \t_2 n} \d_{n>0} + \frac{1}{\t_2^2 |n|} \rp ,
\end{equation*}
which is uniformly convergent in $\t$ over compact subsets of $\mathbb{H}$ thanks to Corollary \ref{cor:dj_bound}.
\end{proof}

\section{Bounds on Non-Principal Parts}\label{sec:bounds_nonprincipal_parts}
In the Circle Method, it is important to distinguish between the principal and non-principal parts of the transformed generating functions. The principal parts, roughly speaking, correspond to negative $q$-powers and contribute to the Fourier coefficients in an exponentially growing manner. The non-principal parts, on the other hand are error terms. In our application, we have a mixed-type object and the functions $u_j$ multiply a weakly holomorphic modular form, which has exponentially growing parts towards cusps. In this section, we analyze combinations of the form $e^{- 2 \pi i d \t} \mathcal I_{\ell,\varrho}(\t)$ with $d \geq 0$ from this point of view, determining where its (continuum of) principal part lies and likewise bounding its non-principal part. In particular, looking at the expression
for $\mathcal I_{\ell,\varrho}$ given
in Theorem \ref{prop:uj_modular_mordell_representation},
for $h' \in \IZ$, $k\in\mathbb N$ with $\gcd (h',k) = 1$ and $\re (V) \geq 1$ we decompose
\begin{equation}\label{eq:calI_principal_nonprincipal_decomposition}
e^{2\pi d V} \mathcal{I}_{\ell, \frac{h'}{k}} \!\lp \frac{h'}{k} + iV \rp
=
\mathcal{I}^*_{\ell, \frac{h'}{k},d} \!\lp \frac{h'}{k} + iV \rp
+
\mathcal{I}^e_{\ell, \frac{h'}{k},d} \!\lp \frac{h'}{k} + iV \rp,
\end{equation}
assuming $d \geq 0$ and $d \not \in \IZ + \b_\ell$, where
\begin{align}
\mathcal{I}^*_{\ell, \frac{h'}{k},d} \!\lp \frac{h'}{k} + iV \rp
&:=
\frac{e^{2\pi d V}}{\pi i} \psum_{n\in\Z+\b_\ell} d_\ell(n)e^{ \frac{2 \pi ih'n}{k}}
\PV\int_0^d \frac{e^{- 2 \pi V t}}{t-n} dt,
\label{eq:calI_principal_part}
\\
\label{eq:calI_nonprincipal_part}
\mathcal{I}^e_{\ell, \frac{h'}{k},d} \!\lp \frac{h'}{k} + iV \rp
&:=
\frac{e^{2\pi d V}}{\pi i} \psum_{n\in\Z+\b_\ell} d_\ell(n)e^{\frac{2 \pi i h'n}{k}}
\PV\int_d^\infty \frac{e^{- 2 \pi V t}}{t-n} dt .
\end{align}

We next find an upper bound on the non-principal part
$\mathcal{I}^e$, where along the way we also prove the convergence of the involved sum. This also justifies the decomposition in \eqref{eq:calI_principal_nonprincipal_decomposition}.

\begin{lem}\label{lem:calI_bounds}
Let $d \geq 0$ with $d \not \in \IZ + \b_\ell$,
$k \in \IN$, $h' \in \IZ$ with $\gcd (h',k) = 1$. Then for $V \in \IC$ with $\re (V) > 0$, the sum defining $\mathcal{I}^e$ in \eqref{eq:calI_nonprincipal_part} is convergent and the resulting function is holomorphic in $V$. Moreover, for $\re (V) \geq 1$ we have, with the implied constant independent of $h'$ and $V$,
\begin{equation*}
\mathcal{I}^e_{\ell, \frac{h'}{k},d} \!\lp \frac{h'}{k} + iV \rp \ll k.
\end{equation*}
\end{lem}
\begin{proof}
As in the proof of Theorem \ref{prop:uj_modular_mordell_representation}, we decompose
$\frac{1}{t-n}=( \frac{1}{t-n} + \frac{1}{n} ) - \frac{1}{n}$ to write
\begin{equation*}
\mathcal{I}^e_{\ell, \frac{h'}{k},d} \!\lp \frac{h'}{k} + iV \rp
=
\frac{e^{2\pi d V}}{\pi i} \psum_{n\in\Z+\b_\ell} \frac{d_\ell(n)}{n} e^{\frac{2\pi ih'n}{k}}
\PV\int_d^\infty \frac{t e^{- 2 \pi V t}}{t-n} dt
-
\frac{1}{2 \pi^2 i V} \psum_{n\in\Z+\b_\ell} \frac{d_\ell(n)}{n} e^{\frac{2\pi ih'n}{k}}
\end{equation*}
with the convergence of the sums on the right-hand side proving the same for $\mathcal{I}^e$.
Our goal is to bound and prove the convergence and holomorphy of each of these two terms.
For the second term, convergence is implied by Lemma \ref{lem:1_over_n_term_convergence} and its holomorphy for $\re (V) > 0$ is immediate. By Lemma \ref{lem:1_over_n_term_bound} it is bounded as $O(k)$ for $\re (V) \geq 1$ with the implied constant independent of $h'$ and $V$.

So we focus on the first term, which making the change of variables $t \mapsto t+d$ we rewrite as
\begin{equation*}
\mathcal{I}^{e,[1]}_{\ell, \frac{h'}{k},d} \!\lp \frac{h'}{k} + iV \rp
:=
\frac{1}{\pi i} \psum_{n\in\Z+\b_\ell} \frac{d_\ell(n)}{n} e^{\frac{2\pi ih'n}{k}}
\PV\int_0^\infty \frac{(t+d) e^{- 2 \pi V t}}{t-(n-d)} dt .
\end{equation*}
The integral is a holomorphic function of $V$ for $\re (V) > 0$.
We next deduce the holomorphy of $\mathcal{I}^{e,[1]}$ by showing that the sum over $n$ is absolutely and uniformly convergent for $\re (V) \geq \d$ for any $\d > 0$. We start by bounding the integral appearing in the summands of $\mathcal{I}^{e,[1]}$. Splitting $t+d=\frac{n}{n-d} t - \frac{d}{n-d} (t-(n-d))$, we rewrite
\begin{equation*}
\PV\int_0^\infty \frac{(t+d) e^{- 2 \pi V t}}{t-(n-d)} dt
=
\frac{n}{n-d} \PV\int_0^\infty \frac{t  e^{- 2 \pi V t}}{t-(n-d)} dt
-
\frac{d}{n-d} \frac{1}{2 \pi V} .
\end{equation*}
The remaining integral is bounded in the proof of Theorem \ref{prop:uj_modular_mordell_representation} as
\begin{equation*}
\left|  \PV\int_0^\infty \frac{t  e^{- 2 \pi V t}}{t-(n-d)} dt \right|
\leq
\pi (n-d) e^{-2 \pi (n-d) \re (V)} \d_{n>d} + \frac{1}{\re (V)^2 |n-d|} .
\end{equation*}
So for $\re (V) \geq \d$ we can estimate
\begin{equation*}
\left| \mathcal{I}^{e,[1]}_{\ell, \frac{h'}{k},d} \!\lp \frac{h'}{k} + iV \rp \right|
\leq
\sum_{\substack{n\in\Z+\b_\ell \\ n >d}} |d_\ell (n)| e^{-2 \pi \d (n-d)}
+ \frac{1}{\pi \d^2} \sum_{n\in\Z+\b_\ell} \frac{|d_\ell(n)|}{|n-d|^2}
+ \frac{d}{2\pi^2 \d} \sum_{n\in\Z+\b_\ell} \frac{|d_\ell(n)|}{|n(n-d)|} .
\end{equation*}
By Corollary \ref{cor:dj_bound}, all three series are convergent, proving the absolute and uniform convergence of the series defining $\mathcal{I}^{e,[1]}$ for $\re (V) \geq \d$. This bound also shows that for $\re (V) \geq 1$ we have $\mathcal{I}^{e,[1]}_{\ell, \frac{h'}{k},d}( \frac{h'}{k} + iV) = O(1)$ as a function of $k$ with the implied constant independent of $h'$ and $V$.
\end{proof}

\section{Proof of Theorem \ref{T: Main Theorem}}\label{sec:circle_method}
In this section we apply the Circle Method to obtain the exponentially growing terms in the asymptotic expansion of the Fourier coefficients $\a_j (n)$ in (see \eqref{eq:alpha_generating_function} and \eqref{eq:U_multiplier_T_S})
\begin{equation*}
	\frac{u_j (\t)}{\eta (\t)} = \sum_{n\ge0} \a_j (n) q^{n+\DD_j},
	\where
	j \in \{0,1,2\}
	\mbox{ and }
	\DD_0 = -\frac{1}{48}, \ \DD_1 = \frac{23}{48}, \ \DD_2 := \frac{11}{12},
\end{equation*}
and in particular prove Theorem \ref{T: Main Theorem}.
We start with
\begin{equation*}
	\a_j (n) = \int_i^{i+1} \frac{u_j (\t)}{\eta (\t)} e^{-2\pi i \left(n+\DD_j\right) \t} d \t ,
\end{equation*}
with any path in $\H$ connecting $i$ to $i+1$.
Here we use Rademacher's path by following the exposition of \cite{Apostol} (as described in Subsection 3.4 of \cite{BN2}).
This gives, after the change of variables $\t = \frac{h}{k} + \frac{iZ}{k^2}$,
\begin{equation}\label{eq:alpha_j_Cauchy_result}
	\a_j (n) = i \sum_{k=1}^N k^{-2}
	\sum_{\substack{0 \leq h < k \\ \gcd (h,k) = 1}}
	\int_{Z_1}^{Z_2} \frac{u_j \lp \frac{h}{k} + \frac{iZ}{k^2} \rp}{\eta \lp \frac{h}{k} + \frac{iZ}{k^2} \rp} e^{-2\pi i \left(n+\DD_j\right) \lp \frac{h}{k} + \frac{iZ}{k^2} \rp} d Z,
\end{equation}
where $Z_j$ with $j \in \{1,2\}$ stands for $Z_j (h,k,N)$ defined by
\begin{equation*}
\hspace{-1.8cm}
	Z_1 (0,1;N) := \frac{1}{1-iN}, \qquad
	Z_2 (0,1;N) := \frac{1}{1+iN},
\end{equation*}
\begin{equation*}
	Z_1 (h,k;N) := \frac{k}{k-ik_1}, \qquad
	Z_2 (h,k;N) := \frac{k}{k+ik_2}
	\quad \mbox{for } k \geq 2,
\end{equation*}
with $\frac{h_1}{k_1} < \frac{h}{k} < \frac{h_2}{k_2}$ denoting consecutive fractions in the Farey sequence $F_N$ of order $N$. The integral runs from $Z_1$ to $Z_2$ over the right arc of the standard circle of radius $\frac{1}{2}$ and center $\frac{1}{2}$ (see Figure \ref{fig:contour_deformation}).

\begin{figure}[h!]
	\vspace{-5pt}
	\centering
	\includegraphics[scale=0.19]{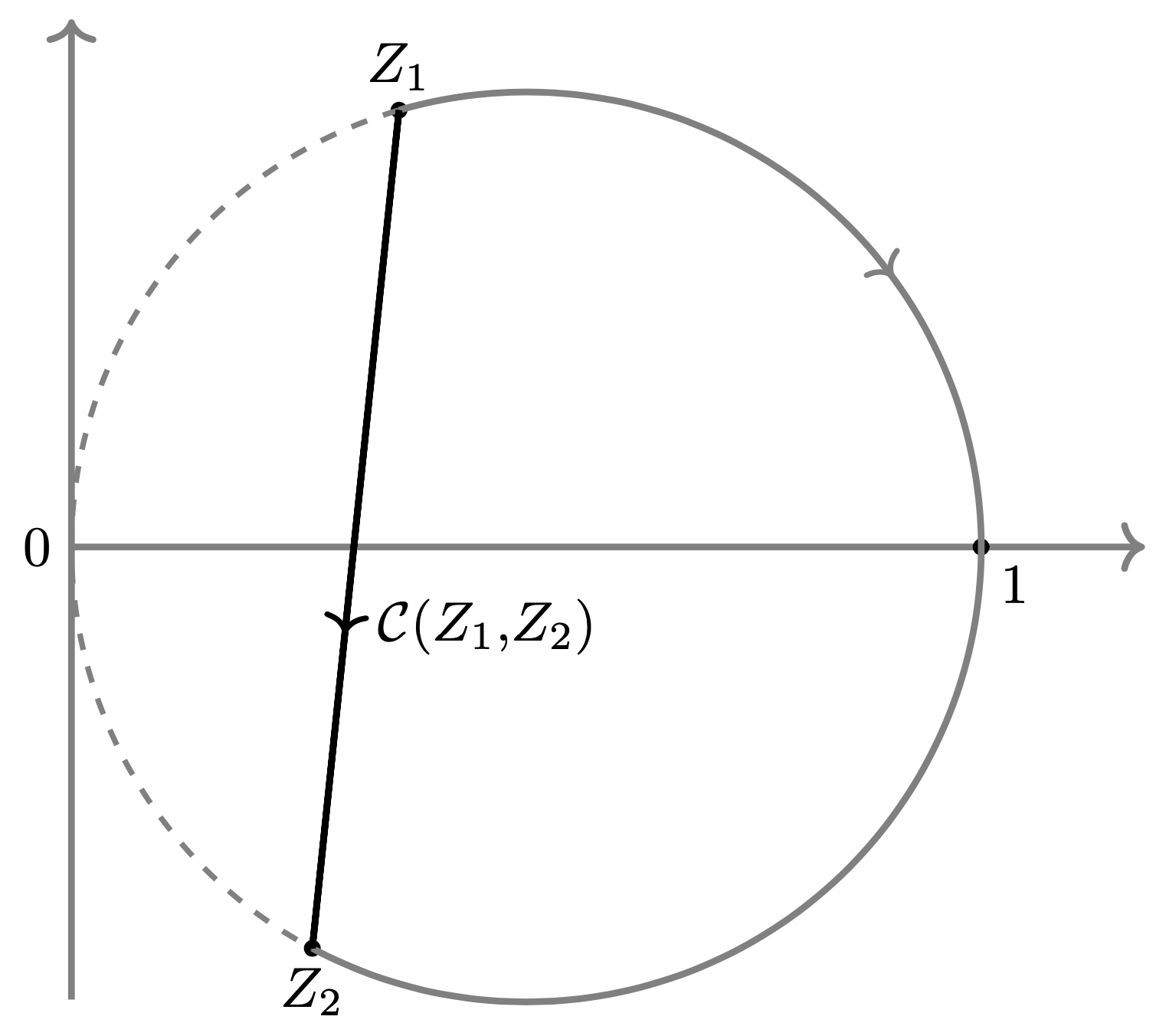}
		\vspace{-10pt}
	\caption{}
	\label{fig:contour_deformation}
	\vspace{-10pt}
\end{figure}

Now consider $M_{h,k} := \pmat{h & -\frac{hh'+1}{k} \\ k & - h'} \in \SL_2 (\IZ)$ with $0 \leq h' < k$ chosen to satisfy $h h' \equiv -1 \pmod{ k}$.
Applying Theorem \ref{prop:uj_modular_mordell_representation} with $M=M_{h,k}$ and $\t = \frac{h'}{k} + \frac{i}{Z}$ we find
\begin{equation}\label{eq:uj_modular_mordell_rep_h_k}
	u_j \!\lp  \frac{h}{k} + \frac{iZ}{k^2} \rp
	=
	\frac{ik}{Z} \sum_{\ell=0}^2 \Psi_{M_{h,k}} (j,\ell)
	\mathcal{I}_{\ell, \frac{h'}{k}} \!\lp \frac{h'}{k} + \frac{i}{Z} \rp
\end{equation}
with $\Psi_M$ defined in \eqref{eq:U_modular_transformation_multiplier}.
Similarly, the Dedekind-eta function satisfies the modular transformation
\begin{equation*}
	\eta \! \lp  \frac{h}{k} + \frac{iZ}{k^2} \rp
	= \nu_\eta(M_{h,k}) \sqrt{\frac{ik}{Z}}
	\eta \!\lp \frac{h'}{k} + \frac{i}{Z} \rp,
\end{equation*}
where $\nu_\eta$ is the $\eta$-multiplier system given for $M = \pmat{a & b \\ c & d} \in \SL_2 (\IZ)$ with $c > 0$:
\begin{equation*}
	\nu_\eta (M) := e^{\pi i \lp \frac{a+d}{12c} - \frac{1}{4} + s (-d,c) \rp}
	\quad \mbox{with }
	s\left(h',k\right) := \sum_{r=1}^{k-1} \frac{r}{k} \lp \frac{h'r}{k} - \left\lfloor \frac{h'r}{k} \right\rfloor
	- \frac{1}{2} \rp,
\end{equation*}
(see e.g.~Theorem 3.4 of \cite{Apostol}). Plugging these into \eqref{eq:alpha_j_Cauchy_result} we find
\begin{equation*}%\label{eq:alpha_j_Cauchy_result2}
	\a_j (n) = - \sum_{\ell=0}^2
	\sum_{k=1}^N k^{-\frac{3}{2}}
	\sum_{\substack{0 \leq h < k \\ \gcd (h,k) = 1}}
	\psi_{h,k} (j,\ell)
	\int_{Z_1}^{Z_2} Z^{-\frac{1}{2}}
	\frac{\mathcal{I}_{\ell, \frac{h'}{k}} \!\lp \frac{h'}{k} + \frac{i}{Z} \rp}{\eta  \!\lp \frac{h'}{k} + \frac{i}{Z} \rp}
	e^{-2\pi i \left(n+\DD_j\right) \lp \frac{h}{k} + \frac{iZ}{k^2} \rp} d Z
\end{equation*}
with
\begin{equation}\label{eq:multiplier_system_definition}
\psi_{h,k} (j,\ell) :=
e^{- \frac{\pi i}{4} } \frac{\Psi_{M_{h,k}} (j,\ell)}{\nu_\eta(M_{h,k})} .
\end{equation}
Now we split off the principal part from the term $q^{-\frac{1}{24}}$ in $\frac{1}{\eta(\t)}$ by following \eqref{eq:calI_principal_nonprincipal_decomposition} and writing
\begin{align}
	\frac{\mathcal{I}_{\ell, \frac{h'}{k}} \!\lp \frac{h'}{k} + \frac{i}{Z} \rp}{\eta  \!\lp \frac{h'}{k} + \frac{i}{Z} \rp}
	&= e^{- \frac{\pi i h'}{12k}}
	\mathcal{I}^*_{\ell, \frac{h'}{k},\frac{1}{24}} \!\lp \frac{h'}{k} +  \frac{i}{Z} \rp
	+  e^{- \frac{\pi i h'}{12k}}
	\mathcal{I}^e_{\ell, \frac{h'}{k},\frac{1}{24}} \!\lp \frac{h'}{k} +  \frac{i}{Z} \rp
	\notag\\
	&\qquad \qquad +
	\mathcal{I}_{\ell, \frac{h'}{k}} \!\lp \frac{h'}{k} + \frac{i}{Z} \rp
	\lp \frac{1}{\eta  \!\lp \frac{h'}{k} + \frac{i}{Z} \rp}
	- e^{- \frac{\pi i}{12} \lp \frac{h'}{k} + \frac{i}{Z} \rp} \rp .
	\label{eq:separating_principal_part}
\end{align}

We next show that the contribution of the second and the third term in \eqref{eq:separating_principal_part} (the non-principal part contributions) are small (i.e.,~non-exponential) for $N = \lfloor \sqrt{n} \rfloor$. Thanks to Lemma \ref{lem:calI_bounds}
each of these terms is holomorphic if $\re (Z) > 0$, so we can deform the path of integration to the chord $\mathcal C(Z_1,Z_2)$ from $Z_1$ to $Z_2$ (see Figure \ref{fig:contour_deformation}) to write their contribution to $\a_j (n)$ as
\begin{multline*}
	\b_j (n) := - \sum_{\ell=0}^2
	\sum_{k=1}^N k^{-\frac{3}{2}}
	\sum_{\substack{0 \leq h < k \\ \gcd (h,k) = 1}}
	\psi_{h,k} (j,\ell)
	\int_{\mathcal C(Z_1,Z_2)} Z^{-\frac{1}{2}}
	e^{-2\pi i \left(n+\DD_j\right) \lp \frac{h}{k} + \frac{iZ}{k^2} \rp}
	\\ \times
	\lp e^{- \frac{\pi i h'}{12k}}
	\mathcal{I}^e_{\ell, \frac{h'}{k},\frac{1}{24}} \!\lp \frac{h'}{k} +  \frac{i}{Z} \rp
	+
	\mathcal{I}_{\ell, \frac{h'}{k}} \!\lp \frac{h'}{k} + \frac{i}{Z} \rp
	\lp \frac{1}{\eta  \!\lp \frac{h'}{k} + \frac{i}{Z} \rp}
	- e^{- \frac{\pi i}{12} \lp \frac{h'}{k} + \frac{i}{Z} \rp} \rp \rp
	d Z.
\end{multline*}
For any point in the right half-plane contained within the disk bounded by the circle with center and radius $\frac12$, and in particular on $\mathcal C(Z_1,Z_2)$, we have $\re (\frac{1}{Z}) \geq 1$. So we can use Lemma \ref{lem:calI_bounds} (once with $d=0$ and once with $d=\frac{1}{24}$) to bound
\begin{equation*}
	e^{- \frac{\pi i h'}{12k}}
	\mathcal{I}^e_{\ell, \frac{h'}{k},\frac{1}{24}} \!\lp \frac{h'}{k} +  \frac{i}{Z} \rp
	+
	\mathcal{I}_{\ell, \frac{h'}{k}} \!\lp \frac{h'}{k} + \frac{i}{Z} \rp
	\lp \frac{1}{\eta  \!\lp \frac{h'}{k} + \frac{i}{Z} \rp}
	- e^{- \frac{\pi i}{12} \lp \frac{h'}{k} + \frac{i}{Z} \rp} \rp
	\ll k
\end{equation*}
with the implied constant independent of $h'$ and $Z$.
Next we recall (e.g.~from Theorem 5.3 and 5.5 of \cite{Apostol}), that if $\frac{h_1}{k_1} < \frac{h}{k} < \frac{h_2}{k_2}$ are consecutive fractions in the Farey sequence $F_N$, then we have
\begin{equation*}
	\max\{k,k_j\} \leq N \leq k+k_j-1
\end{equation*}
and hence
\begin{equation}\label{eq:k_kj_sq_bound}
\frac{N^2}{2} < k^2 + k_j^2 \leq 2 N^2 .
\end{equation}
Consequently, for any point $Z$ on $\mathcal C(Z_1,Z_2)$ (including the case $k=1$) we find
\begin{equation*}%\label{eq:chord_real_part_bound}
	\frac{k^2}{2N^2} \leq \re (Z) < \frac{2k^2}{N^2}
	\andd
	 \frac{k^2}{2N^2} \leq |Z| < \frac{\sqrt{2}k}{N} .
\end{equation*}
In particular, these results imply that on $\mathcal C(Z_1,Z_2)$ we have, for $n \geq 1$
\begin{equation*}
\left| Z^{-\frac{1}{2}}  e^{-2\pi i\left(n+\DD_j\right) \lp \frac{h}{k} + \frac{iZ}{k^2} \rp} \right| \ll \frac{N}{k}
\andd
\mathrm{length}\!\lp \mathcal{C}(Z_1,Z_2) \rp < \frac{2\sqrt{2} k}{N}.
\end{equation*}
Therefore, for $n\in\mathbb{N}$ we find the bound
\begin{equation*}
|\b_j (n)| \ll n^{\frac{3}{4}}.
\end{equation*}
This leads to
\begin{align*}%\label{eq:alpha_j_mid_step}
	\a_j (n) &= - \sum_{\ell=0}^2
	\sum_{k=1}^{N} k^{-\frac{3}{2}}
	\sum_{\substack{0 \leq h < k \\ \gcd (h,k) = 1}}
	\psi_{h,k} (j,\ell) e^{-\frac{2 \pi i}{24k} \lp h' + 24\left(n+\DD_j\right) h \rp}
	\\ &\hspace{5cm} \times
	\int_{Z_1}^{Z_2} \frac{e^{2\pi \left(n+\DD_j\right) \frac{Z}{k^2}}}{\sqrt{Z}}
	\mathcal{I}^*_{\ell, \frac{h'}{k},\frac{1}{24}} \!\lp \frac{h'}{k} +  \frac{i}{Z} \rp
	d Z
	+
	O\!\lp n^{\frac{3}{4}}\rp  .
\end{align*}

We next recall the definition of the principal term $\mathcal{I}^*$ from equation \eqref{eq:calI_principal_part} and focus on the integral on the second line:
\begin{equation*}
\frac{1}{\pi i}
	\int_{Z_1}^{Z_2} \frac{e^{2\pi \left(n+\DD_j\right) \frac{Z}{k^2}}}{\sqrt{Z}}
	\psum_{m\in\Z+\b_\ell} d_\ell(m) e^{ \frac{2 \pi ih'm}{k}}
	\PV\int_0^{\frac{1}{24}} \frac{e^{- \frac{2 \pi}{Z} \left(t-\frac{1}{24}\right)}}{t-m} dt dZ .
\end{equation*}
Once more writing $\frac{1}{t-m}=\frac{t}{(t-m)m} - \frac{1}{m}$ as in the proof of Theorem \ref{prop:uj_modular_mordell_representation}, we can interchange the sum over $m$ and the integral over $t$ to write
\begin{equation*}
	\psum_{m\in\Z+\b_\ell} d_\ell(m) e^{ \frac{2 \pi ih'm}{k}}
	\PV\int_0^{\frac{1}{24}} \frac{e^{- \frac{2 \pi}{Z} \left(t-\frac{1}{24}\right)}}{t-m} dt
	=
	\PV\int_0^{\frac{1}{24}} \Phi_{\ell, \frac{h'}{k}} (t)
	e^{- \frac{2 \pi}{Z} \left(t-\frac{1}{24}\right)}  dt,
\end{equation*}
where
\begin{equation}\label{eq:integral_kernel_definition}
	\Phi_{\ell, \frac{h'}{k}} (t) := \psum_{m \in\Z+\b_\ell}
	\frac{d_\ell(m) e^{ \frac{2 \pi ih'm}{k}} }{t-m}
\end{equation}
is a meromorphic function of $t$ on the entire complex plane with poles in $\IZ+\b_\ell$.
Here note that $d_\ell (m)$ is defined through \eqref{eq:Uj_definition} and \eqref{eq:Uj_general_form}. In particular, the only pole of $\Phi_{\ell, \frac{h'}{k}} (t)$ that lies in the interval $[0,\frac{1}{24}]$ that we are integrating over is the one at $t=\frac{1}{48}$ for the component from $\ell =0$.

Now viewing the principal value integral as an average of two integrals from $0$ to $\frac{1}{24}$, one going above the potential pole and one below, we obtain an iterated integral of a continuous function over two compact intervals. We can then interchange the integral over $t$ with the integral over $Z$ and note that the integral over $Z$ yields an entire function of $t$ to conclude
\begin{equation*}
\int_{Z_1}^{Z_2} \frac{e^{2\pi \left(n+\DD_j\right) \frac{Z}{k^2}}}{\sqrt{Z}}
	\mathcal{I}^*_{\ell, \frac{h'}{k},\frac{1}{24}} \!\lp \frac{h'}{k} +  \frac{i}{Z} \rp
	d Z
=
\frac{1}{\pi i}
	\PV\int_0^{\frac{1}{24}} \Phi_{\ell, \frac{h'}{k}} (t)
	\int_{Z_1}^{Z_2}
	\frac{e^{2\pi \lp \frac{n+\DD_j}{k^2} Z + \lp \frac{1}{24} - t \rp \frac{1}{Z} \rp}}{\sqrt{Z}}
	dZdt .
\end{equation*}

We next focus on the integral in $Z$ and make the change of variables $Z \mapsto \frac{1}{Z}$ to rewrite it as
\begin{equation*}
\int_{Z_1}^{Z_2}
	\frac{e^{2\pi \lp \frac{n+\DD_j}{k^2} Z + \lp \frac{1}{24} - t \rp \frac{1}{Z} \rp}}{\sqrt{Z}} dZ
=
- \int_{1-i\frac{k_1}{k}}^{1+i\frac{k_2}{k}} Z^{-\frac{3}{2}}
	e^{2\pi \lp \frac{n+\DD_j}{k^2} \frac{1}{Z} + \lp \frac{1}{24} - t \rp Z \rp} dZ .
\end{equation*}
Then we decompose the integral on the right-hand side as
\begin{equation*}
\int_{\mathcal C} Z^{-\frac{3}{2}}
	e^{2\pi \lp \frac{n+\DD_j}{k^2} \frac{1}{Z} + \lp \frac{1}{24} - t \rp Z \rp} dZ
	-
	\sum_{r=1}^4
	\int_{\mathcal C_r} Z^{-\frac{3}{2}}
	e^{2\pi \lp \frac{n+\DD_j}{k^2} \frac{1}{Z} + \lp \frac{1}{24} - t \rp Z \rp} dZ ,
\end{equation*}
where the paths of integration are as shown in Figure \ref{fig:bessel_contour}. Here integrals over $\mathcal C_3$ and $\mathcal C_4$ (as well as that over $\mathcal C_5$ as mentioned above) define entire functions of $t$ whereas integrals over $\mathcal C_1$ and $\mathcal C_2$, and hence that over $\mathcal C$, define holomorphic functions of $t$ for $\re (t) < \frac{1}{24}$ extending to a continuous function for $\re (t) \leq \frac{1}{24}$.
Here we note that on $\mathcal C_1, \mathcal C_2$ we have $| e^{2\pi(\frac{1}{24} - t) Z} | \leq 1$ for $\re (t) \leq  \frac{1}{24}$.

\begin{figure}[h!]
	\vspace{-5pt}
	\centering
	\includegraphics[scale=0.21]{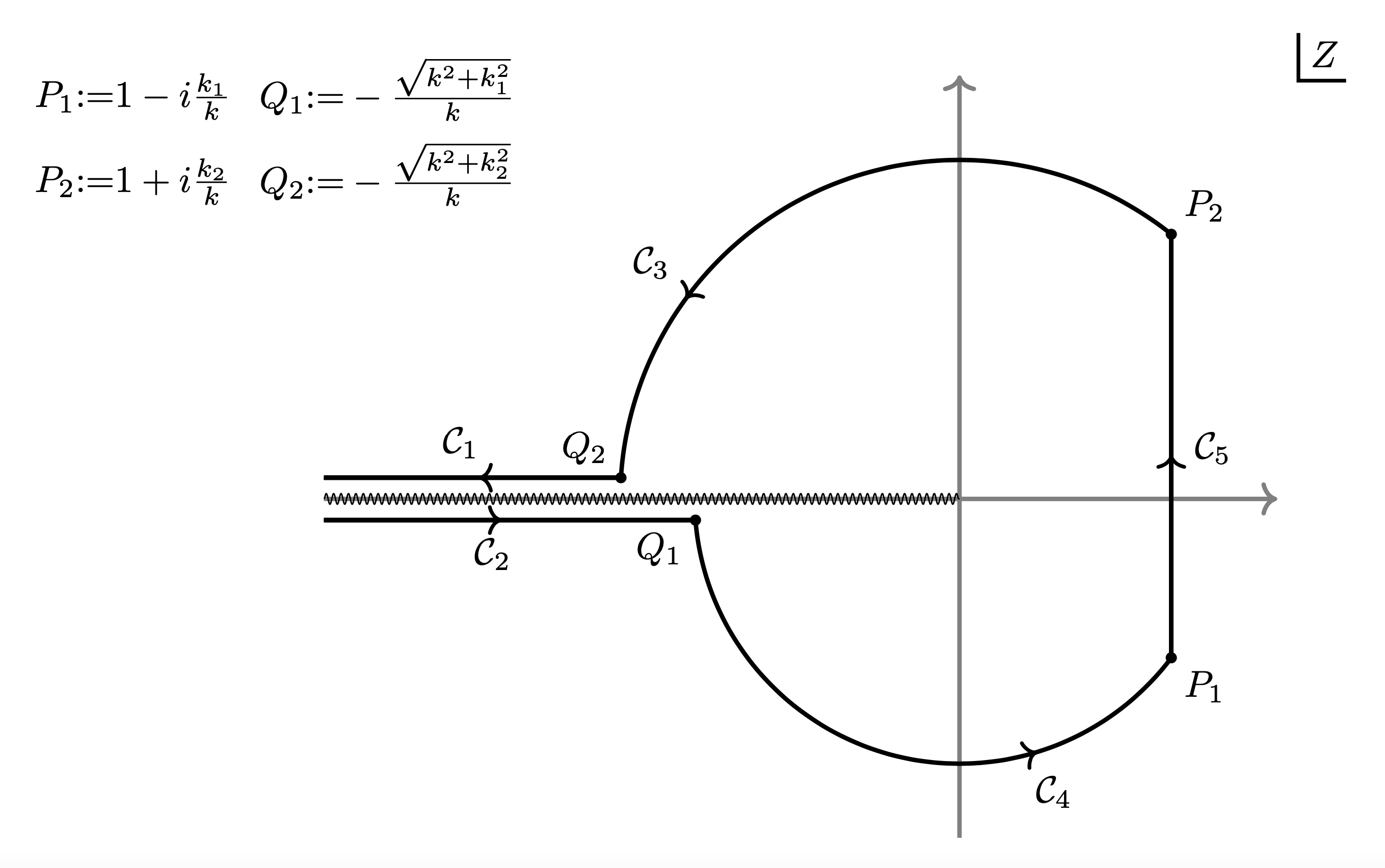}
		\vspace{-15pt}
	\caption{
		$\mathcal C_3$ and $\mathcal C_4$ are circular arcs centered at zero and $\mathcal C := \mathcal C_2 + \mathcal C_4 + \mathcal C_5 + \mathcal C_3 + \mathcal C_1$.
	}
	\label{fig:bessel_contour}
	\vspace{-5pt}
\end{figure}

\noindent We also decompose
\begin{equation}\label{phi_kernel_pole_removed_decomposition}
	\Phi_{\ell, \frac{h'}{k}} (t)
	=:
	\Phi^*_{\ell, \frac{h'}{k}} (t)
	+ \frac{d_0 \!\lp \frac{1}{48} \rp}{t-\frac{1}{48}} e^{\frac{2\pi i h'}{48k}} \d_{\ell,0},
\end{equation}
in order to separate the only pole of $\Phi_{\ell, \frac{h'}{k}} (t)$ that may lie in $[0,\frac{1}{24}]$ (see the comments after \eqref{eq:integral_kernel_definition}). We then define
\begin{align}\label{eq:gamma_j_r_definition}
	\g_j^{[r]} (n) &:= \frac{1}{\pi i} \sum_{\ell=0}^2
	\sum_{k=1}^{N} k^{-\frac{3}{2}}
	\sum_{\substack{0 \leq h < k \\ \gcd (h,k) = 1}}
	\!\!\!\!  \l_{h,k} (n, j,\ell)
	\int_0^{\frac{1}{24}} \Phi^*_{\ell, \frac{h'}{k}} (t)
	\int_{\mathcal C_r} Z^{-\frac{3}{2}}
	e^{2\pi \lp \frac{n+\DD_j}{k^2} \frac{1}{Z} + \lp \frac{1}{24} - t \rp Z \rp}
	dZ dt,
	\\ \label{eq:delta_j_r_definition}
	\d_j^{[r]} (n) &:= \frac{1}{\pi i}
	\sum_{k=1}^{N} k^{-\frac{3}{2}}
	\sum_{\substack{0 \leq h < k \\ \gcd (h,k) = 1}}
	\!\!\!\!  \l^*_{h,k} (n,j,0)
	\PV\int_0^{\frac{1}{24}} \frac{1}{t-\frac{1}{48}}
	\int_{\mathcal C_r} Z^{-\frac{3}{2}}
	e^{2\pi \lp \frac{n+\DD_j}{k^2} \frac{1}{Z} + \lp \frac{1}{24} - t \rp Z \rp}
	dZ dt,
\end{align}
where
\begin{equation*}
\l_{h,k} (n,j,\ell) :=
\psi_{h,k} (j,\ell) e^{-\frac{2 \pi i}{24k} \lp h' + 24\left(n+\DD_j\right) h \rp}
\andd
\l^*_{h,k} (n, j,\ell) := \l_{h,k} (n, j,\ell) e^{\frac{2 \pi i h'}{48k}}.
\end{equation*}
This leads to
\begin{multline}\label{eq:alpha_j_C_Cr_decomposition}
	\a_j (n) = \frac{1}{\pi i} \sum_{\ell=0}^2
	\sum_{k=1}^{N} k^{-\frac{3}{2}}
	\sum_{\substack{0 \leq h < k \\ \gcd (h,k) = 1}}
	\psi_{h,k} (j,\ell) e^{-\frac{2 \pi i}{24k} \lp h' + 24\left(n+\DD_j\right) h \rp}
	\PV\int_0^{\frac{1}{24}} \Phi_{\ell, \frac{h'}{k}} (t)\\ \times
	\int_{\mathcal C} Z^{-\frac{3}{2}}
	e^{2\pi \lp \frac{n+\DD_j}{k^2} \frac{1}{Z} + \lp \frac{1}{24} - t \rp Z \rp}
	dZ dt
	- \sum_{r=1}^4 \g_j^{[r]} (n)
	- d_0 \!\lp \frac{1}{48} \rp \sum_{r=1}^4 \d_j^{[r]} (n)
	+
	O\!\lp n^{\frac{3}{4}}\rp  .
\end{multline}

We next bound $\g_j^{[r]} (n)$ and $\d_j^{[r]} (n)$ for $r \in \{1,2,3,4\}$.
We start with estimating $\g_j^{[r]}$ for $r\in\{3,4\}$.

\begin{lem}\label{lem:gamma_j_3_bound}
For $j \in \{0,1,2\}$ and $n\in\mathbb{N}$ we have
\begin{equation*}
\g_j^{[3]} (n), \  \g_j^{[4]} (n) \ll n^{\frac{3}{4}}.
\end{equation*}
\end{lem}
\begin{proof}
As $Z$ traverses the circular arc $\mathcal C_3$, $\frac{1}{Z}$ moves on a circular arc with center at zero from the angle $-\arctan (\frac{k_2}{k})$ to $-\pi$. The largest value of $\re (\frac{1}{Z})$ is obtained at the initial point and on $\mathcal C_3$
\begin{equation*}
	\re \!\lp \frac{1}{Z} \rp \leq \re \!\lp \frac{1}{1+i\frac{k_2}{k}} \rp
	= \frac{k^2}{k^2+k_2^2} < \frac{2k^2}{N^2},
\end{equation*}
where the last inequality follows from \eqref{eq:k_kj_sq_bound}.
We also have $\re (Z) \leq 1$ on $\mathcal C_3$ so that (for $0 \leq t \leq \frac{1}{24}$)
\begin{equation*}
	\left| e^{2\pi \lp \frac{n+\DD_j}{k^2} \frac{1}{Z} + \lp \frac{1}{24} - t \rp Z \rp} \right|
	\ll 1 .
\end{equation*}
Moreover, by \eqref{eq:k_kj_sq_bound} we find that on $\mathcal C_3$ we have
\begin{equation*}
|Z|^{-1} \ll \frac{k}{N}
\andd
\mathrm{length} (\mathcal C_3) \ll \frac{N}{k} .
\end{equation*}
To bound $\Phi^*_{\ell, \frac{h'}{k}} (t)$ over $0 \leq t \leq \frac{1}{24}$, on the other hand, we start by writing
\begin{equation*}
\Phi^*_{\ell, \frac{h'}{k}} (t)
=
\sum_{\substack{m\in\Z+\b_\ell \\ m \neq \frac{1}{48}}}
	\frac{t d_\ell(m)}{(t-m)m}  e^{ \frac{2 \pi ih'm}{k}}
	-
	\psum_{m\in\Z+\b_\ell}
	\frac{d_\ell(m)}{m} e^{ \frac{2 \pi ih'm}{k}}
	+ 48 d_0 \!\lp \frac{1}{48} \rp \! \d_{\ell,0} e^{\frac{2 \pi i h'}{48k}} .
\end{equation*}
By Corollary \ref{cor:dj_bound}, the first term is $\ll 1$ uniformly in $t$ and $h'$. The same bound holds for the third term.
Finally, we use Lemma \ref{lem:1_over_n_term_bound} to bound the second term and find
\begin{equation}\label{eq:PhiStar_bound}
\Phi^*_{\ell, \frac{h'}{k}} (t) \ll k
\end{equation}
uniformly in $t$ and $h'$. So gathering our results together with $| \l_{h,k} (n, j,\ell) | \leq 1$ we find
\begin{equation*}
	\left|\g_j^{[3]} (n)\right| \ll \sum_{k=1}^{N} k^{-\frac{3}{2}}  \sum_{0 \leq h < k}
	k \lp \frac{k}{N} \rp^{\frac{3}{2}} \mathrm{length} (\mathcal C_3)  \ll
	\frac{1}{\sqrt{N}} \sum_{k=1}^N k \ll N^{\frac{3}{2}} ,
\end{equation*}
which yields the claim from $N = \lfloor \sqrt{n} \rfloor$.
Exactly the same proof yields the result for $\g_j^{[4]}$.
\end{proof}

We next bound $\d_j^{[\ell]}$ for $\ell\in\{3,4\}$.

\begin{lem}\label{lem:delta_j_3_bound}
For $j \in \{0,1,2\}$ and $n\in\mathbb{N}$ we have
\begin{equation*}
\d_j^{[3]} (n) , \  \d_j^{[4]} (n) \ll \log (n) n^{\frac{1}{4}}.
\end{equation*}
\end{lem}
\begin{proof}
Noting that the integral over $\mathcal C_3$ gives an entire function of $t$ as we state before Figure \ref{fig:bessel_contour} and that we have
$\PV\int_0^{\frac{1}{24}} \frac{1}{t-\frac{1}{48}} dt = 0$, we start by rewriting
$\d_j^{[3]}$ defined in \eqref{eq:delta_j_r_definition} as
\begin{equation*}
\d_j^{[3]} (n) = \frac{1}{\pi i}
	\sum_{k=1}^{N} k^{-\frac{3}{2}}
	\sum_{\substack{0 \leq h < k \\ \gcd (h,k) = 1}}
	\l^*_{h,k} (n,j,0)
	\int_0^{\frac{1}{24}}
	\int_{\mathcal C_3} Z^{-\frac{3}{2}}
	e^{2\pi \frac{n+\DD_j}{k^2} \frac{1}{Z}}
	\frac{e^{2\pi \lp \frac{1}{24} - t \rp Z} - e^{\frac{\pi Z}{24}}}{t-\frac{1}{48}}
	dZ dt .
\end{equation*}
Then using the mean value inequality on the complex-valued function $t \mapsto e^{2\pi (\frac{1}{24} - t) Z}$, we find
\begin{equation*}
\left| \frac{e^{2\pi \lp \frac{1}{24} - t \rp Z}  - e^{\frac{\pi Z}{24}} }{t-\frac{1}{48}}
\right|
\leq 2 \pi |Z| e^{2\pi \lp \frac{1}{24} - \xi \rp \re (Z)}
\end{equation*}
for some $\xi$ between $\frac{1}{48}$ and $t$ and including the limit $t \to \frac{1}{48}$.
Since $t \in [0,\frac{1}{24}]$ (implying $\frac{1}{24} - \xi < \frac{1}{24}$) and since $\re (Z) \leq 1$ on $\mathcal C_3$, we have
\begin{equation*}
\left| \frac{e^{2\pi \lp \frac{1}{24} - t \rp Z}  - e^{\frac{2\pi Z}{48}} }{t-\frac{1}{48}}
\right|
\leq  2 \pi e^{\frac{\pi}{12}} |Z| .
\end{equation*}
Moreover, as in the proof of Lemma \ref{lem:gamma_j_3_bound},
on $\mathcal C_3$ we have the bounds
\begin{equation}\label{eq:delta_j_bound_C3_bounds}
	\left| e^{2\pi \frac{n+\DD_j}{k^2} \frac{1}{Z}} \right| \ll 1, \qquad
	|Z|^{-\frac{1}{2}} \ll \sqrt{\frac{k}{N}}, \andd
	\mathrm{length} (\mathcal C_3) \ll \frac{N}{k}  .
\end{equation}
So together with $|\l^*_{h,k} (n,j,0)| \leq 1$ we find
\begin{equation*}
	\left| \d_j^{[3]} (n) \right| \ll
	\sum_{k=1}^{N} k^{-\frac{3}{2}} \sum_{0 \leq h < k}
	\sqrt{\frac{N}{k}}
	\ll \sqrt{N} \log (N)
\end{equation*}
which implies the lemma with $N = \lfloor \sqrt{n} \rfloor$.
Since the bound $\re (Z) \leq 1$ as well as those in \eqref{eq:delta_j_bound_C3_bounds} also hold for $\mathcal C_4$, the exactly the same proof yields the result for $\d_j^{[4]} (n)$ as well.
\end{proof}

We continue with $\g_j^{[\ell]}$ for $\ell\in\{1,2\}$.

\begin{lem}\label{lem:gamma_j_1_bound}
For $j \in \{0,1,2\}$ and $n\in\mathbb{N}$ we have
\begin{equation*}
\g_j^{[1]} (n), \  \g_j^{[2]} (n) \ll n^{\frac{3}{4}}.
\end{equation*}
\end{lem}
\begin{proof}
We first consider $\g_j^{[1]} (n)$ and let $Z = -x$ (with $Z^{-\frac{3}{2}} = e^{-\frac{3\pi i}{2}} x^{-\frac{3}{2}}$ above the cut) to write
\begin{equation*}
\g_j^{[1]} (n) = -\tfrac{1}{\pi} \sum_{\ell=0}^2
	\sum_{k=1}^{N} k^{-\frac{3}{2}} \!
	\sum_{\substack{0 \leq h < k \\ \gcd (h,k) = 1}} \!
	\l_{h,k} (n, j,\ell) \int_0^{\frac{1}{24}} \Phi^*_{\ell, \frac{h'}{k}} (t)
	\int_{\frac{1}{k} \sqrt{k^2+k_2^2}}^\infty x^{-\frac{3}{2}}
	e^{-2\pi \lp \frac{n+\DD_j}{k^2} \frac{1}{x} + \lp \frac{1}{24} - t \rp x \rp}
	dx dt .
\end{equation*}
Using the bound on $\Phi^*_{\ell, \frac{h'}{k}} (t)$ in \eqref{eq:PhiStar_bound} and that $e^{-2\pi ( \frac{n+\DD_j}{k^2} \frac{1}{x} + (\frac{1}{24} - t) x )}  \leq 1$ for $0 \leq t \leq \frac{1}{24}$, we get
\begin{equation*}
	\left| \int_0^{\frac{1}{24}} \Phi^*_{\ell, \frac{h'}{k}} (t)
	\int_{\frac{1}{k} \sqrt{k^2+k_2^2}}^\infty x^{-\frac{3}{2}}
	e^{-2\pi \lp \frac{n+\DD_j}{k^2} \frac{1}{x} + \lp \frac{1}{24} - t \rp x \rp}
	dx dt \right|
	\ll
	\frac{k^{\frac{3}{2}}}{\lp k^2+k_2^2 \rp^{\frac{1}{4}}}
	\ll
	\frac{k^{\frac{3}{2}}}{\sqrt{N}}
\end{equation*}
with the last estimate following from \eqref{eq:k_kj_sq_bound}. Inserting this in the expression for $\g_j^{[1]} (n)$ together with $| \l_{h,k} (n, j,\ell) | \leq 1$ then yields the lemma.
The proof for $\g_j^{[2]}$ is exactly the same.
\end{proof}

Our final task is to bound $\d_j^{[\ell]}$ for $\ell\in\{1,2\}$.

\begin{lem}\label{lem:delta_j_1_bound}
For $j \in \{0,1,2\}$ and $n\in\mathbb{N}$ we have
\begin{equation*}
\d_j^{[1]} (n), \  \d_j^{[2]} (n) \ll n^{\frac{1}{4}}.
\end{equation*}
\end{lem}
\begin{proof}
We next consider $\d_j^{[1]}$
and again use that $\PV\int_0^{\frac{1}{24}} \frac{1}{t-\frac{1}{48}} dt = 0$ to write (with $Z=-x$)
\begin{multline*}
	\d_j^{[1]} (n) = - \frac{1}{\pi}
	\sum_{k=1}^{N} k^{-\frac{3}{2}} \!\!\!
	\sum_{\substack{0 \leq h < k \\ \gcd (h,k) = 1}} \!\!\!
	\l^*_{h,k} (n,j,0)
	\int_0^{\frac{1}{24}}
	\int_{\frac{1}{k} \sqrt{k^2+k_2^2}}^\infty x^{-\frac{3}{2}}
	e^{- 2\pi \frac{n+\DD_j}{k^2} \frac{1}{x}}
	\frac{e^{- 2\pi \lp \frac{1}{24} - t \rp x} - e^{- \frac{\pi x}{24}} }{t-\frac{1}{48}}
	dx dt .
\end{multline*}
Now we use the fact that $e^{- 2\pi \frac{n+\DD_j}{k^2} \frac{1}{x}} \leq 1$ and  that
\begin{equation*}
	\left| \frac{e^{- 2\pi \lp \frac{1}{24} - t \rp x} - e^{- \frac{\pi x}{24}} }{t-\frac{1}{48}} \right|
	= 2 \pi x e^{- 2\pi \lp \frac{1}{24} - \xi \rp x}
\end{equation*}
for some $\xi$ between $\frac{1}{48}$ and $t$.
For $0 \leq t \leq \frac{1}{24}$, we bound the right-hand side above by
$2 \pi x e^{- \pi \lsp \frac{1}{24}  -t \rsp x}$.
Therefore, with $|\l^*_{h,k} (n,j,0)| \leq 1$ we find
\begin{equation*}
	\left| \d_j^{[1]} (n) \right| \leq 2 \sum_{k=1}^{N} k^{-\frac{3}{2}}
	\sum_{0 \leq h < k}
	\int_0^{\frac{1}{24}}
	\int_{\frac{1}{k} \sqrt{k^2+k_2^2}}^\infty
	\frac{e^{- \pi \lp \frac{1}{24}-t \rp x}}{\sqrt{x}}
	dx dt .
\end{equation*}
Making the change of variables $x \mapsto \frac{x}{\frac{1}{24}-t}$ and extending the integration range in $x$ to $\R^+$ gives
\begin{equation*}
\left| \d_j^{[1]} (n) \right| \leq
2 \sum_{k=1}^{N} \frac{1}{\sqrt{k}}
	\int_0^{\frac{1}{24}} \frac{dt}{\sqrt{\frac{1}{24}-t}}
	\int_0^\infty \frac{e^{- \pi x}}{\sqrt{x}} dx
	\ll \sqrt{N} .
\end{equation*}
The lemma follows from $N = \lfloor \sqrt{n} \rfloor$. The same proof also yields the result for $\d_j^{[2]}$.
\end{proof}

Our final goal is to go back to equation \eqref{eq:alpha_j_C_Cr_decomposition} and evaluate the integral over $Z$ for $0 < t < \frac{1}{24}$, which continuously extends to $0 \leq t \leq \frac{1}{24}$ as we state before Figure \ref{fig:bessel_contour}.
We start by recalling (e.g.~from equations 8.406.1 and 8.412.2 of \cite{GR}) that for $\re(w) > 0$ we have
\begin{equation*}
	I_\nu (w) = \frac{\lp \frac{w}{2} \rp^\nu}{2 \pi i} \int_{\mathcal D}
	z^{-\nu-1} e^{z + \frac{w^2}{4z}} dz ,
\end{equation*}
where $\mathcal D$ is Hankel's contour that starts at $-\infty$ from below the real line, circles around the negative real axis, and then goes back to $-\infty$ from above the real line.
Then we can evaluate
\begin{equation*}
\int_{\mathcal C} Z^{-\frac{3}{2}}
	e^{2\pi \lp \frac{n+\DD_j}{k^2} \frac{1}{Z} + \lp \frac{1}{24} - t \rp Z \rp} dZ
	=
	2\pi i \sqrt{k} \lp \frac{\frac{1}{24} - t}{n+\DD_j} \rp^{\frac{1}{4}}
	I_{\frac{1}{2}} \!\lp \frac{4 \pi}{k} \sqrt{\left(n+\DD_j\right) \lp \frac{1}{24} - t \rp} \rp  .
\end{equation*}
Inserting this in \eqref{eq:alpha_j_C_Cr_decomposition} while recalling Lemmas
\ref{lem:gamma_j_3_bound}, \ref{lem:delta_j_3_bound}, \ref{lem:gamma_j_1_bound}, and
\ref{lem:delta_j_1_bound} we find
\begin{multline}\label{eq:alpha_j_asymptotic_expansion_up_to_power}
\a_j (n) = \frac{2}{\left(n+\DD_j\right)^{\frac{1}{4}}} \sum_{\ell=0}^2
\sum_{k=1}^{\left\lfloor \sqrt{n} \right\rfloor} \frac{1}{k}
\sum_{\substack{0 \leq h < k \\ \gcd (h,k) = 1}}
\psi_{h,k} (j,\ell)  e^{-2 \pi i \lp \frac{h'}{24k} + \left(n+\DD_j\right) \frac{h}{k} \rp}
\\
\times \PV\int_0^{\frac{1}{24}} \Phi_{\ell, \frac{h'}{k}} (t)
\lp \frac{1}{24} - t \rp^{\frac{1}{4}}
I_{\frac{1}{2}} \!\lp \frac{4 \pi}{k} \sqrt{\left(n+\DD_j\right) \lp \frac{1}{24} - t \rp} \rp dt
+
O\!\lp n^{\frac{3}{4}}\rp .
\end{multline}
This completes the proof of the final result, Theorem \ref{T: Main Theorem} (including the case $j=2$).

\section{Proof of Corollary \ref{cor:leading_exponential_first_few_example}}\label{sec:leading_exponential}
We next extract explicit expressions for the main exponential term and prove Corollary \ref{cor:leading_exponential_first_few_example}.
We start by bounding the contributions aside from those for $k=1$ and near $t=0$.

\begin{lem}\label{lem:leading_exponential_k1_t0}
For $n\in\mathbb{N}$ and $j \in \{0,1,2\}$ we have
\begin{equation*}
\a_j (n) = \frac{1}{\pi\sqrt{2\left(n+\DD_j\right)}} \sum_{\ell=0}^2 \Psi_S (j,\ell)
\int_0^{\frac{1}{96}} \Phi_{\ell,0} (t)   e^{4\pi \sqrt{\left(n+\DD_j\right)\left(\frac{1}{24}-t\right)}} dt
+ O\!\lp e^{\pi\sqrt{\frac{n+\DD_j}2}} \rp .
\end{equation*}
\end{lem}
\begin{proof}
We first recall that $I_{\frac{1}{2}} (x) = \sqrt{\frac{2}{\pi x}} \sinh (x)$ to rewrite \eqref{eq:alpha_j_asymptotic_expansion_up_to_power} as
\begin{multline}\label{eq:alpha_j_asymptotic_expansion_up_to_power2}
\a_j (n) = \frac{\sqrt{2}}{\pi\sqrt{n + \DD_j}} \sum_{\ell=0}^2
\sum_{k=1}^{\left\lfloor \sqrt{n} \right\rfloor} \frac{1}{\sqrt{k}}
\sum_{\substack{0 \leq h < k \\ \gcd (h,k) = 1}}
\psi_{h,k} (j,\ell)  e^{-\frac{2 \pi i}k \lp \frac{h'}{24} + \left(n+\DD_j\right) h \rp}
\\
\times \PV\int_0^{\frac{1}{24}} \Phi_{\ell, \frac{h'}{k}} (t)
  \sinh \!\lp \frac{4 \pi}{k} \sqrt{\left(n+\DD_j\right) \lp \frac{1}{24} - t \rp} \rp dt
+
O\!\lp n^{\frac{3}{4}}\rp .
\end{multline}
Since the only pole of $\Phi_{\ell, \frac{h'}{k}} (t)$ lying in $[0,\frac{1}{24}]$ is the one at $t=\frac{1}{48}$ with $\ell =0$ (see the comments after \eqref{eq:integral_kernel_definition}), we recall \eqref{phi_kernel_pole_removed_decomposition} and decompose the integral on the second line of \eqref{eq:alpha_j_asymptotic_expansion_up_to_power2} as
\begin{equation*}
\frac{1}{2} (J_1+J_2-J_3) + \frac{1}{2}   d_0 \!\lp \frac{1}{48} \rp
e^{\frac{2\pi i h'}{48k}} \d_{\ell,0}
(J_4 + J_5 - J_6),
\end{equation*}
where
\begin{align*}
J_1 &:= \int_0^{\frac{1}{96}} \Phi_{\ell,\frac{h'}{k}} (t)
  e^{\frac{4\pi}{k} \sqrt{\left(n+\DD_j\right)\left(\frac{1}{24}-t\right)}} dt,
&
J_2 &:= \int_{\frac{1}{96}}^{\frac{1}{24}} \Phi^*_{\ell,\frac{h'}{k}} (t)
e^{\frac{4\pi}{k} \sqrt{\left(n+\DD_j\right)\left(\frac{1}{24}-t\right)}} dt,
\\
J_3 &:= \int_0^{\frac{1}{24}} \Phi^*_{\ell,\frac{h'}{k}} (t)
e^{-\frac{4\pi}{k} \sqrt{\left(n+\DD_j\right)\left(\frac{1}{24}-t\right)}} dt,
&
J_4 &:=
\PV\int_{\frac{1}{96}}^{\frac1{32}} \frac{1}{t-\frac{1}{48}}
e^{\frac{4\pi}{k} \sqrt{\left(n+\DD_j\right)\left(\frac{1}{24}-t\right)}} dt,
\\
J_5 &:=
\int_{\frac1{32}}^{\frac{1}{24}} \frac{1}{t-\frac{1}{48}}
e^{\frac{4\pi}{k} \sqrt{\left(n+\DD_j\right)\left(\frac{1}{24}-t\right)}} dt,
&
J_6 &:=
\PV\int_{0}^{\frac{1}{24}} \frac{1}{t-\frac{1}{48}}
  e^{-\frac{4\pi}{k} \sqrt{\left(n+\DD_j\right)\left(\frac{1}{24}-t\right)}} dt .
\end{align*}
Next we bound $J_1, \ldots, J_6$ in terms of $n,k,h'$.
By \eqref{eq:PhiStar_bound}, uniformly in $h'$ and $t \in [0,\frac{1}{96}]$, we have $\Phi_{\ell, \frac{h'}{k}} (t) \ll k$ and hence
\begin{equation*}
J_1 \ll k   e^{\frac\pi k\sqrt{\frac23\left(n+\DD_j\right)}}.
\end{equation*}
Using $\Phi^*_{\ell, \frac{h'}{k}} (t) \ll k$ we also bound
\begin{equation*}
J_2 \ll k   e^{\frac\pi k \sqrt{\frac{n+\DD_j}2}}
\mbox{ and }
J_3 \ll k .
\end{equation*}
For $J_4$, we first let $t \mapsto \frac{1}{24}-t$ in one half of the integral to remove the pole at $t=\frac{1}{48}$ and write
\begin{equation*}
J_4 = \frac{1}{2} \int_{\frac{1}{96}}^{\frac1{32}}
\frac{e^{\frac{4\pi}{k} \sqrt{\left(n+\DD_j\right)\left(\frac{1}{24}-t\right)}} - e^{\frac{4\pi}{k} \sqrt{\left(n+\DD_j\right)t}} }{t-\frac{1}{48}}   dt .
\end{equation*}
By the mean value theorem we have
\begin{equation*}
\left| \frac{e^{\frac{4\pi}{k} \sqrt{\left(n+\DD_j\right)\left(\frac{1}{24}-t\right)}} - e^{\frac{4\pi}{k} \sqrt{\left(n+\DD_j\right)t}} }{2\lp t-\frac{1}{48} \rp} \right|
=
\frac{2\pi}{k} \sqrt{\frac{n+\DD_j}{\xi}}   e^{\frac{4\pi}{k} \sqrt{\left(n+\DD_j\right)\xi}}
\end{equation*}
for some $\xi$ between $t$ and $\frac{1}{24}-t$. Since $t \in [\frac{1}{96}, \frac1{32}]$, it follows that
\begin{equation*}
J_4 \ll \frac{\sqrt{n+\DD_j}}{k}   e^{\frac\pi k \sqrt{\frac{n+\DD_j}2}}.
\end{equation*}
For $J_6$ we employ the same argument as $J_4$ and use the mean value theorem while noting that $\xi^{-\frac{1}{2}} \leq t^{-\frac{1}{2}} + (\frac{1}{24}-t)^{-\frac{1}{2}}$ if $\xi$ is between $t$ and $\frac{1}{24}-t$ and
$t \in (0, \frac{1}{24})$. This then yields
\begin{equation*}
J_6 \ll \frac{\sqrt{n+\DD_j}}{k}.
\end{equation*}
Finally, we bound $J_5 \ll e^{\frac\pi k \sqrt{\frac{n+\DD_j}6}}$.
With these bounds at hand, we find that $J_2,\ldots,J_6$ for $k \geq 1$ as well $J_1$ for $k \geq 2$
together contribute to  $\a_j (n)$ in \eqref{eq:alpha_j_asymptotic_expansion_up_to_power2} as
$\ll e^{\pi\sqrt{\frac{n+\DD_j}2}}$. The lemma statement then follows from the contribution $J_1$ for $k=1$ while noting that $\psi_{0,1} (j,\ell) = \Psi_S (j,\ell)$.
\end{proof}

\noindent We are now ready to explore the main exponential term in the asymptotic expansion
\eqref{eq:alpha_j_asymptotic_expansion_up_to_power}.

\begin{prop}\label{prop:alpha_n_main_exonential_term}
For $n,N\in\mathbb{N}$ and $j \in \{0,1,2\}$ we have
\begin{equation*}
\a_j (n) = \frac{e^{\pi\sqrt{\frac23\left(n+\DD_j\right)}}}{n+\DD_j}
\lp \sum_{r=0}^{N-1} \frac{a_r}{\left(n+\DD_j\right)^\frac r2}
+ O_N\!\lp n^{-\frac{N}{2}}\rp \rp ,
\end{equation*}
where
\begin{equation*}
a_r :=  \frac{4 \sqrt{3}}{\left(4 \pi \sqrt{6}\right)^{r+2}}
\sum_{\ell =0}^2 \Psi_S (j, \ell) \left[\frac{d^r}{d u^r} \!  \Big( (1-12u) \Phi_{\ell,0} (u(1-6u)) \Big)
\right]_{u=0}.
\end{equation*}
\end{prop}
\begin{proof}
By Lemma \ref{lem:leading_exponential_k1_t0} we have
\begin{equation*}
\a_j (n) = \frac{e^{\pi\sqrt{\frac23\left(n+\DD_j\right)}}}{\pi\sqrt{2\left(n+\DD_j\right)}}
\sum_{\ell=0}^2 \Psi_S (j,\ell)
\int_0^{\frac{1}{96}} \Phi_{\ell,0} (t)   e^{-\pi\sqrt{\frac23\left(n+\DD_j\right)} \lp 1 - \sqrt{1-24t} \rp} dt
+ O\!\lp e^{\pi\sqrt{\frac12\left(n+\DD_j\right)}} \rp .
\end{equation*}
We rewrite the integral here by making the change of variables $u:= \frac{1-\sqrt{1-24t}}{12}$ and get
\begin{equation*}
\int_0^{\frac{1}{96}} \Phi_{\ell,0} (t)   e^{-\pi\sqrt{\frac23\left(n+\DD_j\right)} \lp 1 - \sqrt{1-24t} \rp} dt
=
\int_0^{\frac{2 - \sqrt{3}}{24}} (1-12u)   \Phi_{\ell,0} (u(1-6u))
e^{-4 \pi \sqrt{6 \left(n+\DD_j\right)}   u} du.
\end{equation*}
Defining $f_\ell (u) := (1-12u)\Phi_{\ell,0} (u(1-6u))$, we next look at its Taylor
expansion for $u \in [0,\frac{2 - \sqrt{3}}{24}]$ (where $f_\ell$ is smooth since $\Phi_{\ell,0}$ is smooth on the integration range) as
\begin{equation*}
f_\ell (u) = \sum_{r=0}^{N-1} \frac{f_\ell^{(r)} (0)}{r!}    u^r +
\frac{f_\ell^{(N)} (\xi)}{N!}    u^N
\quad \mbox{for some } 0 \leq \xi \leq u .
\end{equation*}
Since $f_\ell^{(N)} (u)$ is continuous in $[0,\frac{2 - \sqrt{3}}{24}]$, we find $C_N > 0$ such that $|f_\ell^{(N)} (u)| \leq C_N$ there. Thus
\begin{multline*}
\left| \int_0^{\frac{2 - \sqrt{3}}{24}} f_\ell (u)    e^{-4 \pi \sqrt{6 \left(n+\DD_j\right)}   u} du
-
\sum_{r=0}^{N-1} \frac{f_\ell^{(r)} (0)}{r!}
\int_0^{\frac{2 - \sqrt{3}}{24}} u^r    e^{-4 \pi \sqrt{6 \left(n+\DD_j\right)}   u} du
\right|
\\
\leq \frac{C_N}{N!}
\int_0^{\frac{2 - \sqrt{3}}{24}} u^N    e^{-4 \pi \sqrt{6 \left(n+\DD_j\right)}   u} du
\ll_N n^{-\frac{N+1}{2}} .
\end{multline*}
Finally, we conclude the proof by noting
\begin{equation*}
\frac{1}{r!} \int_0^{\frac{2 - \sqrt{3}}{24}} u^r    e^{-4 \pi \sqrt{6 \left(n+\DD_j\right)}   u} du
=
\lp 4 \pi \sqrt{6 \left(n+\DD_j\right)} \rp^{-(r+1)}
+ O_N\! \lp e^{-\pi\sqrt{\frac16\left(n+\DD_j\right)}\lp 2-\sqrt3\rp}\rp.\qedhere
\end{equation*}
\end{proof}

Proposition \ref{prop:alpha_n_main_exonential_term} shows that the
Taylor coefficients of $\Phi_{\ell,0} (t)$ at $t=0$ determine the main exponential term in the asymptotic expansion of $\a_j (n)$. For convenience, we express the first few of the expansion coefficients $a_r$ defined there in terms of those Taylor coefficients:
\begin{align*}
a_0 &= \frac{1}{8 \sqrt{3}   \pi^2 } \sum_{\ell =0}^2 \Psi_S (j, \ell)    \Phi_{\ell,0} (0), \quad a_1 = \frac{1}{96 \sqrt{2}  \pi^3 } \sum_{\ell =0}^2 \Psi_S (j, \ell) \lp \Phi^{(1)}_{\ell,0} (0) - 12 \Phi_{\ell,0} (0) \rp, \\
a_2 &= \frac{1}{768 \sqrt{3}  \pi^4 } \sum_{\ell =0}^2 \Psi_S (j, \ell) \lp \Phi^{(2)}_{\ell,0} (0) - 36 \Phi^{(1)}_{\ell,0} (0) \rp, \\
a_3 &= \frac{1}{9216 \sqrt{2}   \pi^5} \sum_{\ell =0}^2 \Psi_S (j, \ell) \lp \Phi^{(3)}_{\ell,0} (0) - 72 \Phi^{(2)}_{\ell,0} (0) +432 \Phi^{(1)}_{\ell,0} (0) \rp .
\end{align*}
These Taylor coefficients are computed in Appendix \ref{app:taylor_integral_kernel} and using the values of $\sum_{\ell=0}^2 \Psi_{S} (j,\ell)   \Phi^{(r)}_{\ell, 0} (0)$ reported in Table \ref{tab:kernel_taylor_coefficients} we find the first few terms reported in Corollary \ref{cor:leading_exponential_first_few_example}, which concludes its proof.

\section{Conclusion}\label{sec:conclusion}
In this paper, we find asymptotic expressions for certain integer partitions whose generating function feature a false-indefinite theta function multiplied with a weakly holomorphic modular form. Our asymptotic expression includes all the exponentially growing components of the asymptotic behavior. This is as much as one can hope for the overall weight $\frac{1}{2}$. Our methods however are quite general and if the overall weight is smaller than $-1$, then the expressions and bounds we prove immediately lead to Hardy--Ramanujan--Rademacher type exact formulas. These bounds are definitely not optimal, and improving them would also allow one to find exact formulas for larger weights. It would be interesting to develop these arguments and search for applications where such exact formulas are viable. Another direction would be to investigate similar asymptotic results if the false-indefinite theta function in question is not related to Maass forms but instead related to genuine mock Maass forms. Finally, it would be compelling to find a combinatorial interpretation for the component $j=2$. More generally it would be very interesting if one could relate the $\mathrm{SL}_2\lp \Z \rp$-orbits observed on the modular side to a combinatorial relation.

\appendix \section{Estimates on the Fourier Coefficients of False-indefinite Theta Functions}
\label{app:false_indef_estimates}
As seen in Section \ref{sec:parity_partition_false_indefinite}, evaluating the obstruction to modularity for false-indefinite theta functions involves detailed estimates due to the presence of conditional convergence. In this appendix, we provide bounds involving the Fourier coefficients of false-indefinite theta functions. Instead of aiming for optimal bounds, our goal is to provide easily applicable geometry of numbers type arguments. More concretely, we consider the Fourier coefficients $d_j (n)$ of the Maass forms $U_j$ obtained by combining the  mock Maass theta functions $F_{\bm{\mu}}$ as in \eqref{eq:Uj_definition}.\footnote{See equation \eqref{eq:Uj_Fourier_decomposition} for the general form of the Fourier expansion. Also recall that $d_j (n)$ for $n>0$ are the Fourier coefficients of the corresponding $q$-series $u_j$ given in equation \eqref{eq:u0_u1_false_indefinite_description}.
}
Here recall the assumption that the only solution to $Q(\bm{n}) = 0$ with $\bm{n} \in \mathbb{Z} + \bm{\mu}$ and $\bm{\mu} \in A^{-1} \IZ^2$ comes from $\bm{n} = \bm{0}$ with $\bm{\mu} \in \IZ^2$, which does hold for the particular case considered in Section \ref{sec:parity_partition_false_indefinite}. Since $\IZ^2\cap \mathcal{S}_j^\pm = \emptyset$ for $j \in \{0,1,2\}$, the constant term does not appear in the Fourier expansion of $U_j$
given in \eqref{eq:Uj_general_form} and by \eqref{eq:Uj_definition}
\begin{equation*}%\label{eq:dk_decomposition_theta_coefs}
d_j (n) =
\frac{1}{2} \lp \textstyle\sum_{\bm{\mu} \in \mathcal{S}^+_j}   a_{\bm{\mu}} (n)
-  \textstyle\sum_{\bm{\mu} \in \mathcal{S}^-_j} a_{\bm{\mu}} (n) \rp,
\end{equation*}
where for $\bm{\mu} \in \mathcal{S}^\pm_j$ we have
(recall $\b_0 = \frac{1}{48}$, $\b_1 = \frac{25}{48}$,
$\b_2 = \frac{23}{24}$ and
the Fourier expansion \eqref{eq:Fmu_definition} of $F_{\bm{\mu}}$)
\begin{align*}
\sum_{\substack{n \in \IZ+\b_j \\ n > 0}} a_{\bm{\mu}} (n)  q^n &:= \frac{1}{2} \sum_{\bm{n} \in \IZ^2 +\bm{\mu}} \Big( 1 + \sgn (2n_1+n_2)  \sgn (2n_1-n_2) \Big) q^{12n_1^2-2n_2^2},\\
\sum_{\substack{n \in \IZ+\b_j \\ n < 0}} a_{\bm{\mu}} (n)  q^{|n|} &:= \frac{1}{2} \sum_{\bm{n} \in \IZ^2 +\bm{\mu}} \Big( 1 - \sgn (3n_1+n_2) \sgn (3n_1-n_2) \Big) q^{2n_2^2-12n_1^2}.
\end{align*}

Since we require the behavior of $d_j$ in different residue classes modulo $c \in \mathbb{N}$, for $\bm{\mu} \in \mathcal{S}^\pm_j$ with $j \in \{0,1,2\}$ and
$r_1,r_2 \in \{0,1,\ldots,c-1\}$ we define
\begin{align}
\sum_{\substack{n \in \IZ+\b_j \\ n > 0}} a_{c,\bm{\mu},\bm{r}} (n) q^n &:= \frac{1}{2} \sum_{\bm{n} \in \IZ^2 +\frac{\bm{\mu}+\bm{r}}{c}} \Big( 1 + \sgn (2n_1+n_2)  \sgn (2n_1-n_2) \Big) q^{c^2\left(12n_1^2-2n_2^2\right)}, \notag \\
\sum_{\substack{n \in \IZ+\b_j \\ n < 0}} a_{c,\bm{\mu},\bm{r}} (n) q^{|n|} &:= \frac{1}{2} \sum_{\bm{n} \in \IZ^2 +\frac{\bm{\mu}+\bm{r}}{c}} \Big( 1 - \sgn (3n_1+n_2) \sgn (3n_1-n_2) \Big) q^{c^2\left(2n_2^2-12n_1^2\right)},
\label{eq:a_cmur_definition}
\end{align}
so that $a_{\bm{\mu}} (n)  = \sum_{0\leq r_1,r_2 \leq c-1} a_{c,\bm{\mu},\bm{r}} (n)$
for all $\bm{\mu} \in \mathcal{S}^\pm_j$ and $n \in \IZ+\b_j$
and $a_{c,\bm{\mu},\bm{r}} (n)  =  0$ unless $n \equiv Q(\bm{\mu}+\bm{r}) \pmod{c} .$
We begin our work by bounding partial sums of $a_{c,\bm{\mu},\bm{r}}$.

\begin{lem}\label{lem:a_mu_partial_sum_bound}
Let $c \in \IN$, $r_1,r_2 \in \{0,1,\ldots,c-1\}$, and $X\in\R^+$. Then for $\bm{\mu} \in \mathcal{S}^\pm_j$ with $j \in \{0,1,2\}$
\begin{equation*}
\sum_{\substack{n \in \IZ+Q(\bm{\mu}) \\ 0<n\leq X}} a_{c,\bm{\mu},\bm{r}} (n)
=
\CA \frac{X}{c^2} + O\!\lp \max\! \left\{1, \frac{\sqrt{X}}{c} \right\} \rp,
\sum_{\substack{n \in \IZ+Q(\bm{\mu}) \\ -X \leq n < 0}} a_{c,\bm{\mu},\bm{r}} (n)
=
\CA \frac{X}{c^2} + O\!\lp \max\! \left\{1, \frac{\sqrt{X}}{c}\right\} \rp,
\end{equation*}
where $\CA := \frac{\log(\sqrt{2}+\sqrt{3})}{\sqrt{6}} = 0.46794065\ldots$ and with the implied constants independent of $\bm{r}$ and $\bm{\mu}$.
\end{lem}
\begin{proof}
We start our analysis with
\begin{equation*}
D^-_{c, \bm{\mu}, \bm{r}} (X) :=
\sum_{\substack{n \in \IZ+Q(\bm{\mu}) \\ -X \leq n <0}} a_{c,\bm{\mu},\bm{r}} (n) .
\end{equation*}
By equation \eqref{eq:a_cmur_definition}, $D^-_{c, \bm{\mu}, \bm{r}} (X)$ counts the number of points in $\IZ^2 + \frac{\bm{\mu}+\bm{r}}{c}$ within the regions
$R_1^-(\frac{\sqrt{X}}{c})$ and $R_2^-(\frac{\sqrt{X}}{c})$, where we define (see Figure \ref{fig:bound_regions2})
\begin{equation*}
R_1^- (R) := \left\{ \bm{x} \in \mathbb{R}^2 :
-\frac{x_2}{3} \leq x_1 \leq \frac{x_2}{3} \mbox{ and }
2x_2^2-12x_1^2 \leq R^2 \right\},
\end{equation*}
\begin{equation*}
R_2^- (R) := \left\{ \bm{x} \in \mathbb{R}^2 :
\frac{x_2}{3} \leq x_1 \leq -\frac{x_2}{3} \mbox{ and }
2x_2^2-12x_1^2 \leq R^2 \right\}
\end{equation*}
and with the points on the lines $x_2 = \pm 3 x_1$ counted with multiplicity $\frac{1}{2}$.

\begin{figure}[h!]
\vspace{-5pt}
\centering
\includegraphics[scale=0.19]{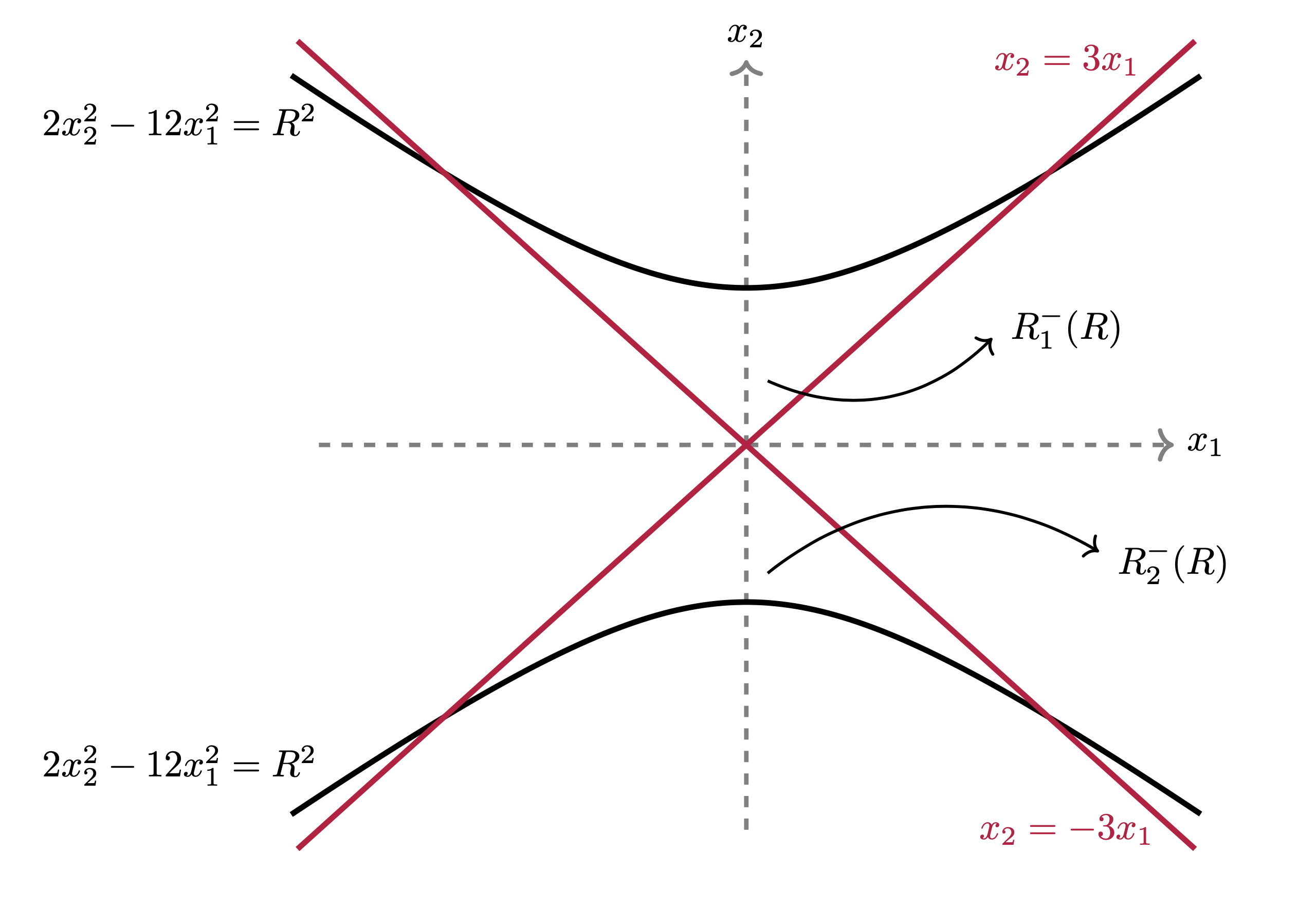}
\vspace{-5pt}
\caption{}
\label{fig:bound_regions2}
\vspace{-5pt}
\end{figure}

We first note that the contribution of the boundary points (and accordingly the change of multiplicity for such points) is negligible, since the number of such points is bounded as
\begin{multline*}
\left| \left\{ \bm{n} \in \IZ^2 + \frac{\bm{\mu}+\bm{r}}{c} :
n_2 = \pm 3 n_1 \mbox{ and } - \frac{\sqrt{X}}{\sqrt{6}c} \leq n_1 \leq  \frac{\sqrt{X}}{\sqrt{6}c} \right\} \right|
\\
\leq 2
\left| \left\{ n_1 \in \IZ + \frac{\mu_1+r_1}{c} :
- \frac{\sqrt{X}}{\sqrt{6}c} \leq n_1 \leq  \frac{\sqrt{X}}{\sqrt{6}c} \right\} \right|
\leq 2 \lp \frac{2 \sqrt{X}}{\sqrt{6}c} + 1 \rp \ll
\max\left\{ 1, \frac{\sqrt{X}}{c} \right\} .
\end{multline*}
Therefore, we have
\begin{equation*}
D^-_{c, \bm{\mu}, \bm{r}} (X)
=
D^{-}_{1,c, \bm{\mu}, \bm{r}} (X)
+
D^{-}_{2,c, \bm{\mu}, \bm{r}} (X)
+
O\!\lp \max\left\{ 1, \frac{\sqrt{X}}{c} \right\} \rp ,
\end{equation*}
where
\begin{equation*}
D^{-}_{j,c, \bm{\mu}, \bm{r}} (X)
:=
\left| \left\{ \bm{n} \in \IZ^2 + \frac{\bm{\mu}+\bm{r}}{c} :
\bm{n} \in R_j^- \!\lp \frac{\sqrt{X}}{c} \rp \right\} \right|
\quad \mbox{for } j \in \{1,2\}.
\end{equation*}
Now we can geometrically bound the number of points in $R_j^- \lsp \frac{\sqrt{X}}{c} \rsp$ as (see e.g.~the discussion of Landau on page 186 and volume 2 of \cite{Lan})
\begin{equation*}
\left| D^{-}_{j,c, \bm{\mu}, \bm{r}} (X) -
\mathrm{area} \!\lp  R_j^- \!\lp \frac{\sqrt{X}}{c}  \rp \rp \right|
\leq 4 \mathrm{length} \lp \del R_j^- \!\lp \frac{\sqrt{X}}{c} \rp \rp +4,
\end{equation*}
where $\del$ denotes the boundary of the given region. The lemma statement for this case then follows from
$\mathrm{area} \!\lp  R_1^- (1) \rp + \mathrm{area} \!\lp  R_2^- (1) \rp = \CA$.

\begin{figure}[h!]
\vspace{-10pt}
\centering
\includegraphics[scale=0.20]{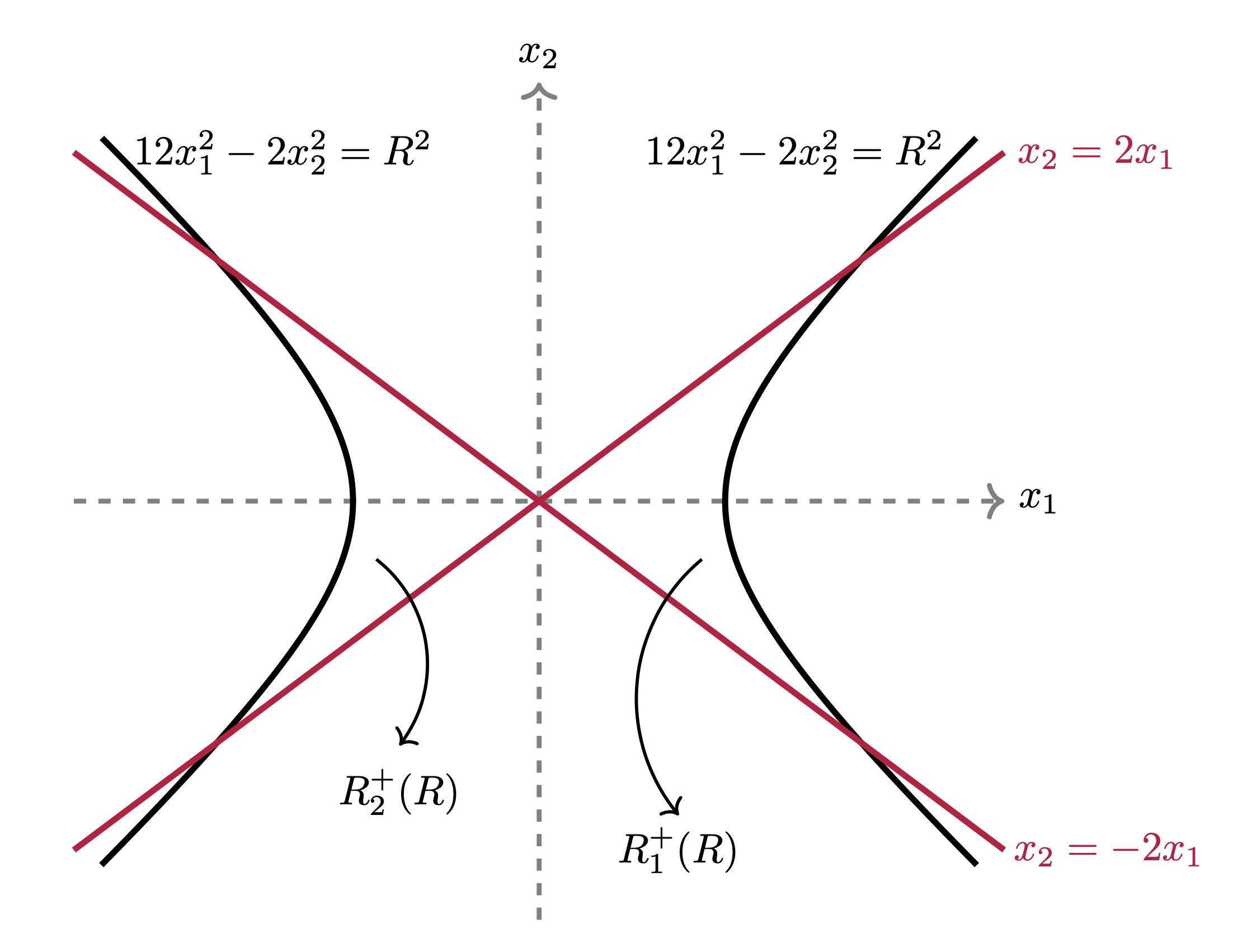}
\vspace{-5pt}
\caption{}
\label{fig:bound_regions}
\vspace{-5pt}
\end{figure}

The result for
\begin{equation*}
D^+_{c, \bm{\mu}, \bm{r}} (X) :=
\sum_{\substack{n \in \IZ+Q(\bm{\mu}) \\ 0<n\leq X}} a_{c,\bm{\mu},\bm{r}} (n)
\end{equation*}
follows from the same argument by counting the number of $\IZ^2 + \frac{\bm{\mu}+\bm{r}}{c}$ points within (see Figure \ref{fig:bound_regions})
\begin{align*}
R_1^+ (R) &:= \left\{ \bm{x} \in \mathbb{R}^2 :
-2x_1 \leq x_2 \leq 2x_1 \mbox{ and }
12x_1^2-2x_2^2 \leq R^2 \right\}, \\
R_2^+ (R) &:= \left\{ \bm{x} \in \mathbb{R}^2 :
2x_1 \leq x_2 \leq -2x_1 \mbox{ and }
12x_1^2-2x_2^2 \leq R^2 \right\}
\end{align*}
and noting that $\mathrm{area} \!\lp  R_1^+ (1) \rp + \mathrm{area} \!\lp  R_2^+ (1) \rp = \CA$.
\end{proof}

Lemma \ref{lem:a_mu_partial_sum_bound} immediately leads to the following two corollaries.
\begin{cor}\label{cor:dj_partial_sum_bound}
Let $c \in \IN$, $r \in \{0,1,\ldots,c-1\}$, $j \in \{0,1,2\}$, and $X \in \mathbb{R}^+$. Then we have
\begin{equation*}
\sum_{\substack{n \in \IZ+\b_j \\ 0<n\leq X \\ n \equiv r + \b_j \pmod{c}}} \!\!\!\!\!\! d_j (n)
\!=\!
\CA_{j,r,c} X + O \!\lp \max\! \left\{c^2, c \sqrt{X} \right\} \rp\!, \sum_{\substack{n \in \IZ+\b_j \\ -X \leq n < 0 \\ n \equiv r + \b_j \pmod{c}}} \!\!\!\!\!\! d_j (n)
\!=\!
\CA_{j,r,c} X + O \!\lp \max\! \left\{c^2, c \sqrt{X} \right\} \rp\!,
\end{equation*}
where the implied constants in the error terms independent of $j$ and $r$ and where
\begin{equation*}
\CA_{j,r,c} := \frac{\CA}{2c^2}
\lp \sum_{\bm{\mu} \in \mathcal{S}^+_j}  -  \sum_{\bm{\mu} \in \mathcal{S}^-_j} \rp
\sum_{\substack{0 \leq r_1,r_2 \leq c-1 \\ Q(\bm{\mu}+\bm{r}) \equiv r+\b_j \pmod{c}}}
1.
\end{equation*}
\end{cor}

\begin{cor}\label{cor:dj_bound}
For $j \in \{0,1,2\}$ and $n \in \IZ + \b_j$, we have $d_j (n) \ll \sqrt{|n|}$.
\end{cor}

\section{Taylor Coefficients of the Integral Kernel}
\label{app:taylor_integral_kernel}

To get explicit statements for the main exponential term from Theorem \ref{T: Main Theorem}, one needs details on the local behavior of $\Phi_{\ell, 0} (t)$ around $t=0$. 
Both for this reason and also to provide more details on the kernel function in general, here we determine its Taylor coefficients at $t=0$. We compute these coefficients by analyzing the false-indefinite theta functions $u_j$ using the Euler--Maclaurin summation formula, which also ties the computation of the main exponential term in $\a_j(n)$ to studies such as \cite{BCN} or its variants applying Wright's Circle Method. We start with Theorem \ref{prop:uj_modular_mordell_representation} stating that, for $V>0$, $j \in \{0,1,2\}$, and integers $0 \leq h,h' < k$ satisfying $hh' \equiv -1 \pmod{k}$,
\begin{equation}\label{eq:uj_modular_mordell_rep_h_k2}
	u_j \!\lp  \frac{h}{k} + \frac{iV}{k^2} \rp
	=
	\frac{ik}{V} \sum_{\ell=0}^2 \Psi_{M_{h,k}} (j,\ell)
	  \mathcal{I}_{\ell, \frac{h'}{k}} \!\lp \frac{h'}{k} + \frac{i}{V} \rp .
\end{equation}
We first show that the asymptotic expansion of the right-hand side is determined by the Taylor coefficients of $\Phi_{\ell, \frac{h'}{k}} (t)$ at $t=0$.

\begin{lem}\label{lem:calI_asymptotic_expansion_taylor}
For $\ell \in \{0,1,2\}$, $N, k \in \IN$, and $h' \in \IZ$ with $\gcd (h',k) = 1$ we have
\begin{equation*}
\mathcal{I}_{\ell, \frac{h'}{k}} \!\lp \frac{h'}{k} + \frac{i}{V} \rp
=
\frac{1}{\pi i} \sum_{r=0}^{N-1} \lp \frac{V}{2\pi} \rp^{r+1}
  \Phi^{(r)}_{\ell, \frac{h'}{k}} (0)
+ O \!\lp V^{N+1} \rp
\quad \mbox{as } V \to 0^+  .
\end{equation*}
\end{lem}
\begin{proof}
Theorem \ref{prop:uj_modular_mordell_representation} gives that
\begin{equation*}
\mathcal{I}_{\ell, \frac{h'}{k}} \!\lp \frac{h'}{k} + \frac{i}{V} \rp
= - \frac{V}{\pi i} \   \psum_{n\in\Z+\b_\ell} d_\ell(n) e^{\frac{2 \pi i h' n}{k}}
\PV\int_0^\infty \frac{e^{-2\pi t}}{n-Vt} dt .
\end{equation*}
Noting that $\frac{1}{n-Vt} = \sum_{r=0}^{N-1} \frac{V^r   t^r}{n^{r+1}} + \frac{1}{n-tV} \frac{V^N   t^N}{n^N}$,
we decompose
\begin{multline*}
\mathcal{I}_{\ell, \frac{h'}{k}} \!\lp \frac{h'}{k} + \frac{i}{V} \rp
= - \frac{1}{\pi i} \sum_{r=0}^{N-1} V^{r+1}
\   \psum_{n\in\Z+\b_\ell} \frac{d_\ell(n) e^{\frac{2 \pi i h' n}{k}}}{n^{r+1}}
\int_0^\infty t^r e^{-2\pi t} dt
\\
- \frac{V^{N+1}}{\pi i}
\   \psum_{n\in\Z+\b_\ell} \frac{d_\ell(n) e^{\frac{2 \pi i h' n}{k}}}{n^{N}}
\PV\int_0^\infty  \frac{t^Ne^{-2\pi t}}{n-Vt} dt  .
\end{multline*}
The decomposition is justified since $n \neq 0$ and by the convergence of the involved terms as
\begin{equation*}
\int_0^\infty t^r e^{-2\pi t} dt = \frac{r!}{(2 \pi)^{r+1}}
\andd
\psum_{n\in\Z+\b_\ell} \frac{d_\ell(n) e^{\frac{2 \pi i h' n}{k}}}{n^{r+1}}
= - \frac{1}{r!} \Phi^{(r)}_{\ell, \frac{h'}{k}} (0) .
\end{equation*}
So the lemma statement follows if we can show
\begin{equation}\label{eq:calI_asymptotic_expansion_taylor_pv_bound}
\left| \PV\int_0^\infty \frac{t^Ne^{-2\pi t}}{n-Vt} dt \right| \ll \frac{1}{|n|},
\end{equation}
which thanks to the fact that $d_\ell (n) \ll \sqrt{|n|}$ (by Corollary \ref{cor:dj_bound}) implies
\begin{equation*}
\sum_{n\in\Z+\b_\ell} \frac{|d_\ell(n)|}{|n|^{N}}
\left| \PV\int_0^\infty \frac{t^Ne^{-2\pi t}}{n-Vt}   dt \right| \ll 1 .
\end{equation*}

To prove \eqref{eq:calI_asymptotic_expansion_taylor_pv_bound} for $n < 0$ (where the pole at $t = \frac{n}{V}$ is not on the integration path), we note that $\frac{1}{|n-Vt|} \leq \frac{1}{|n|}$ for $t \geq 0$. For $n > 0$, we start by writing
\begin{equation*}
\PV\int_0^\infty \frac{t^Ne^{-2\pi t}}{n-Vt} dt
=
\int_0^{\frac{n}{2V}} \frac{t^Ne^{-2\pi t}}{n-Vt} dt
+ \PV\int_{\frac{n}{2V}}^{\frac{3n}{2V}} \frac{t^Ne^{-2\pi t}}{n-Vt} dt
+ \int_{\frac{3n}{2V}}^\infty \frac{t^Ne^{-2\pi t}}{n-Vt} dt .
\end{equation*}
If $t \in [0,\frac{n}{2V}]$ or $t\ge\frac{3n}{2V}$, then we have $\frac{1}{|n-Vt|} \leq \frac{2}{n}$ so that
\begin{equation*}
\left| \int_0^{\frac{n}{2V}} \frac{t^Ne^{-2\pi t}}{n-Vt} dt \right|
+ \left| \int_{\frac{3n}{2V}}^\infty \frac{t^Ne^{-2\pi t}}{n-Vt} dt \right|
\leq
\frac{2}{n} \int_0^\infty t^N e^{-2\pi t} dt
\end{equation*}
and these terms obey the claimed upper bound. For the second integral, we make the change of variables as $t \mapsto \frac{2n}{V} - t$ for one half of this term to remove the pole at $t = \frac{n}{V}$ and write
\begin{equation*}
\PV\int_{\frac{n}{2V}}^{\frac{3n}{2V}} \frac{t^Ne^{-2\pi t}}{n-Vt} dt
=
\frac{1}{2V} \int_{\frac{n}{2V}}^{\frac{3n}{2V}}
\frac{t^N e^{-2\pi t} - \lp \frac{2n}{V} - t \rp^N e^{-2\pi \left(\frac{2n}{V} - t\right)}}{\frac{n}{V} - t} dt.
\end{equation*}
By the mean value theorem we have
\begin{equation*}
\frac{t^N e^{-2\pi t} - \lp \frac{2n}{V} - t \rp^N e^{-2\pi \left(\frac{2n}{V} - t\right)}}{2 \lp t- \frac{n}{V} \rp}
= \xi^{N-1} \lp  N- 2 \pi \xi \rp e^{-2 \pi \xi}
\end{equation*}
for some $\xi$ between $t$ and $\frac{2n}{V}-t$. Note that for $t \in [\frac{n}{2V},\frac{3n}{2V}]$, we have 
\begin{equation*}
\xi \geq \frac{n}{2V} \geq \frac{t}{3} \mbox{ and }
\xi \leq \frac{3n}{2V} \leq 3t
\end{equation*}
so that
\begin{equation*}
\left| \xi^N \lp  N- 2 \pi \xi \rp e^{-2 \pi \xi} \right|
\leq
\xi^N \lp  N+ 2 \pi \xi \rp e^{-2 \pi \xi}
\leq
(3t)^N \lp N + 6 \pi t \rp e^{-\frac{2\pi t}3}
\ \ \mbox{ and }\ \
\frac{1}{V \xi} \leq \frac{2}{n} .
\end{equation*}
Therefore, we have
\begin{equation*}
\left| \PV\int_{\frac{n}{2V}}^{\frac{3n}{2V}} \frac{t^Ne^{-2\pi t}}{n-Vt} dt \right|
\leq
\frac{2}{n} \int_0^\infty (3t)^N \lp N + 6 \pi t \rp e^{-\frac{2\pi t}3}  d t
\end{equation*}
and this term also obeys the claimed bound in \eqref{eq:calI_asymptotic_expansion_taylor_pv_bound}.
\end{proof}

Inserting Lemma \ref{lem:calI_asymptotic_expansion_taylor} in equation \eqref{eq:uj_modular_mordell_rep_h_k2}, we obtain
\begin{equation}\label{eq:uj_asymptotic_expansion_taylor}
u_j \!\lp  \frac{h}{k} + \frac{iV}{k^2} \rp
=
\frac{k}{2\pi^2}
\sum_{r=0}^{N-1} \lp \frac{V}{2\pi} \rp^r
\sum_{\ell=0}^2 \Psi_{M_{h,k}} (j,\ell)   \Phi^{(r)}_{\ell, \frac{h'}{k}} (0)
+ O \!\lp V^{N} \rp
\quad \mbox{as } V \to 0^+ .
\end{equation}
We next provide an independent computation from the left-hand side of equation \eqref{eq:uj_modular_mordell_rep_h_k2}.

\begin{lem}\label{lem:uj_asymptotic_expansion_taylor}
For $j \in \{0,1,2\}$, $N, k \in \IN$, and $h \in \IZ$ with $\gcd (h,k) = 1$ we have, as $V \to 0^+$,
\begin{multline*}
u_j \!\lp  \frac{h}{k} + \frac{iV}{k^2} \rp
=
\sum_{r=0}^{N-1} (8 \pi V)^r
\lp \sum_{\bm{\mu} \in \mathcal{S}^+_j}  -  \sum_{\bm{\mu} \in \mathcal{S}^-_j}  \rp
\sum_{\a = 0}^3   \sum_{0 \leq r_1,r_2 < k}
e^{ \frac{2 \pi ih}{k} \lp 12 (r_1+\mu_1)^2 - 2 (r_2+\mu_2)^2 \rp}
\\
\times \left(
- \frac{\wt{B}_{2r+2} \!\lp \frac{2(\mu_1+r_1)-(\mu_2+r_2)-k\a}{4k}\rp +
\wt{B}_{2r+2} \!\lp \frac{2(\mu_1+r_1)+\mu_2+r_2+k\a}{4k}\rp   }{(2r+2)!}
\int_0^\infty f^{(2r+1,0)}(0,x) dx  \right.
\\
\left.
+\!\! \sum_{n_1+n_2 = 2r}  \!\!
\frac{\wt{B}_{n_1+1} \!\lp \frac{2(\mu_1+r_1)-(\mu_2+r_2)-k\a}{4k}\rp
\wt{B}_{n_2+1} \!\lp \frac{2(\mu_1+r_1)+\mu_2+r_2+k\a}{4k}\rp   }{(n_1+1)!   (n_2+1)!}
f^{(n_1,n_2)}(\bm{0})
\right)  + O\left(V^N\right),
\end{multline*}
where $\mathcal{S}^\pm_j$ are defined in Section \ref{sec:parity_partition_false_indefinite} and
$f(\bm{x}) := e^{-x_1^2-10x_1x_2-x_2^2}$.
\end{lem}
\begin{proof}
The result follows from the two-dimensional Euler--Maclaurin formula (see e.g.~\cite{BJM}).
\end{proof}

Comparing Lemma \ref{lem:uj_asymptotic_expansion_taylor} to equation \eqref{eq:uj_asymptotic_expansion_taylor} we find the Taylor coefficients of $\Phi_{\ell, \frac{h'}{k}} (t)$ at $t=0$.
\nolisttopbreak
\begin{prop}\label{prop:integral_kernel_taylor_coefficients}
For $j \in \{0,1,2\}$, integers $0 \leq h,h' < k$ satisfying $hh' \equiv -1 \pmod{k}$, and $r \in \mathbb{N}_0$ we have
\begin{multline*}
\sum_{\ell=0}^2 \Psi_{M_{h,k}} (j,\ell)   \Phi^{(r)}_{\ell, \frac{h'}{k}} (0)
=
\frac{(4\pi)^{2r+2}}{8k}
\lp \sum_{\bm{\mu} \in \mathcal{S}^+_j}  -  \sum_{\bm{\mu} \in \mathcal{S}^-_j}  \rp
\sum_{\a = 0}^3   \sum_{0 \leq r_1,r_2 < k}
e^{ \frac{2 \pi ih}{k} \lp 12 (r_1+\mu_1)^2 - 2 (r_2+\mu_2)^2 \rp}
\\
\times \left(
- \frac{\wt{B}_{2r+2} \!\lp \frac{2(\mu_1+r_1)-(\mu_2+r_2)-k\a}{4k}\rp +
\wt{B}_{2r+2} \!\lp \frac{2(\mu_1+r_1)+\mu_2+r_2+k\a}{4k}\rp   }{(2r+2)!}
\int_0^\infty f^{(2r+1,0)}(0,x) dx  \right.
\\
\left.
+\!\! \sum_{n_1+n_2 = 2r}  \!\!
\frac{\wt{B}_{n_1+1} \!\lp \frac{2(\mu_1+r_1)-(\mu_2+r_2)-k\a}{4k}\rp
\wt{B}_{n_2+1} \!\lp \frac{2(\mu_1+r_1)+\mu_2+r_2+k\a}{4k}\rp   }{(n_1+1)!   (n_2+1)!}
f^{(n_1,n_2)}(\bm{0})
\right) .
\end{multline*}
\end{prop}

For the main exponential term in the expansion of $\a_j (n)$ in Theorem \ref{T: Main Theorem}, we need these values for $\Phi_{\ell, 0} (t)$, which can be obtained by specializing Proposition \ref{prop:integral_kernel_taylor_coefficients} to $k=1$ (with $h=h'=0$ and $M_{h,k} = S$).
For the convenience of the reader we note the first few Taylor coefficients in Table \ref{tab:kernel_taylor_coefficients}.

\begin{table}[H]
\centering
\begin{tabular}{| c || c | c | c |}
\toprule
& $j=0$ & $j=1$ & $j=2$
\\
\hline
$r=0$ & $0$ & $4 \pi^2$ & $4 \pi^2$
\\
\hline
$r=1$ & $16 \pi^4$ & $\frac{23}{3} \pi^4$ & $\frac{50}{3} \pi^4$
\\
\hline
$r=2$ & $\frac{284}{3} \pi^6$ & $\frac{9745}{72} \pi^6$ & $\frac{2929}{18} \pi^6$
\\
\hline
$r=3$ & $\frac{32881}{18} \pi^8$ & $\frac{3965831}{2592} \pi^8$ & $\frac{769033}{324} \pi^8$
\\
\hline
$r=4$ &
$\frac{20222423}{648} \pi^{10}$ & $\frac{4241759521}{124416} \pi^{10}$ &
$\frac{359054305}{7776} \pi^{10}$
\\
\bottomrule
\end{tabular}
\vspace{-8pt}
\caption{\rule{0cm}{1cm}Values of $\displaystyle\sum_{\ell=0}^2 \Psi_{S} (j,\ell)   \Phi^{(r)}_{\ell, 0} (0)$ for the first few $r$.}
\label{tab:kernel_taylor_coefficients}
\vspace{-8pt}
\end{table}

\end{document}